\documentclass[12pt,oneside,reqno,fleqn]{amsart}
\usepackage{bbm}
\usepackage{cases}
\usepackage{graphicx}
\usepackage{mathrsfs}
\usepackage{stmaryrd}
\usepackage{color}
\usepackage[dvipsnames]{xcolor}
\usepackage{amsfonts}
\newif\ifdraft\draftfalse
\ifdraft\usepackage[notcite,notref]{showkeys}\fi
\usepackage{amsmath,amssymb,amsthm,mathtools,bm}

\usepackage{libertine}
\usepackage[T1]{fontenc}
\usepackage[libertine,varbb]{newtxmath}
\usepackage[colorlinks,linkcolor=Red,citecolor=Red,urlcolor=Red]{hyperref}
\usepackage[shortlabels] {enumitem}
\usepackage[nameinlink,capitalize]{cleveref}
\usepackage[spacing=true,kerning=true,tracking=true,letterspace=50]{microtype}

\makeatletter
\renewcommand{\eqref}[1]{\textup{\tagform@{\ref{#1}}}}
\makeatother

\usepackage{comment}
\hypersetup{pdftitle={Well-posedness of Rough Differential Equations at the Krylov-Röckner Integrability Threshold}, pdfauthor={Zimo Hao, Khoa L\^e}}
\numberwithin{equation}{section}
\ifdraft\usepackage{todonotes}\else\usepackage[disable]{todonotes}\fi

\newtheorem{theorem}{Theorem}[section]
\newtheorem{lemma}[theorem]{Lemma}
\newtheorem{proposition}[theorem]{Proposition}
\newtheorem{corollary}[theorem]{Corollary}
\theoremstyle{remark}
\newtheorem{assumption}[theorem]{Assumption}

\newtheorem{remark}[theorem]{Remark}
\newtheorem{definition}[theorem]{Definition}

\crefname{equation}{Equation}{Equations}
\crefname{assumption}{Assumption}{Assumptions}
\crefname{innercustomthm}{Condition}{Conditions}
\crefrangelabelformat{innercustomthm}{#3#1#4-#5#2#6}

\newcommand{\R}{{\mathbb R}}
\newcommand{\E}{{\mathbb E}}
\newcommand{\Rd}{{\R^d}}
\newcommand{\tand}{\quad\text{and}\quad}

\newcommand{\1}{{\mathbf 1}}

\newcommand{\norm}[1]{{\left\vert\kern-0.25ex\left\vert\kern-0.25ex\left\vert #1
    \right\vert\kern-0.25ex\right\vert\kern-0.25ex\right\vert}}

\renewcommand{\le}{\leq}
\renewcommand{\ge}{\geq}

\allowdisplaybreaks

\newcommand{\Hsigma}{\hyperlink{assumption:Hsigma}{\textbf{(H$_\sigma$)}}}
\newcommand{\Hb}{\hyperlink{assumption:Hb}{\textbf{(H$_b$)}}}
\newcommand{\SsigmaK}{\hyperlink{assumption:SsigmaK}{\textbf{(S$^k_\sigma$)}}}
\newcommand{\Ssigmao}{\hyperlink{assumption:SsigmaK}{\textbf{(S$^0_\sigma$)}}}
\newcommand{\Ssigmai}{\hyperlink{assumption:SsigmaK}{\textbf{(S$^1_\sigma$)}}}
\newcommand{\Aa}{\hyperlink{assumption:Aa}{\textbf{(A$_a$)}}}
\newcommand{\Asigma}{\hyperlink{assumption:Asigma}{\textbf{(A$_\sigma$)}}}

\def\dif{{\mathord{{\rm d}}}}

\def\nor{|\mspace{-3mu}|\mspace{-3mu}|}
\def\eps{{\epsilon}}

\def\cA{{\mathcal A}}
\def\cB{{\mathcal B}}

\def\cD{{\mathcal D}}

\def\cJ{{\mathcal J}}

\def\cL{{\mathcal L}}
\def\cM{{\mathcal M}}

\def\cP{{\mathcal P}}

\def\cV{{\mathcal V}}

\def\mE{{\mathbb E}}

\def\mI{{\mathbb I}}

\def\mL{{\mathbb L}}

\def\mN{{\mathbb N}}

\def\mP{{\mathbb P}}
\def\mQ{{\mathbb Q}}
\def\mR{{\mathbb R}}

\def\mW{{\mathbb W}}
\def\mX{{\mathbb X}}

\def\1{{\mathbf{1}}}

\def\sF{{\mathscr F}}

\def\E{\mathbb E}

\def\sF{{\mathcal F}}
\def\mR{{\mathbb R}}
\def\mP{{\mathbb P}}
\def\mX{{\mathbb X}}
\newcommand{\mxx}{\mX}
\def\bX{{\boldsymbol X}}

\def\mW{{\mathbb W}}
\def\bW{{\boldsymbol W}}

\def\bbb1{\boldsymbol{1}}

\allowdisplaybreaks

\begin{document}
\title{Well-posedness of Rough Differential Equations at the Krylov-Röckner Integrability Threshold}

	\date{\today}

\author{Zimo Hao}
\address{School of Mathematics and Statistics, Beijing Institute of Technology, Beijing 100081, China}
\email{zmhao@bit.edu.cn}

\author{Khoa L\^e}
\address{University of Leeds, School of Mathematics, Leeds LS2 9JT, U.K.}
\email{k.le@leeds.ac.uk}

\thanks{
Z. Hao is supported by the National Natural Science Foundation of
China, Grant No.~12601259. K. L\^e is supported by the Engineering
and Physical Sciences Research Council (EPSRC), grant No.~EP/Y016955/1}

	\begin{abstract}
  We study rough differential equations (RDEs) driven by Brownian (It\^o) rough paths with uniformly locally integrable drifts. The diffusion coefficient is assumed to be uniformly elliptic, of class $C^{2+}$ in space and of bounded variation in time, while the drift is $q$-integrable in time and $p$-integrable in space, uniformly on balls of a fixed radius, for some $p,q$ satisfying the Krylov--R\"ockner condition $\frac dp+\frac2q<1$. We show that, for almost every realization of the It\^o rough path lift of the Brownian motion, such an RDE has a unique solution for every initial value. In other words, for almost every realization of the It\^o lift, the corresponding rough differential equation has exactly one solution in the class of Davie-type solutions, for every initial value; this is the appropriate notion of path-by-path uniqueness for the corresponding stochastic differential equation (SDE) with multiplicative noise.  
  The proof follows the strategy of Shaposhnikov (2016). The principal new ingredients are the stability of RDEs with bounded variation perturbation and Davie-type estimates for the flow of the driftless equation with non-constant diffusion coefficients, which we derive from known maximal regularity estimates for the corresponding parabolic partial differential equations (PDEs).


	\end{abstract}

	\maketitle

\section{Introduction} 
\label{sec:introduction}
We consider the stochastic differential equation
\begin{align}\label{eqn.SDE}
    \dif X_t=b(t,X_t)\dif t+\sigma(t,X_t)\dif W_t,
\end{align}
where $\sigma:\mR_+\times\mR^d\to\mR^d\otimes\mR^d$ and $b:\mR_+\times\mR^d\to\mR^d$ are measurable functions and $W$ is a $d$-dimensional Brownian motion.

In the framework of It\^o's calculus, SDEs have been studied extensively.
Recall that a weak solution to \eqref{eqn.SDE} consists of a filtered probability space $(\Omega,\sF,(\sF_t)_{t\ge0},\mP)$, an $(\sF_t)$-Brownian motion $W$ and a continuous $(\sF_t)$-adapted process $X$ such that, almost surely,
\begin{align*}
&\int_0^t|b(s,X_s)|\dif s+\int_0^t |\sigma(s,X_s)|^2\dif s<\infty\quad\text{and}\\
&X_t=X_0+\int_0^tb(s,X_s)\dif s+\int_0^t \sigma(s,X_s)\dif W_s\quad \text{for all $t\ge0$},
\end{align*}
where $\int_0^t \sigma(s,X_s)\dif W_s$ is the It\^o integral. A weak solution is called a strong solution if $X$ is adapted to the filtration generated by $X_0$ and $W$. Uniqueness in law holds for \eqref{eqn.SDE} if any two weak solutions with the same initial law have the same law on the space of continuous paths, and pathwise uniqueness holds if any two weak solutions $X$ and $Y$ defined on the same filtered probability space with the same Brownian motion and $X_0=Y_0$ satisfy
\begin{align*}
\mP( X_t=Y_t \ \text{ for all $t\ge0$})=1.
\end{align*}
When $\sigma$ and $b$ are Lipschitz continuous, strong existence and pathwise uniqueness are classical; see, for instance, \cite{RY99}. For $T>0$ and $p,q\in[1,\infty]$, let $\mL^p_q(T):=L^q([0,T];L^p(\Rd))$ and define
$$
\mathcal{J}:=\{(p,q)\in(2,\infty)^2~|~d/p+2/q<1\}.
$$
When $\sigma=\mI_{d\times d}$ and $b\in \mL^p_q(T)$ for some $(p,q)\in\mathcal{J}$, Krylov and R\"ockner \cite{Kr-Ro} established strong existence and pathwise uniqueness. These results were subsequently extended by Zhang and his collaborators \cite{Zha05, Zhang2011, Zhang2016, XXZZ} to the case where $\sigma$ is bounded, uniformly elliptic, and satisfies $\nabla\sigma\in \mL^p_q(T)$. A central ingredient in these works is Zvonkin's transformation \cite{Zv74} and its regularity. We refer the reader to \cite{HZ23} and the references therein for further results on weak and strong well-posedness, uniqueness in law, and pathwise uniqueness for \eqref{eqn.SDE}, including regimes beyond the Krylov--R\"ockner threshold $\mathcal{J}$.

Path-by-path uniqueness is a stronger notion than pathwise uniqueness. In the additive noise case $\sigma=\mI_{d\times d}$, it requires uniqueness of solutions for almost every realization of the driving Brownian path.
\begin{definition}[Path-by-path uniqueness, additive noise]
We say that path-by-path uniqueness holds for \eqref{eqn.SDE} with $\sigma=\mI_{d\times d}$ if, for almost every $\omega\in\Omega$ and every initial condition $x\in\Rd$, the integral equation
\begin{align*}
X_t=x+\int_0^t b(r,X_r)\dif r+W_t(\omega)
\end{align*}
has at most one continuous solution.
\end{definition}
Pathwise uniqueness asserts that for any two solutions $X$ and $Y$ of \eqref{eqn.SDE}, there exists an event $\Omega_{X,Y}$ of full probability, possibly depending on the pair $(X,Y)$, such that $ X(\omega)=Y(\omega)$ for all $\omega\in\Omega_{X,Y}$. In contrast, path-by-path uniqueness requires the existence of a single event $\Omega_0$ of full probability such that, for every $\omega\in\Omega_0$ and every initial condition $x\in\Rd$, the integral equation driven by the fixed sample path $W(\omega)$ admits at most one continuous solution. In particular, the competing solutions are allowed to be arbitrary continuous paths and need not be adapted to any filtration or arise as solutions of the SDE. We refer to \cite{ShaWre,Anzeletti} for examples showing that pathwise uniqueness does not, in general, imply path-by-path uniqueness. The first path-by-path uniqueness results were obtained in \cite{Davie} for equations with bounded drifts and additive noise. Davie's results were later simplified and reproduced with a different proof in \cite{AVS}. Recently, \cite{ALL23} established path-by-path uniqueness under the Krylov--R\"ockner condition $b\in \mL^p_q(T)$ for some $(p,q)\in\mathcal J$, building upon  an argument from \cite{AVS}. 

For multiplicative noise, the stochastic integral in \eqref{eqn.SDE} is understood in the It\^o sense, so a pathwise formulation of the equation requires a pathwise interpretation of this integral. Rough path theory, introduced by Lyons \cite{Ly98}, provides such an interpretation. More precisely, for almost every $\omega\in\Omega$, the Brownian path $W(\omega)$ admits a  lift to a rough path $\bW(\omega)$, namely its It\^o lift, and \eqref{eqn.SDE} can then be interpreted as a rough differential equation (RDE) driven by $\bW(\omega)$, see  \cref{def:413}. 
Davie \cite{Dav11} claimed path-by-path uniqueness for bounded Borel drifts and non-degenerate diffusion coefficients $b$ with $b,\nabla b$ locally H\"older, with solutions understood in a rough path sense (remainder $O((t-s)^{1+\delta})$ in the Davie expansion). As in the additive case \cite{Davie}, one compares a candidate solution with the strong solution, but now the difference of two solutions is a rough integral of Brownian size rather than a Lipschitz path. The shift estimates of \cite{Davie} must therefore be replaced by estimates for two diffusions started at different points. In \cite{Dav11} these are obtained from transition-density bounds and a $T(1)$ theorem, and the key step is only outlined; we do not use that argument. Our proof instead restarts the comparison at each time from the current point of the candidate solution, using the flow of the driftless RDE, for which we prove almost sure Davie-type estimates from maximal regularity for the Kolmogorov equation (\cref{sub:asreg}). This covers drifts in the localized Krylov--R\"ockner class, including bounded drifts, and coefficients of bounded variation in time, at the price of $C^{2+\varepsilon}$ regularity in space and a $p$-variation notion of solution.
Rough paths and RDEs with time-dependent coefficients are briefly reviewed in \cref{sec:2}, while the notion of a path-by-path solution is introduced in \cref{def:413}. We refer to the monograph \cite{FH} and the references therein for classical results on rough differential equations, and to \cite{FZ18} for the extension of rough path theory to jump processes.

To accommodate bounded drifts and, more generally, drifts that are uniformly locally integrable, we work with localized $L^p$ spaces. Let $\chi\in C^\infty_c(\Rd;[0,1])$ be such that $\chi(x)=1$ for $|x|\le 1$ and $\chi(x)=0$ for $|x|>2$. For $r>0$ and $x,z\in\Rd$, define \[ \chi_r(x):=\chi(x/r), \qquad \chi_r^z(x):=\chi_r(x-z). \]
\begin{definition}
    Let $p,q\in[1,\infty]$, and fix $r,T>0$. We define
     \begin{align*}
        \widetilde{L}^p:=\left\{f\in L^1_{loc}(\mR^d) ~\Big|~
            \nor f\nor_{\widetilde{L}^{p}}:=\sup_{z\in\mR^d}\|f\chi^z_r\|_{L^p}<\infty
             \right\}
    \end{align*}
    and
    \begin{align*}
        \widetilde{\mL}^p_q(T):=\left\{f\in L^1_{loc}([0,T]\times\mR^{d}) ~\Big|~
            \nor f\nor_{\widetilde{\mL}^p_q(T)}:=\sup_{z\in\mR^d}\|f\chi^z_r\|_{\mL^p_q(T)}<\infty
             \right\}.
    \end{align*}
\end{definition}
\begin{remark}\label{in:rk13}
    (i) By a standard finite-covering argument, the spaces $\widetilde{L}^p$ and $\widetilde{\mL}^p_q(T)$ are independent of the choice of $r>0$, and the corresponding norms are equivalent; see \cite[Proposition 4.1]{ZZ21}.\\
    (ii) For any measurable function $f$, $\nor f\nor_{ \widetilde{\mL}^p_q(T)}\le \|f\|_{L^q([0,T];\widetilde{L}^p)}$, and the inequality may be strict.\\
    (iii) For any $1\le p_1\le p_2\le \infty$ and $1\le q_1\le q_2\le\infty$,
    \begin{align*}
        L^{p_2}\subset\widetilde{L}^{p_2}\subset\widetilde{L}^{p_1}\quad\text{and}\quad \mL^{p_2}_{q_2}(T)\subset\widetilde{\mL}^{p_2}_{q_2}(T)\subset\widetilde{\mL}^{p_1}_{q_1}(T).
    \end{align*}
    In particular, every bounded measurable function belongs to $\widetilde{\mL}^{p}_{q}(T)$ for all $p,q\in[1,\infty]$.
\end{remark}
Throughout this paper, we work with the following standing assumptions.
\begin{assumption}\label{astn}
\begin{itemize}
    \item[{\Hsigma}]\hypertarget{assumption:Hsigma}{}
    The diffusion coefficient $\sigma$ is continuous on $\mR_+\times\Rd$, and $\sigma(t,\cdot)$ is twice continuously differentiable for every $t\ge0$. Moreover, there exist constants $\kappa_0>1$ and $\eps>0$ such that, for all $t\ge0$ and $x,y,\xi\in\mR^{d}$,
    \begin{align*}
        \kappa_0^{-1}|\xi|\le |\sigma(t,x)\xi|\le \kappa_0|\xi|\quad\text{and} \quad |\nabla^2\sigma(t,x)-\nabla^2\sigma(t,y)|\le\kappa_0|x-y|^\eps.
    \end{align*}
   In addition, for every $T>0$,
    \begin{align}\label{in:con}
\|\sigma\|_{C^{1-var}L^\infty}:=\sup_{\cP}\sum_{[s,t]\in\cP}\|\sigma(t,\cdot)-\sigma(s,\cdot)\|_{L^\infty }<\infty,
\end{align}
where the supremum is taken over all partitions $\cP$ of $[0,T]$.
\item[{\Hb}]\hypertarget{assumption:Hb}{} The drift coefficient $b$ is Borel measurable and belongs locally in time to $\widetilde{\mL}^p_q$, in the sense that $\|b\|_{\widetilde{\mL}^p_q(T)}<\infty$ for every $T>0$ for some $(p,q)\in\cJ$.
\end{itemize}
\end{assumption}
For example, \eqref{in:con} holds whenever $t\mapsto \sigma(t,\cdot)\in L^\infty(\Rd)$ is Lipschitz continuous. The spatial regularity in \Hsigma{} and the time regularity condition \eqref{in:con} are primarily required for the analysis of rough differential equations with time-dependent coefficients; see \cref{sec:2,sec:4.5}. In contrast, strong well-posedness of \eqref{eqn.SDE} is available under considerably weaker assumptions on $\sigma$, for example when $\nabla\sigma\in\mL^p_q(T)$; see \cite{XXZZ}. Our main result is as follows.
\begin{theorem}\label{thm:main} Assume \Hsigma{} and \Hb{}. Let $W$ be a $d$-dimensional Brownian motion with It\^o rough path lift $\bW$, and let $T>0$. Then there exists an event $\Omega_{\sigma,b,T}$ of full probability such that, for every $\omega\in\Omega_{\sigma,b,T}$, the rough path $\bW(\omega)$ determines the solution uniquely: for every $x\in\Rd$, the RDE \[ X_t = x+\int_0^t b(r,X_r)\,\dif r +\int_0^t \sigma(r,X_r)\,\dif\bW_r(\omega), \qquad t\in[0,T], \] has exactly one path-by-path solution in the sense of \cref{def:413}. Furthermore, for each $x\in\Rd$, this solution agrees almost surely with the strong solution of \eqref{eqn.SDE} with initial condition $x$.
\end{theorem}
We stress that the existing theory of rough differential equations is designed for equations with regular coefficients, typically measured in Lipschitz--H\"older scales, and generally does not capture the regularization induced by non-degenerate diffusion coefficients. Consequently, it cannot be directly applied to prove \cref{thm:main}. Conversely, from the viewpoint of stochastic differential equations, \cref{thm:main} establishes path-by-path uniqueness, the strongest natural notion of uniqueness, for multiplicative noise and drifts at the Krylov--R\"ockner threshold.
\cref{thm:main} follows from \cref{thm:61,thm:64} below. In particular, path-by-path uniqueness holds for \eqref{eqn.SDE}: by \cref{thm:0311}, every weak solution starting from $x$, defined on the same probability space with the same Brownian motion and satisfying the integrability condition in \cref{thm:0311}~(ii), coincides almost surely with this path-by-path solution.
\begin{remark} The event $\Omega_{\sigma,b,T}$ in \cref{thm:main} depends on the specific Borel representative of the drift $b$, rather than only on its equivalence class in $L^0(\R_+\times\Rd)$ modulo Lebesgue-null sets. This contrasts with the notions of weak and strong solutions to \eqref{eqn.SDE}, which are invariant under modifications of $b$ on sets of Lebesgue measure zero, thanks to the Krylov estimate. Path-by-path uniqueness is more delicate, since it concerns all continuous paths, and such paths may spend positive time in subsets of space-time having zero Lebesgue measure. It therefore appears to be an open question whether the full-probability event $\Omega_{\sigma,b,T}$ can be chosen independently of the representative of $b$. \end{remark}

\subsection*{Methods} We follow the strategy of \cite{AVS} and \cite{ALL23}: a solution $Y$ is compared on short intervals with a reference solution started from $Y_s$, and an estimate of order $w(s,t)^\beta$ with $\beta>1$ and some control function $w$ yields uniqueness by a telescoping argument. In the additive case, the reference object is the shifted Brownian path and the required Davie-type estimates follow from explicit heat kernel bounds. For multiplicative noise, the natural reference object is the flow $\psi^{s,x}$ of the driftless RDE (\cite{DHL26,DGLL24}, also called regularizing flow in \cite{DHaL26}), which depends on the starting time. We therefore need almost sure Davie-type estimates for $\psi$ that are uniform in the starting time and the starting point. We obtain the corresponding moment estimates from maximal regularity estimates for the Kolmogorov equation with coefficients $\frac12\sigma\sigma^*$ in localized spaces (\cref{sub.DavieEst}), and we turn them into almost sure estimates by embeddings of anisotropic fractional Sobolev spaces (\cref{sub:asreg}). Moreover, the rough integral along an arbitrary solution has to be controlled as well, which requires the almost sure estimate for rough integrals in \cref{thm.DavieEstDiff} and the theory of RDEs with time-dependent coefficients in \cref{sec:2}.

\subsection*{Structure of the paper}
In \cref{sec:2}, we review rough paths and controlled rough paths, and state well-posedness and stability results for RDEs with time-dependent coefficients; their proofs are given in \cref{S:App}. In \cref{sec.SDERDE}, we collect results on the SDE \eqref{eqn.SDE} (\cref{sec:31}), introduce the flow of the driftless RDE (\cref{sec:4.5}), and relate weak solutions of SDEs to path-by-path solutions of RDEs (\cref{sec:33}).

\cref{sec:4} contains the regularization estimates. \cref{sec:24} recalls maximal regularity estimates for parabolic PDEs in localized spaces. These are used in \cref{sub.DavieEst} to obtain Davie-type moment estimates for the stochastic flow, which rely on the Krylov estimates of \cref{sec:31}. In \cref{sub:asreg}, these moment estimates are turned into almost sure estimates for the driftless flow, using a Garsia--Rodemich--Rumsey type inequality and a Kolmogorov continuity theorem for anisotropic fractional Sobolev spaces, which are proved in \cref{sec:D}.

The proof of our main result is given in \cref{sec:5}. Existence relies on the almost sure estimates of \cref{sub:asreg} and the stability estimates of \cref{sec:2}, and uniqueness follows the strategy of \cite{ALL23}. Finally, \cref{App:C} contains an interpolation inequality used in \cref{sec:4.5}.

\subsection*{Notation}
Throughout this paper, we use the following notation.
\begin{itemize}
    \item $|\cdot|$ denotes the Euclidean norm on $\mR^d$, $d\in\mathbb{N}$, and $\cL(\Rd,E)$ denotes the space of linear maps from $\Rd$ to $E$.
\item For $j\in\mathbb{N}\cup\{0\}$, we use $\nabla^j$ to denote
the $j$th order derivative in the spatial variable.
\item For $\alpha>0$, $[\alpha]$ denotes the largest integer which is strictly smaller than $\alpha$. We write $C^0$ for the space of bounded continuous functions with the supremum norm. For an integer $k\ge1$ we use $\|f\|_{C^k}:=\sum_{j\le k}\|\nabla^jf\|_\infty$ (as in \cref{App:C}); for non-integer $\alpha$ the norm is defined below.
\item Let $U\subset\mathbb{R}^d$ be open, let $E$ be a Banach space with norm $\|\cdot\|_E$, and let $\alpha \geq 0$. The space $C^\alpha(U;E)$ consists of all functions $f: U \to E$ with finite norm 
\begin{align*}
\|f\|_{C^\alpha(U;E)}:=\sup_{x\in U}\sum_{k=0}^{[\alpha]}{\|\nabla^kf(x)\|_E}+[f]_{C^\alpha(U;E)},
\end{align*}
where
\begin{align*}
[f]_{C^\alpha(U;E)}:=\sup_{x,y\in U; x\ne y}\frac{\|\nabla^{[\alpha]}f(x)-\nabla^{[\alpha]}f(y)\|_E}{|x-y|^{\alpha-[\alpha]}}.
\end{align*}
When $E$ is finite dimensional, we write $C^\alpha(U):=C^\alpha(U;E)$ and $C^\alpha:=C^\alpha(\mR^d;E)$.
\item We use $C$, with or without subscripts, to denote a constant that may change from line to line.
We write $A\lesssim B$ if there is a constant $C\ge1$ such that $A\le CB$, and $:=$ means equal by definition. We use the Einstein summation convention.
\item In \cref{sec:2}, \cref{S:App} and the statements on RDEs in \cref{sec.SDERDE}, $p\in[2,3)$ denotes the variation exponent of a rough path and $q,r$ denote variation exponents of controlled rough paths. In connection with drifts and localized spaces, $(p,q)\in\cJ$ denote integrability exponents. From \cref{sub:asreg} on, the rough path exponents are fixed and denoted by $p_*,q_*,r_*$.
\item The assumptions used in this paper are \Hsigma{} and \Hb{} (\cref{astn}), \SsigmaK{} (\cref{asm.sigma}), \Asigma{} (\cref{asm.As}), and \Aa{} (\cref{asm.a}).
\end{itemize}

\subsection*{AI disclosure}
The language model Claude Sonnet 5.5 was used to audit the manuscript and to suggest improvements; which were reviewed  and implemented by the authors into the final version. 
\section{Primer on rough differential equations}\label{sec:2}

In most standard references (\cite{FH,Gub04,MR2387018}), rough differential equations are considered with time-independent vector fields. A dependence on time can be absorbed by enlarging the system, which works well when the dependence on time is sufficiently regular.
On the other hand, stochastic differential equations (such as \eqref{eqn.SDE}) are classically considered with coefficients that depend on both time and space, and the dependence on time is often so mild that the enlargement procedure breaks down. Hence, to treat \eqref{eqn.SDE} from the perspective of rough path theory in a way that is consistent (as much as possible) with the classical stochastic setting, we need to extend standard results for rough differential equations to time-dependent vector fields. Since these extensions are routine for experts, we only state them in this section; the proofs are collected in \cref{S:App}.

\subsection{Rough paths and controlled rough paths}\label{sub.roughpath}

We recall some standard concepts from Lyons' rough path theory.
We fix a finite time horizon $T>0$.
For any $0\le s_0<s_1$, we define the simplex
$$
\Delta([s_0,s_1]):=\{(s,t)\in\mR_+^2~|~s_0\le s\le t\le s_1\}.
$$
For any $p\in(0,\infty)$, $0\le S<T$ and Banach space $E$, we say that a continuous function $\Xi:\Delta([S,T])\to E$ has {\it finite $p$-variation}, if the quantity
\begin{align*}
    \|\Xi\|_{C^{p-var}([S,T];E)}:=\left(\sup_{\cP}\sum_{[s,t]\in \cP}\|\Xi_{s,t}\|_{E}^p\right)^{1/p}
\end{align*}
is finite,
where the supremum is taken over all finite partitions $\cP$ of $[S,T]$. When $p=\infty$, we denote
$$
\|\Xi\|_{C^{\infty-var}([S,T];E)}:=\sup_{S\le s<t\le T}\|\Xi_{s,t}\|_{E}.
$$
A continuous function $\xi:[0,T]\to E$ has finite \(p\)-variation if its increment map \((s,t)\mapsto \xi_t-\xi_s\) does. In such case, we write \(\xi\in C^{p-var}\) and denote
\begin{align*}
     \|\xi\|_{C^{p-var}([S,T];E)}:=\|\Xi\|_{C^{p-var}([S,T];E)}, \quad \text{with}\quad \Xi_{s,t}=\xi_t-\xi_s.
\end{align*}
When $E$ is finite dimensional, we use the abbreviated notations
\begin{align*}
    \|\Xi\|_{p,[S,T]}:=\|\Xi\|_{C^{p-var}([S,T];E)}
    \tand
    \|\xi\|_{p,[S,T]}:=\|\xi\|_{C^{p-var}([S,T];E)},
\end{align*}
and we write $\|\xi\|_{\infty}:=\sup_{t\in[0,T]}|\xi_t|$.

A {\it control} is a continuous function $w:\Delta([0,T])\to \mR_+$ with $w(t,t)=0$ which is super-additive, i.e. $w(s,u)+w(u,t)\le w(s,t)$ whenever $s\le u\le t$. It is well-known that for any $\xi\in C^{p-var}$, the map  $(s,t)\mapsto  w_{\xi,p}(s,t):=\|\xi\|_{p,[s,t]}^p$ is a control.

\begin{definition}
    Given $p\in[2,3)$,  a {\it $p$-rough path} (on $\mR^d$) is a continuous map $\bX=(X_{s,t},\mX_{s,t}):\Delta([0,T])\to\mR^d\times (\mR^d\otimes \mR^d)$ with $X_{s,t}:=X_t-X_s$ and \(X:[0,T]\to \Rd\), \(X_0=0\) such that
    \begin{enumerate}[(i)]
        \item $\mX$ satisfies Chen's relation \begin{align*} \mX_{s,
        t}=\mX_{s,u}+\mX_{u,t}+X_{s,u}\otimes X_{u,t},\quad \forall 0\le s\le
        u\le t\le T, \end{align*}
        \item \(X\) has finite \(p\)-variation and \(\mX\) has finite \(p/2\)-variation.
    \end{enumerate}
    We denote $\nor \bX\nor_{p,[s,t]}:=\|X\|_{p,[s,t]}+\|\mX\|_{p/2,[s,t]}$.
\end{definition}

\begin{definition}\label{App:def:A2}
Suppose that $X$ is a path in $C^{p-var}$ for some $p\in[2,3)$ and let $E$ be a finite dimensional normed space.
Let
\begin{align}\label{0124:06}
    q\in [1,\tfrac{p}{p-2})\quad \text{and}\quad r\in[1,\tfrac{p}{p-1}).
\end{align}
A continuous path \((Y,Y'):[0,T]\to E\times \cL(\Rd,E)\) belongs to \(\cV^{q,r}_X\) if \(Y\) has finite \(p\)-variation, \(Y'\) has finite \(q\)-variation and
\begin{align*}
 (s,t)\mapsto  R^Y_{s,t}:=Y_t-Y_s-Y_s'X_{s,t}
\end{align*}
has finite \(r\)-variation.
The quantity
\begin{align*}
\|Y,Y'\|_{q,r}:=|Y_0|+|Y_0'|+\|Y'\|_{q,[0,T]}+\|R^Y\|_{r,[0,T]}
\end{align*}
defines a norm on \(\cV^{q,r}_X\). We call \((Y,Y')\) a \textit{controlled rough path}.
\end{definition}
In the case when $q = 2r = p$, \cref{App:def:A2} reduces to the standard definition of a controlled rough path from \cite{Gub04}. In that framework, $\sigma$ is independent of time and the variation of the function $t \mapsto \sigma(X_t)$ is determined by the H\"older regularity of $x \mapsto \sigma(x)$ and by $X$. For time-dependent coefficients, the regularity of $t \mapsto \sigma(t, x)$ enters as well. The two additional parameters $q$ and $r$ are introduced to accommodate it. Since $\|\cdot\|_{p_2,[0,T]}\le\|\cdot\|_{p_1,[0,T]}$ whenever $p_1\le p_2$, we have $\cV^{q,r}_X\subset\cV^{q\vee p,r\vee\frac p2}_X$, and $(q\vee p,r\vee\frac p2)$ still satisfies \eqref{0124:06}.

As in \cite[Proposition 1]{Gub04} and \cite[Lemma 4.2]{FH}, we can define rough integrals of controlled paths in $\cV^{q,r}_X$.
\begin{theorem}\label{App:lem:A4}
Let $p\in[2,3)$, $\bX=(X,\mX)$ be a $p$-rough path, and $q,r$ satisfy \eqref{0124:06}. For any $(Y,Y')\in \cV_X^{q,r}$ with $E=\cL(\Rd,\mR^n)$ and any $t\in[0,T]$, the rough integral
\begin{align}\label{def:con}
\int_0^t Y_r\dif \bX_r:=\lim_{|\cP|\to 0}\sum_{[u,v]\in \cP}\left(Y_uX_{u,v}+Y_u'\mX_{u,v}\right) \quad \text{exists},
\end{align}
where $\cP$ denotes a generic partition of $[0,t]$ and $|\cP|$ its mesh.
Furthermore, there is a finite constant $C_{p,q,r}$ such that for any $0\le s \le t\le T$
\begin{align}\label{0125:01}
\left|\int_s^t Y_r\dif \bX_r-Y_sX_{s,t}-Y_s'\mX_{s,t} \right|\le C_{p,q,r}\left(\|R^Y\|_{r,[s,t]}\|X\|_{p,[s,t]}+\|Y'\|_{q,[s,t]}\|\mX\|_{p/2,[s,t]}\right).
\end{align}
\end{theorem}
The proof is a direct application of the sewing lemma and is given in \cref{appsec:A}.

\subsection{Rough differential equations}\label{sub.RDE}
Let $\bX=(X,\mX)$ be a $p$-rough path with some $p\in[2,3)$, $\sigma:[0,T]\times\mR^d\to\mR^d\otimes\mR^d$ be a measurable function which is differentiable on $\Rd$, and $Z:[0,T]\to\Rd$ be a continuous path.
We consider the RDE
\begin{align}\label{eqn.RDE}
    dY_t=\sigma(t,Y_t)d \bX_t+dZ_t.
\end{align}
Solutions of \eqref{eqn.RDE} are interpreted in the following sense, after \cite{MR2387018}.

\begin{definition}\label{def.daviesoln}
   A solution to the RDE$( \sigma, Z)$, namely \eqref{eqn.RDE}, (with input $(\bX, \sigma, Z)$) on $[0,T]$ is a continuous path $Y:[0,T]\to\Rd$ such that the map
    \begin{align*}
         (s,t)\mapsto J_{s,t}=J^Y_{s,t}:=  Y_t-Y_s-Z_t+Z_s-\sigma(s,Y_s)X_{s,t}-\nabla \sigma(s,Y_s)\sigma(s,Y_s)\mxx_{s,t}
    \end{align*}
    has finite \(\varepsilon\)-variation on $\Delta([0,T])$ for some \(\varepsilon\in(0,1)\).
    In other words, there exists a control \(w\) such that $ |J_{s,t}|\le w(s,t)^{\frac{1}{\varepsilon} }$ whenever $(s,t)\in \Delta([0,T])$.
\end{definition}
Here $\nabla\sigma(s,Y_s)\sigma(s,Y_s)\mxx_{s,t}$ denotes the contraction $\partial_k\sigma^i_{j}\,\sigma^k_{l}\,\mxx^{lj}_{s,t}$, where $\mxx^{lj}_{s,t}=\int_s^t X^l_{s,r}\dif X^j_r$ for smooth $X$; similarly for $Y'_s\mxx_{s,t}$.
\smallskip

To derive well-posedness results for \eqref{eqn.RDE}, we need to introduce regularity conditions on the coefficient $\sigma$.
 Let $f:[0,T]\times \mR^d\to \mR$ be a measurable function. We set $$\|f\|_{C^{p-var}L^\infty }:=\|f\|_{C^{p-var}([0,T];L^\infty )}$$ and for any $\alpha>0$,
 \begin{align*}
     \|f\|_{C^{p-var}C^\alpha}:=\|f\|_{C^{p-var}([0,T];C^\alpha)}
     \tand \|f\|_{L^\infty C^\alpha}:=\|f\|_{L^\infty ([0,T];C^\alpha)}.
 \end{align*}
 For each $k=0,1$, we introduce the following assumption.
\begin{assumption}\SsigmaK{}\label{asm.sigma}\hypertarget{assumption:SsigmaK}{}
    Suppose $p\in[2,3)$ and $(q,r)$ satisfies \eqref{0124:06}. Assume $(t,x)\to \sigma(t,x)$ is continuous on $[0,T]\times \mR^d$ and
    \begin{align*}
\Lambda^k_{q,r}(\sigma):=\|\sigma\|_{C^{q-var}C^{1+k}}+\|\sigma\|_{C^{r-var} C^k }+\|\sigma\|_{L^\infty C^{\frac{p+q}{q}\vee\frac{p}{r}+k}}<\infty.
\end{align*}
\end{assumption}
Note that $\Lambda^k_{q\vee p,r\vee\frac p2}(\sigma)\le C\Lambda^k_{q,r}(\sigma)$ with $C=C(p,q,r)$ (a lower order H\"older norm is bounded by a constant multiple of a higher order one), so that \SsigmaK{} with $(q,r)$ implies \SsigmaK{} with $(q\vee p,r\vee\frac p2)$.

Under \Ssigmao{}, for any $(Y,Y')\in\cV^{q,r}_X$ we set $\sigma(Y)_t:=\sigma(t,Y_t)$ and $\sigma(Y)'_t:=\nabla\sigma(t,Y_t)Y'_t$; then $(\sigma(Y),\sigma(Y)')$ is again a controlled rough path (see \cref{lem:030800}), so that $\int_0^t\sigma(r,Y_r)\dif\bX_r$ is defined by \cref{App:lem:A4}. The following result relates \cref{def.daviesoln} with the more familiar integral form of \eqref{eqn.RDE}.
\begin{proposition}\label{sec2:def}
Assume \Ssigmao{} holds with $q\ge p$ and $r\ge \frac p2$, and let $Z\in C^{r-var}$.
\begin{enumerate}[(i)]
\item If $Y$ is a solution to RDE$(\sigma,Z)$ on $[0,T]$, then $(Y,\sigma(Y))\in \cV^{q,r}_X$ and
\begin{align}\label{App:RDE}
    Y_t=Y_0+\int_0^t \sigma(r,Y_r)\dif \bX_r+Z_t-Z_0, \quad t\in[0,T].
\end{align}
\item Conversely, if $(Y,\sigma(Y))\in \cV^{q,r}_X$ and \eqref{App:RDE} holds, then $Y$ is a solution to RDE$(\sigma,Z)$ on $[0,T]$.
\end{enumerate}
\end{proposition}
We can now state well-posedness results for \eqref{eqn.RDE}.
\begin{theorem}\label{App:thm:00}
 Assume {\bf(S$_\sigma^1$)} holds. Then for any $Z\in C^{r-var}$, any $y\in\mR^d$ and any $p$-rough path $\bX=(X,\mX)$, there is a unique solution $Y$ to RDE$( \sigma, Z)$ on $[0,T]$ with $Y_0=y$.
 Moreover, $(Y,\sigma(Y))\in\cV^{q\vee p,r\vee\frac p2}_X$ and there is a constant $C$ depending only on $T,p,q,r,\Lambda^1_{q,r}(\sigma),\nor\bX\nor_{p,[0,T]}$ and an upper bound of $\|Z\|_{r,[0,T]}$ such that
 \begin{align}\label{0124:07}
 \|\sigma(Y)\|_{q\vee p,[0,T]}+\|R^Y\|_{r\vee\frac p2,[0,T]}\le C.
 \end{align}
\end{theorem}

\begin{theorem}\label{App:thm:sta}
Let $p\in[2,3)$, $q\ge p$ and $r\ge \frac{p}{2}$ satisfy \eqref{0124:06}. Let $\sigma$ and $\tilde{\sigma}$ satisfy \Ssigmai{}, and let $Y$ and $\tilde{Y}$ be the solutions to RDE$( \sigma, 0)$ and RDE$( \tilde\sigma, 0)$ on $[0,T]$ respectively.
Then there is a constant $C>0$, which depends only on $T,p,q,r,\Lambda_{q,r}^1(\sigma),\Lambda_{q,r}^1(\tilde{\sigma})$ and $\nor\bX\nor_{p,[0,T]}$, such that
\begin{align*}
\|\sigma(Y)-\tilde\sigma(\tilde{Y})\|_{q,[0,T]}+\|R^Y-R^{\tilde{Y}}\|_{r,[0,T]}\le C\left(|Y_0-\tilde{Y}_0|+\Lambda_{q,r}^0(\sigma-\tilde{\sigma})\right).
\end{align*}
\end{theorem}

\begin{theorem}\label{cor:0125:00}
Let $p\in[2,3)$, let $q\ge p$ and $r\ge \frac{p}{2}$ satisfy \eqref{0124:06}, and let $\sigma$ satisfy \Ssigmai{}.
Let $Z\in C^{r-var}$ with $Z_0=0$ and $\|Z\|_{r,[0,T]}\le K$. Let $Y$ and $Y^Z$ be the solutions to RDE$( \sigma, 0)$ and RDE$( \sigma, Z)$ respectively, with $Y_0=Y^Z_0$. Then there is a constant $C>0$, which only depends on $T,p,q,r,\Lambda_{q,r}^1(\sigma),\nor\bX\nor_{p,[0,T]}$ and $K$, such that for all $t\in[0,T]$
\begin{align*}
|Y^{Z}_t-Y_t|\le C\|Z\|_{r,[0,t]}.
\end{align*}
\end{theorem}
\cref{App:thm:00,cor:0125:00,App:thm:sta} extend standard results for rough differential equations (cf. \cite[Theorems 8.3, 8.5]{FH}) to time-dependent coefficients and perturbation paths. Their proofs, together with the proof of \cref{sec2:def}, are presented in \cref{sec:A.3}.

\section{SDEs and RDEs driven by Brownian rough paths}\label{sec.SDERDE}

Let $T>0$. In this section, we collect results on the equation \eqref{eqn.SDE}, regarded both as an SDE and as an RDE, and we establish the connection between weak solutions and path-by-path (rough) solutions. 

\subsection{SDEs}\label{sec:31}

For any $(s,x)\in[0,T)\times\mR^d$, we consider the following SDE:
\begin{align}\label{Pre:SDE}
    X_{t}^{s,x}=x+\int_s^t b(r,X_{r}^{s,x})d r+\int_s^t \sigma(r,X_{r}^{s,x}) dW_r,\ \ s\le t\le T,
\end{align}
where $b$ satisfies the condition \Hb{} with some $(p,q)\in\cJ$ and $\sigma$ satisfies the following assumption:
\begin{assumption}\label{asm.As}{\Asigma}\hypertarget{assumption:Asigma}{} There is a constant $\kappa_1>1$ such that for all $(t,x,\xi)\in\mR_+\times\mR^d\times\mR^d$
    \begin{align*}
        \kappa_1^{-1}|\xi|\le |\sigma(t,x)\xi|\le \kappa_1|\xi|.
    \end{align*}
    Moreover, there is a $(p_0,q_0)\in\mathcal{J}$ such that 
    \begin{align*}
        \nor\nabla_x\sigma\nor_{\widetilde{\mL}^{p_0}_{q_0}(T)}\le\kappa_1.
    \end{align*}
\end{assumption}
Throughout this subsection, we assume \Asigma{} and \Hb{}, and we write 
$$
\Theta_X:=(d,T,p,q,\nor b\nor_{\widetilde{\mL}^{p}_{q}(T)},p_0,q_0,\kappa_1)
$$
for the set of parameters on which constants may depend. The subscript $X$ indicates the parameter set is associated with the SDE flow $X$.
It is well-known that for any $(s,x)\in[0,T)\times\mR^d$, there is a unique strong solution $\{X_{t}^{s,x}\}_{t\in[s,T]}$ to \eqref{Pre:SDE} (see \cite{XXZZ}). Moreover, by \cite[page 5207, (4.10)]{XXZZ}, for any $m\ge1$ there is a constant $C=C(\Theta_X,m)>0$ such that for any $x,y\in\mR^d$ and $s\in[0,T)$
\begin{align}\label{Pre:SFp}
    \E\left(\sup_{t\in[s,T]}|X_t^{s,x}-X_t^{s,y}|^m\right)\le C|x-y|^m.
\end{align}
By \eqref{Pre:SFp} and Kolmogorov's continuity theorem, we may and do choose a version of $X^{s,x}$ which is continuous in $x$. This result can be extended to random initial data.
\begin{lemma}\label{lem:SF}
    Let $\xi,\eta$ be $\sF_s$-measurable random variables, and let $X^{s,\xi}$ and $X^{s,\eta}$ be the unique strong solutions to SDE \eqref{Pre:SDE} with initial data $\xi$ and $\eta$ respectively. Then, for any $m\ge1$, there is a constant $C=C(\Theta_X,m)>0$ such that
    \begin{align}\label{Pre:SF}
        \E\left(\sup_{t\in[s,T]}|X_t^{s,\xi}-X_t^{s,\eta}|^m\right)\le C\E|\xi-\eta|^{m}.
    \end{align}
\end{lemma}
\begin{proof}
Let $(x,\omega)\mapsto X^{s,x}(\omega)$ be the version which is continuous in $x$; it is jointly measurable and, for each $x$, measurable with respect to $\sigma(W_r-W_s;r\in[s,T])$, hence independent of $\sF_s$. For an $\sF_s$-measurable $\xi$ put $Y:=X^{s,x}|_{x=\xi}$, a continuous adapted process.

\emph{Step 1: $Y$ solves \eqref{Pre:SDE} with initial value $\xi$.} If $\xi=\sum_k x_k\bbb1_{A_k}$ is simple ($A_k\in\sF_s$ a partition), then $Y=\sum_k X^{s,x_k}\bbb1_{A_k}$, and since $A_k\in\sF_s$ we have $\bbb1_{A_k}\int_s^t\sigma(r,X^{s,x_k}_r)\dif W_r=\bbb1_{A_k}\int_s^t\sigma(r,Y_r)\dif W_r$; hence $Y$ solves \eqref{Pre:SDE} with initial value $\xi$. For general $\xi$ take simple $\xi_n\to\xi$ pointwise and put $Y^n:=X^{s,x}|_{x=\xi_n}$. By continuity in $x$, $\sup_t|Y^n_t-Y_t|\to0$ almost surely.
Hence, by independence of $\xi$ and $(X^{s,x})_x$ and Fubini, for every Borel $f\ge0$ and $(p',q')$ as in \cref{prop:ItoKrylov}~(ii), \eqref{Kry01} gives $\mE\int_s^T f(r,Y_r)\dif r=\int\mP_\xi(\dif x)\,\mE\int_s^T f(r,X^{s,x}_r)\dif r\le C\nor f\nor_{\widetilde{\mL}^{p'}_{q'}(T)}$, and likewise for $Y^n$ (on $A_k$ it equals $X^{s,x_k}$, and $\mE[\bbb1_{A_k}\int_s^T f(r,X^{s,x_k}_r)\dif r]=\mP(A_k)\,\mE\int_s^T f(r,X^{s,x_k}_r)\dif r$), with $C$ independent of $n$. The Krylov estimate is available here because \eqref{Kry01} concerns the strong solution started from a \emph{deterministic} point and does not rely on the present lemma (only the final sentence of \cref{prop:ItoKrylov}, on random initial data, does). 
We apply this with $f=f_k$ below, for which $(p',q')=(p,q)\in\cJ$.

Since $\sigma$ is bounded and continuous and $Y^n\to Y$ uniformly, $\int_s^t\sigma(r,Y^n_r)\dif W_r\to\int_s^t\sigma(r,Y_r)\dif W_r$ in probability. For the drift, fix $R\in\mN$, let $b^R:=b\bbb1_{\{|x|\le R\}}\in\mL^p_q(T)$, choose $b_k\in C_c([0,T]\times\Rd)$ with $b_k\to b^R$ in $\mL^p_q(T)$ and put $f_k:=|b^R-b_k|\bbb1_{\{|x|\le R\}}$. On $\Omega_{R,n}:=\{\sup_t|Y_t|\vee\sup_t|Y^n_t|\le R\}$,
\begin{multline*}
    \Big|\int_s^t [b(r,Y^n_r)-b(r,Y_r)]\dif r\Big|
    \\\le\int_s^T f_k(r,Y^n_r)\dif r+\int_s^T f_k(r,Y_r)\dif r+\Big|\int_s^t[b_k(r,Y^n_r)-b_k(r,Y_r)]\dif r\Big|.
\end{multline*}
The expectation of the first two terms is at most $2C\nor f_k\nor_{\widetilde{\mL}^p_q(T)}\to0$ ($k\to\infty$), uniformly in $n$, and the last term tends to $0$ as $n\to\infty$ for each $k$. Since $Y^n\to Y$ uniformly, $\bbb1_{\Omega_{R,n}}\to1$ a.s. on $\{\sup_t|Y_t|<R\}$, and $R$ is arbitrary, we get $\int_s^t b(r,Y^n_r)\dif r\to\int_s^t b(r,Y_r)\dif r$ in probability. Passing to the limit in the equation for $Y^n$ proves Step 1; $\int_s^T|b(r,Y_r)|\dif r<\infty$ a.s. follows from the Krylov estimate.

\emph{Step 2.} By Step 1 and pathwise uniqueness, $X_t^{s,\xi}=X_t^{s,x}|_{x=\xi}$ for all $t\in[s,T]$, $\mP$-a.s. Since for any $x\in\mR^d$, $(X_t^{s,x})_{t\in[s,T]}$ is independent of $\sF_s$, \eqref{Pre:SFp} implies that
\begin{align*}
        \E\left(\sup_{t\in[s,T]}|X_t^{s,\xi}-X_t^{s,\eta}|^m\right)=\E\left(\E\left(\sup_{t\in[s,T]}|X_t^{s,x}-X_t^{s,y}|^m\right)\Big|_{(x,y)=(\xi,\eta)}\right)\le C\E|\xi-\eta|^{m}.
    \end{align*}
\end{proof}
Next we recall the generalized It\^o formula and the Krylov-type estimate.
\begin{proposition}\label{prop:ItoKrylov}
    (i) {\rm[Generalized It\^o's formula] (\cite[Lemma 4.1-(iii)]{XXZZ})}
Let $p',q'\in[2,\infty)$ with $2/q'+d/p'<1$.
For any $u$ with $\partial_tu,\nabla^k u\in\widetilde{\mL}_{q'}^{p'}(T)$, $k=0,1,2$,
we have for any $0\le v\le s\le t\le T$ and $x\in\mR^d$
\begin{align}\label{GIF001}
\begin{split}
u(t,X^{v,x}_{t})=u(s,X^{v,x}_s)&+\int_s^t(\partial_ru+a_{ij}\partial_i\partial_j u
+b\cdot\nabla u)(r,X_{r}^{v,x})d r\\
&+\int_s^t (\partial_i u\,\sigma_{ij})(r,X_{r}^{v,x})d W^j_r,
\end{split}
\end{align}
where $a_{ij}=\frac12\sigma_{ik}\sigma_{jk}$.\\
(ii) {\rm[Krylov-type estimate] (\cite[Lemma 3.7-(ii)]{CHR22})} Let $p',q'\in(1,\infty)$ with $2/q'+d/p'<2$. There is a constant $C=C(\Theta_X,p',q')>0$ such that for any $x\in\mR^d$, $0\le v\le s\le t\le T$ and $f\in \widetilde{\mL}^{p'}_{q'}(T)$
\begin{align}\label{Kry01}
\E\left|\int_s^t f(r,X_{r}^{v,x})d r\right|\le C(t-s)^{1-\frac{1}{q'}-\frac{d}{2p'}}\nor f\nor_{\widetilde{\mL}^{p'}_{q'}(T)}.
\end{align}
(iii) Under the conditions of (ii), by the Markov property and the classical Khasminskii argument (see, e.g., \cite[Lemma 4.1]{XXZZ}) applied to $|f|$, for any $m\in\mN$,
\begin{align}\label{Kry02}
\left(\E\left|\int_s^t f(r,X_{r}^{v,x})d r\right|^m\right)^{1/m}\le Cm(t-s)^{1-\frac{1}{q'}-\frac{d}{2p'}}\nor f\nor_{\widetilde{\mL}^{p'}_{q'}(T)}.
\end{align}
By the substitution argument in the proof of \cref{lem:SF}, \eqref{GIF001}--\eqref{Kry02} remain valid when $x$ is replaced by an $\sF_v$-measurable random variable.
\end{proposition}

\subsection{RDEs with $b\equiv0$}\label{sec:4.5}
Let $\bW=(W,\mW)$ be the It\^o lift of the $d$-dimensional Brownian motion $W$, that is, $\mW_{s,t}:=\int_s^t W_{s,r}\otimes \dif W_r$ (It\^o integral). It is well known that, almost surely, $\bW$ is a $p$-rough path for every $p\in(2,3)$ (see \cite[Chapter 3]{FH}). Redefining $\bW$ on a null set, we assume without loss of generality that this holds for every $\omega\in\Omega$, and we omit the dependence on $\omega$ whenever no confusion arises.

Without loss of generality, we assume that the constant $\eps$ in \Hsigma{} belongs to $(0,1)$.
\begin{lemma}\label{lem:HS}
Assume \Hsigma{} holds and let $p\in(2,2+\eps)$. Then $\sigma$ satisfies {\textbf{(S$_\sigma^1$)}} on $[0,T]$ with
\begin{align}\label{n0928:00}
    q=\tfrac{2+\eps}{\eps}\in[p,\tfrac{p}{p-2}),\quad r=\tfrac{2+\eps}{1+\eps}\in [\tfrac{p}{2},\tfrac{p}{p-1}),
\end{align}
and $\Lambda^1_{q,r}(\sigma)\le C(d,\eps)\big(\kappa_0+\|\sigma\|_{C^{1-var}L^\infty}\big)$, where $\|\sigma\|_{C^{1-var}L^\infty}$ is the quantity in \eqref{in:con}.
\end{lemma}
\begin{proof}
Since $\sigma$ is bounded by $\kappa_0$ and $\nabla^2\sigma(t,\cdot)$ is $\eps$-H\"older continuous with constant $\kappa_0$, an elementary interpolation (e.g. using second order differences with step size $1$) shows that $\|\sigma\|_{L^\infty C^{2+\eps}}\le C(d,\eps)\kappa_0$. By \eqref{0305:01} with $(p_1,\alpha_1)=(1,0)$ and $(p_0,\alpha_0)=(\infty,2+\eps)$, and since $\|\sigma\|_{C^{\infty-var}C^{2+\eps}}\le 2\|\sigma\|_{L^\infty C^{2+\eps}}$, Young's inequality gives
\begin{align*}
   \|\sigma\|_{C^{\frac{2+\eps}{1+\eps}-var}C^1}+\|\sigma\|_{C^{\frac{2+\eps}{\eps}-var}C^2}\lesssim \|\sigma\|_{C^{1-var}L^\infty }+\|\sigma\|_{L^\infty C^{2+\eps}}.
\end{align*}
Since $C^{q-var}$ is decreasing in $q$, this controls the first two terms in $\Lambda^1_{q,r}(\sigma)$ with $q,r$ given by \eqref{n0928:00}. Moreover, a direct computation shows that $p<2+\eps$ is equivalent to each of $q<\frac p{p-2}$, $r<\frac p{p-1}$, $\frac{p+q}{q}+1< 2+\eps$ and $\frac pr+1<2+\eps$, while $q\ge p$ and $r\ge \frac p2$ hold since $\eps<1$. Hence the last term in $\Lambda^1_{q,r}(\sigma)$ is bounded by $\|\sigma\|_{L^\infty C^{2+\eps}}$.
\end{proof}
\begin{lemma}\label{lem:sigcont}
Assume \Hsigma{}. Then $t\mapsto\sigma(t,\cdot)$ is continuous from $\mR_+$ into $C^k(\Rd)$, $k=0,1,2$. Consequently, for $q\ge1$ and $k\in\{0,1,2\}$, the maps $(s,t)\mapsto\|\sigma\|^q_{C^{q\text{-var}}([s,t];C^k)}$ are controls, and $\nabla\sigma,\nabla^2\sigma$ are jointly continuous.
\end{lemma}
\begin{proof}
A Banach-space valued function of bounded variation has one-sided limits everywhere.
Let $\sigma(t^\pm)\in L^\infty$ be the one-sided limits of $t\mapsto\sigma(t,\cdot)$ in $L^\infty$. Norm convergence implies a.e.\ convergence along a subsequence, and by joint continuity of $\sigma$ the pointwise limit is $\sigma(t,\cdot)$; both being continuous, $\sigma(t^\pm)=\sigma(t,\cdot)$. Hence $\|\sigma(t_n)-\sigma(t)\|_{L^\infty}\to0$ whenever $t_n\to t$. Since $\|\sigma\|_{L^\infty C^{2+\eps}}<\infty$ (proof of \cref{lem:HS}), \cref{lem:C1} with $(\alpha_0,\alpha_1)=(0,2+\eps)$ gives $\|\sigma(t_n)-\sigma(t)\|_{C^k}\lesssim\|\sigma(t_n)-\sigma(t)\|_{L^\infty}^{1-k/(2+\eps)}\to0$. A continuous path of finite $q$-variation in a Banach space has a continuous $q$-variation control (see \cite[Proposition 5.8]{FN10}).
\end{proof}
Let \Hsigma{} hold, fix $p\in(2,2+\eps)$ and let $(q,r)$ be given by \eqref{n0928:00}. By \cref{lem:HS,App:thm:00}, for every $\omega\in\Omega$ and $(s,x)\in[0,T)\times\mR^d$, there is a unique solution $(\psi^{s,x}_t)_{t\in[s,T]}$ to the RDE
\begin{align}\label{S2:RDE}
   \psi^{s,x}_t=x+\int_s^t \sigma(r,\psi^{s,x}_r)\dif \bW_r.
\end{align}
Since the restriction of a solution to a subinterval is again a solution, uniqueness yields the flow property
\begin{align}\label{0316:05}
    \psi^{s,\psi^{v,x}_s}_t=\psi^{v,x}_t,\quad \text{for any $0\le v\le s\le t\le T$ and $x\in\mR^d$.}
\end{align}
Moreover, by \cref{App:thm:sta} (applied with $\tilde\sigma=\sigma$ on $[s,T]$), \eqref{0124:07} and \eqref{0316:05}, the map $(s,t,x)\mapsto \psi^{s,x}_t$ is continuous on $\{0\le s\le t\le T\}\times\Rd$ for every $\omega$. Indeed, for $v\le s\le t$ we have $|\psi^{v,x}_t-\psi^{s,y}_t|=|\psi^{s,\psi^{v,x}_s}_t-\psi^{s,y}_t|\le C(|\psi^{v,x}_s-x|+|x-y|)$, and, by \eqref{0124:02} and \eqref{n0927:00} applied on $[v,s]$, $|\psi^{v,x}_s-x|\le C\nor\bW\nor_{p,[v,s]}$, with $C$ depending only on $T$, $p$, $\Lambda^1_{q,r}(\sigma)$ and $\nor\bW\nor_{p,[0,T]}$.

\subsection{Relation between SDEs and RDEs with drifts}\label{sec:33}
Let $b:\mR_+\times\mR^d\to\mR^d$ be Borel measurable and let $\sigma$ satisfy the condition \Hsigma{}. Consider the following SDE
\begin{align}\label{0311:00}
X_t=X_0+\int_0^t b(r,X_r)\dif r+\int_0^t \sigma(r,X_r)\dif W_r.
\end{align}
\begin{definition}\label{def:413}
Let $\omega\in\Omega$. We call a continuous path $X:[0,T]\to\Rd$ a path-by-path solution on $[0,T]$ to the RDE
\begin{align}\label{0311:02}
X_t=X_0+\int_0^t b(r,X_r)\dif r+\int_0^t \sigma(r,X_r)\dif \bW_r(\omega),
\end{align}
if $\int_0^T|b(r,X_r)|\dif r<\infty$ and $X$ is a solution to RDE$(\sigma,Z^X)$ with input $\bW(\omega)$ in the sense of \cref{def.daviesoln}, where $Z^X_t:=\int_0^t b(r,X_r)\dif r$.
We say that path-by-path uniqueness holds for \eqref{0311:02} if, for every $x\in\Rd$, there is at most one path-by-path solution with $X_0=x$.
\end{definition}

\begin{remark}
Since $Z^X$ has finite $1$-variation, by \cref{lem:HS} and \cref{sec2:def}, a continuous path $X$ with $\int_0^T|b(r,X_r)|\dif r<\infty$ is a path-by-path solution if and only if $(X,\sigma(X))\in \cV^{q,r}_{W}$ and \eqref{0311:02} holds, where $(q,r)$ is given by \eqref{n0928:00} and the last integral in \eqref{0311:02} is the rough integral of \cref{App:lem:A4}. In particular, \cref{def:413} does not depend on the choice of $p\in(2,2+\eps)$.
\end{remark}

The following result is a version of \cite[Proposition 5.1]{FH}, which identifies It\^o and rough integrals, adapted to our setting.
\begin{theorem}\label{thm:0311}
Assume \Hsigma{} holds, let $p\in(2,2+\eps)$ and let $(q,r)$ be given by \eqref{n0928:00}. Let $(\Omega,\sF,(\sF_t),\mP)$ be a filtered probability space with an $(\sF_t)$-Brownian motion $W$, and let $\bW$ be the It\^o lift of $W$.
\begin{enumerate}[(i)]
\item Let $X$ be any continuous $(\sF_t)$-adapted process (not necessarily a solution) such that $(X,\sigma(X))\in\cV^{q,r}_W$ almost surely. Then for any $g:[0,T]\times\Rd\to \cL(\Rd,\mR^n)$ satisfying \Ssigmao{} (with $\sigma$ replaced by $g$) for the above $(p,q,r)$,
\begin{align}\label{0314:00}
\int_0^t g(r,X_r)\dif W_r=\int_0^t g(r,X_r)\dif \bW_r\quad \text{for all } t\in[0,T],\quad \mP\text{-a.s.},
\end{align}
\item If $(X_t)_{t\in[0,T]}$ is a weak solution to \eqref{0311:00} on this space and
\begin{align*}
    \mE\left(\int_0^T|b(t,X_t)|\dif t\right)^{1+\delta}<\infty
\end{align*}
for some $\delta>0$, then $(X,\sigma(X))\in\cV^{q,r}_W$ almost surely, and there is an event $\bar\Omega=\bar\Omega(W,X)$ of full probability such that for all $\omega\in\bar\Omega$, $(X_t(\omega))_{t\in[0,T]}$ is a path-by-path solution to \eqref{0311:02}.
\item Conversely, any continuous $(\sF_t)$-adapted process $X$ such that, for almost every $\omega$, $X(\omega)$ is a path-by-path solution to \eqref{0311:02}, is a weak solution to \eqref{0311:00}.
\end{enumerate}
\end{theorem}
The proof is given in \cref{sec:A.4}.

\section{Regularization estimates}\label{sec:4}

\subsection{Maximal regularity estimates for parabolic PDEs in $\widetilde{\mL}^p_q$}\label{sec:24}
In this part, we consider the following parabolic PDE on $\mR_+\times\mR^d$:
\begin{align}\label{Pre4:PDE}
    \partial_t u=a_{ij}\partial_i\partial_j u+b\cdot\nabla u-\lambda u+f,\quad u(0)=0,
\end{align}
where $\lambda\ge0$, $a=(a_{ij}):\mR_+\times\R^d\rightarrow\R^d\otimes\R^d$
is a symmetric matrix-valued Borel measurable function satisfying the following condition \Aa{}, and $b:\mR_+\times\R^d\rightarrow\R^d$ is a vector-valued Borel measurable function.
\begin{assumption}\label{asm.a}{\bf(A$_a$)}\hypertarget{assumption:Aa}{} There exist constants $c_0>0$ and $\beta\in(0,1)$ such that for all $t\ge0$ and $\xi,x,y\in\R^d$
  \begin{equation*}
    c_0^{-1}|\xi|^2\leq\langle a(t,x)\xi,\xi\rangle\leq c_0|\xi|^2,\quad
    |a(t,x)-a(t,y)|\leq c_0|x-y|^\beta.
  \end{equation*}
\end{assumption}
First we introduce the definition of a solution to PDE \eqref{Pre4:PDE}.
\begin{definition}\label{Pre4:Def}\rm
Let $T>0$, $p,q\in(1,\infty)$, $\lambda\ge0$ and
$b,f\in \widetilde{\mL}^p_q(T)$. We call a measurable function $u$ a {\em solution} of PDE \eqref{Pre4:PDE} on $[0,T]$ if $\partial_tu,\nabla^ku\in \widetilde{\mL}^p_q(T)$, $k=0,1,2$, and for Lebesgue almost all $(t,x)\in[0,T]\times\mR^d$
\begin{align*}
u(t,x)=\int_0^t\left(a_{ij}\partial_i\partial_ju(s,x)-\lambda u(s,x)+b\cdot\nabla u(s,x)+f(s,x)\right)d s.
\end{align*}
\end{definition}

The following result concerns the well-posedness of the PDE \eqref{Pre4:PDE}.
\begin{theorem}\label{lem210}
Let $T>0$, $p,q\in(1,\infty)$ with $2/q+d/p<1$, $b\in \widetilde{\mL}^{p}_{q}(T)$, and assume \Aa{}. Set
$$
\Theta:=(d,T,p,q,\nor b\nor_{\widetilde{\mL}^{p}_{q}(T)},c_0,\beta).
$$
Then for all $\lambda\ge0$ and $f\in \widetilde{\mL}^{p}_{q}(T)$,
there is a unique solution $u$ to the PDE \eqref{Pre4:PDE} on $[0,T]$ in the sense of Definition \ref{Pre4:Def}. Moreover,
for any $k=0,1$, $p'\in[p,\infty]$, $q'\in[q,\infty]$ with
$
\beta_k:=2-k+2/q'+d/p'-(2/q+d/p)>0,
$
there is a constant $C=C(\Theta,p',q')>0$ such that for all $\lambda\ge0$
\begin{align}\label{CCA02}
(\lambda+1)^{\frac{\beta_k}{2}}
\nor \nabla^ku\nor_{\widetilde{\mL}^{p'}_{q'}(T)}
+\nor\partial_tu\nor_{\widetilde{\mL}^p_q(T)}
+\nor\nabla^2 u\nor_{\widetilde{\mL}^p_q(T)}
\le C\nor f\nor_{\widetilde{\mL}^p_q(T)}.
\end{align}
\end{theorem}
\begin{proof}
  By \cite[Theorem 3.2]{XXZZ}, there is a constant $\lambda_0=\lambda_0(\Theta)$ such that for all $\lambda\ge\lambda_0$ and $f\in \widetilde{\mL}^{p}_{q}(T)$,
there is a unique solution $u$ to the PDE \eqref{Pre4:PDE} on $[0,T]$ satisfying \eqref{CCA02}.
For $\lambda\in[0,\lambda_0]$, the same holds by \cite[Proposition 2.6]{CHR22}.
\end{proof}
By the change of variables $r\mapsto t_0-r$, \cref{lem210} also applies, with the same constants, to the backward equation $\partial_r u+a_{ij}\partial_i\partial_j u+b\cdot\nabla u-\lambda u+f=0$ on $[0,t_0]$ with $u(t_0)=0$, for any $t_0\in(0,T]$.

\subsection{Davie estimates for stochastic flows}\label{sub.DavieEst}
Throughout this subsection, we fix $(p,q)\in\cJ$, set $\vartheta:=\frac dp+\frac2q<1$, and assume that \Asigma{} holds, that $a=(a_{ij})$ with $a_{ij}:=\frac12 \sigma_{ik}\sigma_{jk}$ satisfies \Aa{}, and that $b\in\widetilde{\mL}^p_q(T)$. For instance, all of this holds if \Hsigma{} and \Hb{} hold, where $(p,q)$ is the pair in \Hb{}, or $b\equiv0$ and $(p,q)\in\cJ$ is arbitrary. We denote by $X^{s,x}$ the unique strong solution to \eqref{Pre:SDE}, and constants below depend on $\Theta_X$, $c_0$ and $\beta$, unless stated otherwise.
We aim to show variants of the Davie estimates \cite[Propositions 2.1, 2.2]{Davie}. We use the local Hardy--Littlewood maximal function
\begin{align*}
\mathcal{M} g(x):=\sup_{r\in(0,1)}\frac{1}{|B_r|}\int_{B_r}|g(x+y)| d y,
\end{align*}
which is bounded on $\widetilde{\mL}^{p'}_{q'}(T)$ for every $p'\in(1,\infty]$, $q'\in[1,\infty]$. By \cite[Lemma 2.1]{XXZZ}, for any bounded $g$ with $\nabla g\in L^1_{loc}$, there is a Lebesgue null set $N$ (for instance, the complement of the common Lebesgue points of $g$ and $\nabla g$) such that for all $x',y'\notin N$,
\begin{align}\label{maxineq}
|g(x')-g(y')|\le C_d|x'-y'|\left(\mathcal{M}|\nabla g|(x')+\mathcal{M}|\nabla g|(y')+\|g\|_\infty\right),
\end{align}
where the term $\|g\|_\infty$ accounts for the case $|x'-y'|>1$. By \eqref{Kry01} with $f=\bbb1_{N}$, $\int_s^t\bbb1_N(r,X^{v,x}_r)\dif r=0$ almost surely for any Lebesgue null set $N\subset[0,T]\times\Rd$. Hence, if $g=g(r,\cdot)$ depends measurably on $r$, so that the corresponding space-time exceptional set (defined via Lebesgue points as above) is Borel, \eqref{maxineq} can be applied along $X^{v,x}_r$ for almost every $r$, almost surely.

\begin{theorem}\label{lem:21}
For any $m\in\mathbb{N}$, there is a constant $C=C(\Theta_X,c_0,\beta,m)>0$ such that for any $f\in \widetilde{\mL}^p_q(T)$, $0\le v\le s\le t\le T$ and $\sF_v$-measurable random variables $x,y$,
\begin{align}\label{lem21}
\left(\E\left|\int_s^t [f(r,X^{v,x}_r)-f(r,X^{v,y}_r)]dr\right|^m\right)^{\frac1m}\le C\left(\E|x-y|^{2m}\right)^{\frac1{2m}}(t-s)^{\frac{1-\vartheta}{2}}\nor f\nor_{\widetilde{\mL}^p_q(T)}.
\end{align}
\end{theorem}

\begin{proof}
Fix $f\in \widetilde{\mL}^p_q(T)$ and $0\le v\le s<t\le T$, and let $\lambda\ge0$ be such that $1+\lambda=(t-s)^{-1}\vee1$. Let $u$ be the solution to the backward PDE
 \begin{align*}
  \partial_r u+a_{ij}\partial_i\partial_j u+b\cdot \nabla u-\lambda u+f=0,\quad r\in[0,t],\quad u(t)=0.
 \end{align*}
By \cref{lem210} with $(k,p',q')=(1,\infty,\infty)$, for which $\beta_1=1-\vartheta$, we have
\begin{align}\label{AA01}
    (\lambda+1)^{\frac{1-\vartheta}{2}}\|\nabla u\|_{L^\infty ([0,t]\times \mathbb{R}^d)}+\nor\nabla^2 u\nor_{\widetilde{\mL}^p_q(t)}\le c\nor f\nor_{\widetilde{\mL}^p_q(T)},
\end{align}
where the constant $c$ is independent of $\lambda\ge0$ and $t\in(0,T]$. Indeed, $u$ is the restriction to $[0,t]$ of the solution on $[0,T]$ of the backward equation with source $f\bbb1_{[0,t]}$ and terminal value $0$ at time $T$ (by uniqueness, the latter vanishes on $[t,T]$), and $\nor f\bbb1_{[0,t]}\nor_{\widetilde{\mL}^p_q(T)}\le\nor f\nor_{\widetilde{\mL}^p_q(T)}$, so \cref{lem210} applies on $[0,T]$.
By the generalized It\^o formula \eqref{GIF001} applied to $r\mapsto u(r,X^{v,x}_r)$ on $[s,t]$,
 \begin{align}\label{AA00}
 \int_s^t f(r,X^{v,x}_r) d r=u(s,X^{v,x}_s)+\lambda\int_s^t u(r,X^{v,x}_r) d r+\int_s^t h_j(r,X^{v,x}_r) d W^j_r,\qquad h_j:=\sigma_{ij}\partial_i u.
 \end{align}
Denote the left-hand side of \eqref{lem21} by $I$. Since $\lambda(t-s)\le1$, \eqref{AA00}, the BDG inequality and \eqref{Pre:SF} yield
\begin{align}\label{AA02}
\begin{split}
I&\le C\left(\E|x-y|^{m}\right)^{\frac{1}{m}}\|\nabla u\|_{L^\infty ([0,t]\times \mathbb{R}^d)}\\
&\quad+C\left(\E\left[\int_s^t |h(r,X^{v,x}_r)-h(r,X^{v,y}_r)|^2 d r\right]^{\frac{m}{2}}\right)^{\frac1m}=:I_1+I_2.
\end{split}
\end{align}
By \eqref{AA01}, $I_1\le C(\E|x-y|^m)^{1/m}(t-s)^{\frac{1-\vartheta}2}\nor f\nor_{\widetilde{\mL}^p_q(T)}$ (if $t-s\ge1$, we use $(t-s)^{\frac{1-\vartheta}{2}}\ge1$).
Since $|\nabla h|\le |\nabla\sigma||\nabla u|+\kappa_1|\nabla^2u|$ and $\|h\|_\infty\le\kappa_1\|\nabla u\|_\infty$, \eqref{maxineq} gives
\begin{align*}
|h(r,X^{v,x}_r)-h(r,X^{v,y}_r)|\le C|X^{v,x}_r-X^{v,y}_r|\left(\Phi(r,X^{v,x}_r)+\Phi(r,X^{v,y}_r)+\|\nabla u\|_\infty\right),
\end{align*}
where $\Phi:=\mathcal{M}|\nabla^2 u|+\|\nabla u\|_\infty\mathcal{M}|\nabla\sigma|$. By H\"older's inequality and \eqref{Pre:SF},
\begin{align}\label{AB00}
I_2\le C\left(\E|x-y|^{2m}\right)^{\frac1{2m}} \left(\E\left[\int_s^t \left(\Phi(r,X^{v,x}_r)^2+\Phi(r,X^{v,y}_r)^2+\|\nabla u\|^2_\infty\right) d r\right]^{m}\right)^{\frac1{2m}}.
\end{align}
We apply \eqref{Kry02} with $(p',q')=(p/2,q/2)$ to $(\mathcal{M}|\nabla^2u|)^2$ and with $(p',q')=(p_0/2,q_0/2)$ to $(\mathcal{M}|\nabla\sigma|)^2$; this is possible since $\frac{4}{q}+\frac{2d}{p}=2\vartheta<2$, and similarly for $(p_0,q_0)$. With $\vartheta_0:=\frac{d}{p_0}+\frac{2}{q_0}<1$, we obtain
\begin{align*}
I_2&\le C\left(\E|x-y|^{2m}\right)^{\frac1{2m}}\left((t-s)^{\frac{1-\vartheta}{2}}\nor \nabla^2u\nor_{\widetilde{\mL}^p_q(t)}+\big((t-s)^{\frac{1-\vartheta_0}{2}}+(t-s)^{\frac12}\big)\|\nabla u\|_\infty\right)\\
&\le C\left(\E|x-y|^{2m}\right)^{\frac1{2m}}(t-s)^{\frac{1-\vartheta}{2}}\nor f\nor_{\widetilde{\mL}^p_q(T)},
\end{align*}
where in the last step we used \eqref{AA01}: the factor $\|\nabla u\|_\infty\le c(1+\lambda)^{-\frac{1-\vartheta}2}\nor f\nor_{\widetilde{\mL}^p_q(T)}\le c(t-s)^{\frac{1-\vartheta}2}\nor f\nor_{\widetilde{\mL}^p_q(T)}$ (when $t-s<1$) compensates for the possibly smaller exponent $\frac{1-\vartheta_0}{2}$. Combining the estimates of $I_1$ and $I_2$, we complete the proof.
\end{proof}
As a consequence, we have the following result.
\begin{corollary}\label{cor:23}
For any $m\in\mN$, there is a constant $C=C(\Theta_X,c_0,\beta,m)>0$ such that for any $f\in \widetilde{\mL}^p_q(T)$, $0\le u,v\le s\le t\le T$ and $x,y\in\mathbb{R}^d$
\begin{align}\label{cor23}
\left(\E\left|\int_s^t [f(r,X^{u,x}_r)-f(r,X^{v,y}_r)]dr\right|^m\right)^{\frac1m}\le C(|x-y|+|u-v|^{\frac12})(t-s)^{\frac{1-\vartheta}{2}}\nor f\nor_{\widetilde{\mL}^p_q(T)}.
\end{align}
\end{corollary}
\begin{proof}
Without loss of generality, we assume $v\le u$.
By the strong well-posedness of the SDE \eqref{Pre:SDE}, $X^{v,y}_r=X^{u,X^{v,y}_{u}}_r$ for all $r\in[u,T]$, almost surely. Hence \eqref{lem21} with initial time $u$ and initial data $x$ and $X^{v,y}_u$ gives
\begin{align*}
&\left(\E\left|\int_s^t [f(r,X^{u,x}_r)-f(r,X^{v,y}_r)]dr\right|^m\right)^{\frac1m}\\
&\quad\le C\big(|x-y|+\big(\E|X^{v,y}_{u}-y|^{2m}\big)^{\frac{1}{2m}}\big)(t-s)^{\frac{1-\vartheta}{2}}\nor f\nor_{\widetilde{\mL}^p_q(T)}.
\end{align*}
By \eqref{Kry02} applied to $b$ and the BDG inequality, $(\E|X^{v,y}_{u}-y|^{2m})^{\frac{1}{2m}}\le C(u-v)^{1-\frac\vartheta2}+C(u-v)^{\frac12}\le C|u-v|^{\frac12}$, which completes the proof.
\end{proof}

We next consider stochastic integrals.
\begin{proposition}\label{lem:22}
For any $m\in\mathbb{N}$, there is a constant $C=C(\Theta_X,m)>0$ such that for any bounded measurable $g:[0,T]\times\Rd\to\cL(\Rd,\mR^n)$ with $\nabla g\in \widetilde{\mL}^{p}_q(T)$, $0\le v\le s\le t\le T$ and $\sF_v$-measurable random variables $x,y$,
\begin{align*}
&\left(\E\left|\int_s^t [g(r,X^{v,x}_r)-g(r,X^{v,y}_r)]d W_r\right|^m\right)^{\frac1m}\\
&\quad\le C\left(\E|x-y|^{2m}\right)^{\frac1{2m}}(t-s)^{\frac{1-\vartheta}{2}}\left(\|g\|_\infty+\nor \nabla g\nor_{\widetilde{\mL}^{p}_q(T)}\right).
\end{align*}
Moreover, for any $0\le u,v\le s\le t\le T$ and $x,y\in\mathbb{R}^d$,
\begin{align*}
&\left(\E\left|\int_s^t [g(r,X^{u,x}_r)-g(r,X^{v,y}_r)]d W_r\right|^m\right)^{\frac1m}\\
&\quad\le C(|x-y|+|u-v|^{\frac12})(t-s)^{\frac{1-\vartheta}{2}}\left(\|g\|_\infty+\nor \nabla g\nor_{\widetilde{\mL}^{p}_q(T)}\right).
\end{align*}
\end{proposition}
\begin{proof}
The first estimate follows from the BDG inequality and the argument leading to \eqref{AB00}, with $h$ replaced by $g$ and $\Phi$ replaced by $\mathcal{M}|\nabla g|$. The second one follows from the first one as in the proof of \cref{cor:23}.
\end{proof}

\subsection{Almost sure regularization by the driftless semiflow}\label{sub:asreg}
The aim of this subsection is to obtain, on a single event of full probability which does not depend on the starting points, H\"older estimates for the maps $(u,x)\mapsto\int_s^t f(r,\psi^{u,x}_r)\dif r$ and $(u,x)\mapsto\int_s^t g(r,\psi^{u,x}_r)\dif\bW_r$ with a factor $(t-s)^\alpha$ (\cref{thm.DavieEst,thm.DavieEstDiff}). Such uniform estimates are needed in \cref{sec:5}, where the starting points are determined by an arbitrary solution path. They are derived from the moment estimates of \cref{sub.DavieEst} by two applications of a Garsia--Rodemich--Rumsey type inequality, one in the starting point, with the parabolic distance $\rho$ below (reflecting the factor $|u-v|^{1/2}$ in \eqref{cor23}), and one in time. For this purpose, the functionals are first extended to a fixed box of starting points. For general $f$, the estimates are then transferred from continuous functions by monotone and Lusin-type approximations.

Throughout this subsection, we assume \Hsigma. Fix $p_*\in(2,2+\eps)$ and let $(q_*,r_*)$ be given by \eqref{n0928:00} with $p=p_*$. Then $\bW(\omega)$ is a $p_*$-rough path for every $\omega$, and $\psi^{s,x}_t$ denotes the solution to the RDE \eqref{S2:RDE}, which is continuous in $(s,t,x)$ (see \cref{sec:4.5}). By \cref{thm:0311}~(ii) (with $b\equiv0$) and \cref{App:thm:00}, for every $(s,x)$, $(\psi^{s,x}_t)_{t\in[s,T]}$ coincides almost surely with the strong solution to \eqref{Pre:SDE} with $b\equiv0$. Hence the estimates of \cref{sub.DavieEst} (with $b\equiv0$) apply to $\psi$. We also fix $(p,q)\in\cJ$, and set
$$
\vartheta:=\frac dp+\frac2q,\qquad \gamma:=\frac{1-\vartheta}2,\qquad K_M:=[-M,M]^d,\qquad O_M:=[0,T]\times K_M,
$$
and
$$
\rho((u,x),(v,y)):=|u-v|^{\frac12}+|x-y|,
$$
for $M\ge1$. Note that $\rho$ is equivalent to the distance $|\cdot|_a$ of \cref{sec:D} on $\mR\times\Rd$ with $n=d+1$, $\ell=(1,\dots,1)$ and $a=(2,1,\dots,1)$, whose homogeneous dimension is $D=d+2$, and that $O_M$ is the closure of a box.

We shall use the following two results, which are proved in \cref{sec:D}. The first one is a Garsia--Rodemich--Rumsey type inequality.
\begin{lemma}\label{lem:GRR}
Let $|\cdot|_a$ be a distance on $\mR^N$ as in \cref{sec:D}, with homogeneous dimension $D$, and let $O$ be the closure of a box. For any $m\ge1$ and $\theta\in(\frac Dm,1)$, there is a constant $C$ depending only on $O,a,m,\theta$ such that for any continuous function $F:O\to\mR$ and $z,z'\in O$,
\begin{align*}
|F(z)-F(z')|\le C|z-z'|_a^{\theta-\frac Dm}\left(\int_{O^2}\frac{|F(z_1)-F(z_2)|^m}{|z_1-z_2|_a^{\theta m+D}}\dif z_1\dif z_2\right)^{\frac1m}.
\end{align*}
\end{lemma}
The second one is an anisotropic version of the Kolmogorov continuity theorem.
\begin{lemma}\label{thm.Kolmogorov}
Let $O$, $|\cdot|_a$ and $D$ be as in \cref{lem:GRR}, and let $(X(z))_{z\in O}$ be a real-valued random field which is jointly measurable in $(z,\omega)$. Suppose that for some $m, \kappa>D$ there is a constant $C$ such that $\mE|X(z)-X(z')|^m\le C|z-z'|_a^{\kappa}$ and $\mE|X(z)|^m\le C$ for all $z,z'\in O$. Then $X$ has a continuous modification $\tilde X$, and for any $\eta\in(0,\frac{\kappa-D}{m}\wedge(1-\frac Dm))$,
\begin{align*}
    \mE\left[\sup_{z,z'\in O; z\ne z'}\frac{|\tilde{X}(z)-\tilde{X}(z')|}{|z-z'|_a^{\eta}}\right]<\infty.
\end{align*}
\end{lemma}

\subsubsection*{The drift part}
For $f\in\widetilde{\mL}^p_q(T)$, $z=(u,x)\in O_M$ and $t\in[0,T]$, we define
\begin{align*}
G_f(t;z):=\int_u^{t\vee u} f(r,\psi^{u,x}_r)\dif r\quad\text{if }\int_u^T|f(r,\psi^{u,x}_r)|\dif r<\infty,\qquad G_f(t;z):=0\quad\text{otherwise}.
\end{align*}
For $0\le u\le s\le t\le T$ and $x\in\Rd$, we set
\begin{align*}
I_s^t(f)(u,x):=\int_s^t f(r,\psi^{u,x}_r)\dif r,
\end{align*}
so that $I_s^t(f)(u,x)=G_f(t;u,x)-G_f(s;u,x)$ whenever $\int_u^T|f(r,\psi^{u,x}_r)|\dif r<\infty$,
and for $s,t\in[0,T]$ and $z,z'\in O_M$,
$$
\Box G_f(s,t;z,z'):=G_f(t;z)-G_f(s;z)-G_f(t;z')+G_f(s;z').
$$
Given $\alpha\in(0,\gamma)$ and $\nu\in(0,1)$, we fix an integer $m=m(\alpha,\nu,d,p,q)$ such that $\frac2m\le\gamma-\alpha$ and $\frac{d+2}{m}\le\frac\nu2$, and define
\begin{align}\label{Xif}
\begin{split}
\Xi^0_{M}(f):=&\left(\int_{O_M^2}\int_{[0,T]^2}\frac{|\Box G_f(t_1,t_2;z_1,z_2)|^m}{|t_1-t_2|^{\gamma m}\rho(z_1,z_2)^{(1-\nu/2)m+d+2}}\dif t_1\dif t_2\dif z_1\dif z_2\right)^{\frac1m}\\
&+\left(\int_{[0,T]^2}\frac{|G_f(t_1;0,0)-G_f(t_2;0,0)|^m}{|t_1-t_2|^{\gamma m}}\dif t_1\dif t_2\right)^{\frac1m}.
\end{split}
\end{align}
The dependence of $\Xi^0_{M}$ on $\alpha,\nu,T$ and $(p,q)$ is suppressed. Note that $G_f$ is linear in $f$ wherever all integrals involved are finite.
\begin{lemma}\label{lem:Ximoment}
There is a constant $C$ depending only on $\Theta_X,c_0,\beta,M,\alpha,\nu$ such that for any $f\in\widetilde{\mL}^p_q(T)$, $s\le t$ and $z,z'\in O_M$,
\begin{align}\label{0315:02}
\begin{split}
\big(\mE|\Box G_f(s,t;z,z')|^m\big)^{\frac1m}&\le C|t-s|^{\gamma}\rho(z,z')\nor f\nor_{\widetilde{\mL}^p_q(T)},\\ \big(\mE|G_f(t;0,0)-G_f(s;0,0)|^m\big)^{\frac1m}&\le C|t-s|^{\gamma}\nor f\nor_{\widetilde{\mL}^p_q(T)}.
\end{split}
\end{align}
Consequently, $(\mE\,\Xi^0_{M}(f)^m)^{\frac1m}\le C\nor f\nor_{\widetilde{\mL}^p_q(T)}$, and for $f,g\in\widetilde{\mL}^p_q(T)$ we have $\Xi^0_{M}(f+g)\le \Xi^0_{M}(f)+\Xi^0_{M}(g)$ almost surely.
\end{lemma}
\begin{proof}
Fix $z=(u_1,x_1)$, $z'=(u_2,x_2)$ and $s\le t$, and assume without loss of generality that $u_1\le u_2$. By \eqref{Kry02}, for each fixed $z$ the integral $\int_u^T|f(r,\psi^{u,x}_r)|\dif r$ is finite almost surely, so that, almost surely,
\begin{align*}
\Box G_f(s,t;z,z')=\int_{[s,t]\cap[u_2,T]}\big[f(r,\psi^{u_1,x_1}_r)-f(r,\psi^{u_2,x_2}_r)\big]\dif r+\int_{[s,t]\cap[u_1,u_2)}f(r,\psi^{u_1,x_1}_r)\dif r.
\end{align*}
The first term is an integral over an interval $[s',t]$ with $s'\ge u_1\vee u_2$, so by \eqref{cor23} its $L^m(\Omega)$-norm is bounded by $C(t-s)^{\gamma}\rho(z,z')\nor f\nor_{\widetilde{\mL}^p_q(T)}$. The second term is an integral over an interval $J\subset[u_1,T]$ of length $|J|\le (t-s)\wedge(u_2-u_1)$, so by \eqref{Kry02} its $L^m(\Omega)$-norm is bounded by
\begin{align*}
C|J|^{1-\frac\vartheta2}\nor f\nor_{\widetilde{\mL}^p_q(T)}=C|J|^{\gamma}|J|^{\frac12}\nor f\nor_{\widetilde{\mL}^p_q(T)}\le C(t-s)^{\gamma}|u_1-u_2|^{\frac12}\nor f\nor_{\widetilde{\mL}^p_q(T)}.
\end{align*}
The second estimate in \eqref{0315:02} follows from \eqref{Kry02} in the same way. The bound on $\mE\,\Xi^0_{M}(f)^m$ follows by Fubini's theorem, since $\int_{O_M^2}\rho(z_1,z_2)^{\nu m/2-d-2}\dif z_1\dif z_2<\infty$. Finally, by Fubini's theorem and \eqref{Kry02}, almost surely the set of $z\in O_M$ for which one of the integrals defining $G_f,G_g,G_{f+g}$ diverges is Lebesgue null, so $G_{f+g}=G_f+G_g$ for almost every $z$, and the subadditivity follows from Minkowski's inequality.
\end{proof}

\begin{theorem}\label{thm.DavieEst}
Assume \Hsigma. Let $(p,q)\in\cJ$, $\alpha\in(0,\gamma)$, $\nu\in(0,1)$ and $M\ge1$. There is a constant $C_*=C_*(T,M,d,p,q,\alpha,\nu)$ such that, with $\Xi_M:=C_*\Xi^0_M$, the following holds. For every $f\in \widetilde{\mL}^p_q(T)$, there is an event $\Omega_{f,M}$ of full probability such that for every $\omega\in \Omega_{f,M}$:
\begin{enumerate}[(i)]
\item $\Xi_M(f)(\omega)<\infty$ and $\int_u^T|f(r,\psi^{u,x}_r(\omega))|\dif r<\infty$ for all $(u,x)\in O_M$;
\item for all $0\le u,v\le s\le t\le T$ and $x,y\in K_M$,
\begin{align}\label{thm21}
    |I_s^t(f)(u,x)(\omega)-I_s^t(f)(v,y)(\omega)|\le \Xi_{M}(f)(\omega)(t-s)^\alpha\rho((u,x),(v,y))^{1-\nu};
\end{align}
\item for all $0\le u\le s\le t\le T$ and $x\in K_M$,
\begin{align}\label{0316:04}
    |I_s^t(f)(u,x)(\omega)|\le \Xi_{M}(f)(\omega)(t-s)^\alpha.
\end{align}
\end{enumerate}
Moreover, $\mE\,\Xi_{M}(f)\le C\nor f\nor_{\widetilde{\mL}^p_q(T)}$ and $\Xi_{M}$ is subadditive almost surely. In particular, applying \eqref{0316:04} to $|f|$, for $\omega\in\Omega_{|f|,M}$,
\begin{align}\label{0317:00}
\sup_{0\le u\le s\le t\le T}\,\sup_{x\in K_M}(t-s)^{-\alpha}\,I^t_s(|f|)(u,x)(\omega)\le \Xi_{M}(|f|)(\omega)<\infty.
\end{align}
\end{theorem}
\begin{proof}
The moment bound and the subadditivity follow from \cref{lem:Ximoment}; in particular, $\Xi_M(f)<\infty$ almost surely, and all events constructed below are tacitly intersected with $\{\Xi_M(f)<\infty\}$. We prove (i)--(iii) in five steps, writing $\Xi:=\Xi_M$ and $\rho:=\rho((u,x),(v,y))$.

{\bf Step 1 ($f$ bounded and continuous).} By the continuity of $\psi$ and dominated convergence, $G_f$ is continuous on $[0,T]\times O_M$ for every $\omega$, and (i) is trivial. Fix $\omega$ and $s\le t$. Applying \cref{lem:GRR} on $O_M$ (with $D=d+2$ and $\theta=1-\frac\nu2$) to $z\mapsto G_f(t;z)-G_f(s;z)$, we obtain for $z,z'\in O_M$
\begin{align*}
|\Box G_f(s,t;z,z')|\le C\rho(z,z')^{1-\frac\nu2-\frac{d+2}m}\left(\int_{O_M^2}\frac{|\Box G_f(s,t;z_1,z_2)|^m}{\rho(z_1,z_2)^{(1-\nu/2)m+d+2}}\dif z_1\dif z_2\right)^{\frac1m}.
\end{align*}
For fixed $z_1,z_2$, applying \cref{lem:GRR} on $[0,T]$ (with $D=1$ and $\theta=\gamma-\frac1m$) to $t\mapsto G_f(t;z_1)-G_f(t;z_2)$ gives
\begin{align*}
|\Box G_f(s,t;z_1,z_2)|^m\le C|t-s|^{\gamma m-2}\int_{[0,T]^2}\frac{|\Box G_f(t_1,t_2;z_1,z_2)|^m}{|t_1-t_2|^{\gamma m}}\dif t_1\dif t_2.
\end{align*}
Combining the two inequalities with the choice of $m$, we get $|\Box G_f(s,t;z,z')|\le C\Xi^0_M(f)(t-s)^{\alpha}\rho(z,z')^{1-\nu}$. For $u,v\le s$, $\Box G_f(s,t;(u,x),(v,y))=I^t_s(f)(u,x)-I^t_s(f)(v,y)$, which gives \eqref{thm21}. Similarly, the one-dimensional inequality gives $|I^t_s(f)(0,0)|\le C\Xi^0_M(f)(t-s)^\alpha$, and \eqref{0316:04} follows from \eqref{thm21} with $(v,y)=(0,0)$, since $\rho$ is bounded on $O_M$.

{\bf Step 2 ($f=\bbb1_U$ with $U\subset[0,T]\times\Rd$ open and bounded).} Let $h_k(t,x):=1\wedge (k\,{\rm dist}((t,x),U^c))$. Then $h_k$ is continuous and $0\le h_k\uparrow\bbb1_U$, so by monotone convergence $I^t_s(h_k)(u,x)\uparrow I^t_s(\bbb1_U)(u,x)$ for every $\omega$, $u\le s\le t$ and $x$. Since $\nor h_k-\bbb1_U\nor_{\widetilde{\mL}^p_q(T)}\to0$, \cref{lem:Ximoment} gives $\mE\,\Xi(h_k-\bbb1_U)\to0$, so there is an event of full probability on which $\Xi(h_{k_j}-\bbb1_U)\to0$ along a subsequence. Since all functions involved are bounded, $\Xi(h_{k})\le\Xi(\bbb1_U)+\Xi(h_{k}-\bbb1_U)$ for every $\omega$. Passing to the limit in \eqref{thm21} and \eqref{0316:04} for $h_{k_j}$ yields them for $\bbb1_U$.

{\bf Step 3 ($f$ bounded, Borel, with bounded support).} By Lusin's theorem, there are continuous functions $f_n$ with $\|f_n\|_\infty\le\|f\|_\infty$ and open sets $U_n$, all contained in a fixed bounded set, such that $\{f\ne f_n\}\subset U_n$ and $|U_n|\le 4^{-n}$. Then $|f-f_n|\le 2\|f\|_\infty\bbb1_{U_n}$, so $\nor f-f_n\nor_{\widetilde{\mL}^p_q(T)}+\nor\bbb1_{U_n}\nor_{\widetilde{\mL}^p_q(T)}\to0$. By \cref{lem:Ximoment}, along a subsequence, $\Xi(f-f_n)\to0$ and $\Xi(\bbb1_{U_n})\to0$ on an event of full probability, on which also Step 2 applies to all $U_n$. On this event, for all $u\le s\le t$ and $x\in K_M$, by \eqref{0316:04} for $\bbb1_{U_n}$,
\begin{align*}
|I^t_s(f)(u,x)-I^t_s(f_n)(u,x)|\le 2\|f\|_\infty I^t_s(\bbb1_{U_n})(u,x)\le 2\|f\|_\infty\Xi(\bbb1_{U_n})(t-s)^\alpha\to0.
\end{align*}
Together with \eqref{thm21} and \eqref{0316:04} for $f_n$ and $\Xi(f_n)\le\Xi(f)+\Xi(f_n-f)$, we obtain \eqref{thm21} and \eqref{0316:04} for $f$.

{\bf Step 4 ($f$ with bounded support and $\|f\|_{L^q([0,T];L^p)}<\infty$).} For $n\in\mN$, the functions $|f|\wedge n$ are bounded with bounded support and $|f|\wedge n\uparrow|f|$. By Step 3 and monotone convergence, as in Step 2, \eqref{0316:04} holds for $|f|$ on an event of full probability; in particular, (i) holds on this event. Let $f_n:=(f\wedge n)\vee(-n)$. On the intersection of this event with the events of Step 3 for all $f_n$ and with an event on which $\Xi(f_{n_j}-f)\to0$ along a subsequence (which exists by \cref{lem:Ximoment}), dominated convergence gives $I^t_s(f_n)(u,x)\to I^t_s(f)(u,x)$ for all $u\le s\le t$ and $x\in K_M$, and \eqref{thm21}, \eqref{0316:04} pass to the limit.

{\bf Step 5 (general $f\in\widetilde{\mL}^p_q(T)$).} For $R\in\mN$, let $f_R:=f\bbb1_{\{|x|\le R\}}$, which satisfies the assumptions of Step 4, and let $\Omega^R:=\{\sup_{(u,x)\in O_M}\sup_{r\in[u,T]}|\psi^{u,x}_r|<R\}$. Since $\psi$ is continuous, $\bigcup_R\Omega^R=\Omega$. On $\Omega^R$, $f(r,\psi^{u,x}_r)=f_R(r,\psi^{u,x}_r)$ for all $(u,x)\in O_M$ and $r\in[u,T]$, so $G_f=G_{f_R}$ on $[0,T]\times O_M$ and $\Xi(f)=\Xi(f_R)$. Hence (i)--(iii) hold on $\Omega_{f,M}:=\bigcup_R(\Omega^R\cap\Omega_{f_R,M})$, which has full probability.
\end{proof}
\begin{remark}\label{rem:general}
The proofs of \cref{lem:Ximoment} and \cref{thm.DavieEst} use only the following two properties of $\psi$: (a) $(u,r,x)\mapsto\psi^{u,x}_r$ is continuous on $\{u\le r\}$ for every $\omega$; (b) for each fixed $(u,x)$, the process $(\psi^{u,x}_r)_{r\in[u,T]}$ satisfies \eqref{Kry02} and \eqref{cor23} (with constants independent of $(u,x)$). Hence \cref{lem:Ximoment} and \cref{thm.DavieEst} hold for any random field $(\phi^{u,x}_r)$ with properties (a) and (b) in place of $\psi$.
\end{remark}

\subsubsection*{The rough integral part}
Let $g:[0,T]\times\Rd\to\cL(\Rd,\mR^n)$ satisfy {\bf (S$^1_\sigma$)} (with $g$ in place of $\sigma$) for the exponents $(p_*,q_*,r_*)$. For $z=(u,x)\in O_M$ and $t\in[0,T]$, we define
\begin{align*}
H_g(t;z):=\int_0^{t\wedge u}g(r,x)\dif\bW_r+\int_u^{t\vee u}g(r,\psi^{u,x}_r)\dif\bW_r,
\end{align*}
where both integrals are rough integrals in the sense of \cref{App:lem:A4}, of the controlled paths $(g(\cdot,x),0)$ and $(g(\cdot,\psi^{u,x}),\nabla g(\cdot,\psi^{u,x})\sigma(\cdot,\psi^{u,x}))$ respectively (see \cref{lem:030800}). That is, $H_g(\cdot;z)$ is the rough integral of $g$ along the path $\bar\psi^{z}$ that is frozen at $x$ before time $u$ and equals $\psi^{u,x}$ afterwards. For $0\le u\le s\le t\le T$ and $x\in\Rd$, we set
\begin{align*}
J_s^t(g)(u,x):=\int_s^t g(r,\psi^{u,x}_r)\dif \bW_r=H_g(t;u,x)-H_g(s;u,x),
\end{align*}
and we define $\Box H_g$ and
\begin{align}\label{Xig}
\begin{split}
\widetilde{\Xi}^0_{M}(g):=&\left(\int_{O_M^2}\int_{[0,T]^2}\frac{|\Box H_g(t_1,t_2;z_1,z_2)|^m}{|t_1-t_2|^{\gamma m}\rho(z_1,z_2)^{(1-\nu/2)m+d+2}}\dif t_1\dif t_2\dif z_1\dif z_2\right)^{\frac1m}\\
&+\left(\int_{[0,T]^2}\frac{|H_g(t_1;0,0)-H_g(t_2;0,0)|^m}{|t_1-t_2|^{\gamma m}}\dif t_1\dif t_2\right)^{\frac1m}
\end{split}
\end{align}
in the same way as for $G_f$, with the same $m$.

\begin{theorem}\label{thm.DavieEstDiff}
Assume \Hsigma. Let $(p,q)\in\cJ$, $\alpha\in(0,\gamma)$, $\nu\in(0,1)$, $M\ge1$, and let $g$ be as above. With $\widetilde\Xi_M:=C_*\widetilde\Xi^0_M$ and the constant $C_*$ of \cref{thm.DavieEst}, the following holds for every $\omega\in\Omega$: for all $0\le u,v\le s\le t\le T$ and $x,y\in K_M$,
\begin{align}\label{thm22}
    \begin{split}
    \left|J^t_s(g)(u,x)-J^t_s(g)(v,y)\right|&\le \widetilde\Xi_{M}(g)(t-s)^\alpha\rho((u,x),(v,y))^{1-\nu},\\ |J^t_s(g)(u,x)|&\le \widetilde\Xi_{M}(g)(t-s)^\alpha.
\end{split}
\end{align}
Moreover, $\widetilde\Xi_M(g_1+g_2)\le\widetilde\Xi_M(g_1)+\widetilde\Xi_M(g_2)$, and there is a constant $C$ depending only on $\Theta_X,M,\alpha,\nu$ such that
\begin{align}\label{AB01}
    \E\,\widetilde\Xi_{M}(g)\le C\|g\|_{L^\infty C^1}.
\end{align}
\end{theorem}
\begin{proof}
By \cref{App:lem:A4}, \cref{thm:0125:01,thm:0125:03}, \cref{App:thm:sta} and the continuity of $\psi$, the map $(t,z)\mapsto H_g(t;z)$ is continuous on $[0,T]\times O_M$ for every $\omega$. (For the continuity in $u$, note that for $u<u'$ the two paths $\bar\psi^{(u,x)}$ and $\bar\psi^{(u',x)}$ differ on $[u,u']$, where both rough integrals are small, and on $[u',T]$, where $\psi^{u,x}=\psi^{u',\psi^{u,x}_{u'}}$ by \eqref{0316:05}.) Hence Step 1 of the proof of \cref{thm.DavieEst} applies verbatim and yields \eqref{thm22} for every $\omega$. The subadditivity is clear, since $H_g$ is linear in $g$.

It remains to show \eqref{AB01}, for which it suffices, as in \cref{lem:Ximoment}, to show that for $s\le t$ and $z=(u_1,x_1)$, $z'=(u_2,x_2)$ with $u_1\le u_2$,
\begin{align}\label{AB011}
\begin{split}
(\mE|\Box H_g(s,t;z,z')|^m)^{\frac1m}&\le C(t-s)^\gamma\rho(z,z')\|g\|_{L^\infty C^1},\\ (\mE|H_g(t;0,0)-H_g(s;0,0)|^m)^{\frac1m}&\le C(t-s)^\gamma\|g\|_{L^\infty C^1}.
\end{split}
\end{align}
For fixed $z$, the rough integral $H_g(t;z)-H_g(s;z)$ coincides almost surely with the It\^o integral $\int_s^t g(r,\bar\psi^z_r)\dif W_r$, by \cref{thm:0311}~(i) on $[u,T]$ (applied to the strong solution $\psi^{u,x}$) and by the same Riemann sum argument for the deterministic integrand $g(\cdot,x)$ on $[0,u]$. We split $[s,t]$ into $[s,t]\cap[0,u_1)$, $[s,t]\cap[u_1,u_2)$ and $[s,t]\cap[u_2,T]$. On the first interval, the integrand is $g(r,x_1)-g(r,x_2)$, and the BDG inequality gives the bound $C(t-s)^{\frac12}\|\nabla g\|_\infty|x_1-x_2|$. On the second interval, the integrand is $g(r,\psi^{u_1,x_1}_r)-g(r,x_2)$, whose absolute value is at most $\|\nabla g\|_\infty(|\psi^{u_1,x_1}_r-x_1|+|x_1-x_2|)$; since $(\mE|\psi^{u_1,x_1}_r-x_1|^m)^{1/m}\le C|r-u_1|^{\frac12}\le C|u_2-u_1|^{\frac12}$, the BDG inequality gives the bound $C(t-s)^{\frac12}\|\nabla g\|_\infty\rho(z,z')$. On the third interval, the bound $C(t-s)^\gamma\rho(z,z')\|g\|_{L^\infty C^1}$ follows from the second estimate in \cref{lem:22}, since $\nor\nabla g\nor_{\widetilde{\mL}^p_q(T)}\le C\|\nabla g\|_\infty$. Since $(t-s)^{\frac12}\le T^{\frac12-\gamma}(t-s)^\gamma$, the first estimate in \eqref{AB011} follows. The second one follows from the BDG inequality.
\end{proof}

\section{Path-by-path existence and uniqueness}\label{sec:5}
In this section, we assume \Hsigma{} and \Hb{} with some $(p,q)\in\cJ$, where $b$ is a fixed Borel function, and we prove our main results, \cref{thm:61,thm:64}. As in \cref{sub:asreg}, we fix $p_*\in(2,2+\eps)$ and $(q_*,r_*)$ given by \eqref{n0928:00}, so that $\bW(\omega)$ is a $p_*$-rough path for every $\omega$ and $\sigma$ satisfies \Ssigmai{} with these exponents (\cref{lem:HS}). We also fix $T>0$, set $\vartheta:=\frac dp+\frac2q$, $\gamma:=\frac{1-\vartheta}2$, and fix $\alpha\in(0,\gamma)$ and $\nu\in(0,\alpha)$. For $s_0\in[0,T)$, a path-by-path solution on $[s_0,T]$ is defined as in \cref{def:413} with $[0,T]$ replaced by $[s_0,T]$.

The proof has three steps. First (\cref{thm:61}), we construct a version $\tilde X$ of the stochastic flow of \eqref{Pre:SDE} which, for every $\omega$ in an event of full probability, is a continuous flow of path-by-path solutions. Second (\cref{lem:AB,prop:42}), for an arbitrary path-by-path solution $Y$, we compare both $Y$ and $\tilde X^{s,Y_s}$ with the driftless flow $\psi^{s,Y_s}$ on $[s,t]$, using \cref{thm.DavieEst,thm.DavieEstDiff}; this yields $|Y_t-\tilde X^{s,Y_s}_t|\le Cw(s,t)^\beta$ with $\beta=1+\alpha-\nu>1$. The condition $\nu<\alpha$ is used exactly here, both in the sewing lemma and to obtain $\beta>1$. Third (\cref{thm:64}), the flow property of $\tilde X$ and a telescoping argument as in \cite{ALL23} show that $Y=\tilde X^{0,Y_0}$.

\subsection{Existence}
\begin{theorem}\label{thm:61}
There are an event $\Omega_X$ of full probability and a map $(s,t,x,\omega)\mapsto\tilde{X}^{s,x}_t(\omega)$, defined for $0\le s\le t\le T$, $x\in\Rd$ and $\omega\in\Omega$, such that the following hold for every $\omega\in\Omega_X$.
\begin{enumerate}[(i)]
\item For every $(s,x)\in[0,T)\times \mR^d$, $(\tilde{X}^{s,x}_t(\omega))_{t\in[s,T]}$ is a path-by-path solution on $[s,T]$ to the RDE
    \begin{align}\label{0306:03}
\tilde{X}^{s,x}_t=x+\int^t_sb(r,\tilde{X}^{s,x}_r)\dif r+\int_s^t \sigma(r,\tilde{X}^{s,x}_r)\dif \bW_r(\omega).
\end{align}
\item (Flow property) $\tilde{X}^{s,x}_t=\tilde{X}^{u,\tilde{X}^{s,x}_u}_t$ for all $0\le s\le u\le t\le T$ and $x\in\Rd$.
\item For any $\kappa\in(0,1)$ and $M\ge1$, there is a constant $C=C(\omega,\kappa,M)$ such that for all $s_i\le t_i$, $i=1,2$, and $|x|\vee|y|\le M$,
\begin{align}\label{0326:03}
    |\tilde{X}^{s_1,x}_{t_1}-\tilde{X}^{s_2,y}_{t_2}|\le C\big(|x-y|+|s_1-s_2|^{\frac12}+|t_1-t_2|^{\frac12}\big)^\kappa.
\end{align}
\end{enumerate}
Moreover, for every $(s,x)$, $(\tilde{X}^{s,x}_t)_{t\in[s,T]}$ coincides almost surely with the strong solution $X^{s,x}$ to \eqref{Pre:SDE}.
\end{theorem}
\begin{proof}
{\bf Step 1 (a continuous version).} For $(s,x)\in[0,T)\times \mR^d$, let $X^{s,x}$ be the strong solution to \eqref{Pre:SDE} (see \cite{XXZZ}). By pathwise uniqueness, for all $0\le s\le u\le T$ and $x\in\Rd$,
 \begin{align}\label{0306:02}
X^{s,x}_t=X^{u,X^{s,x}_u}_t\quad \text{for all } t\in[u,T],\quad \text{a.s.}
\end{align}
Consider the random field $\mathcal{X}(s,t,x):=X^{s,x}_{t\vee s}$ on $[0,T]^2\times\Rd$. Let $s_1\le s_2$, $t_1,t_2\in[0,T]$, $x,y\in\Rd$, and $\tau_i:=t_i\vee s_i$. We write
\begin{align*}
\mathcal{X}(s_1,t_1,x)-\mathcal{X}(s_2,t_2,y)=\big(X^{s_1,x}_{\tau_1}-X^{s_1,x}_{\tau_1\vee s_2}\big)+\big(X^{s_1,x}_{\tau_1\vee s_2}-X^{s_2,y}_{\tau_1\vee s_2}\big)+\big(X^{s_2,y}_{\tau_1\vee s_2}-X^{s_2,y}_{\tau_2}\big).
\end{align*}
The first and the last terms are time increments of $X^{s_1,x}$ and $X^{s_2,y}$ over intervals of length at most $s_2-s_1$ and $|t_1-t_2|$ respectively; by \eqref{Kry02} (applied to $b$) and the BDG inequality, their $L^m(\Omega)$-norms are bounded by $C|s_2-s_1|^{\frac12}$ and $C|t_1-t_2|^{\frac12}$. For the middle term, by \eqref{0306:02} and \cref{lem:SF}, its $L^m(\Omega)$-norm is bounded by $C(\mE|X^{s_1,x}_{s_2}-y|^m)^{1/m}\le C(|x-y|+|s_2-s_1|^{\frac12})$. Hence for every $m\ge1$,
\begin{align*}
\big(\mE|\mathcal{X}(s_1,t_1,x)-\mathcal{X}(s_2,t_2,y)|^m\big)^{\frac1m}\le C\big(|x-y|+|s_1-s_2|^{\frac12}+|t_1-t_2|^{\frac12}\big),
\end{align*}
and $\mE|\mathcal X(s,t,x)|^m$ is bounded on bounded sets. In particular, $\mathcal X$ is continuous in probability, so it has a jointly measurable modification (approximate $\mathcal X$ by its values on finer and finer dyadic grids); we use this modification. For $M\in\mN$, we apply \cref{thm.Kolmogorov} on the box $[0,T]^2\times[-M,M]^d$ with the distance $|s_1-s_2|^{\frac12}+|t_1-t_2|^{\frac12}+|x-y|$, whose homogeneous dimension is $d+4$, and with $m$ arbitrarily large. Since continuous modifications are indistinguishable, we obtain a random field $\tilde{\mathcal X}$ on $[0,T]^2\times\Rd$ which is a modification of $\mathcal X$ and satisfies the H\"older estimate \eqref{0326:03} for all $\kappa\in(0,1)$ and $M\in\mN$, almost surely. Changing $\tilde{\mathcal X}$ on a null set, we may assume that $\tilde{\mathcal X}(\cdot)(\omega)$ is continuous for every $\omega$ (for instance, $\tilde{\mathcal X}(s,t,x)=x$ on the exceptional set). We set $\tilde X^{s,x}_t:=\tilde{\mathcal X}(s,t,x)$ for $s\le t$; then $\tilde X^{s,x}_s=x$ for every $\omega$. Since $t\mapsto\tilde X^{s,x}_t$ and $t\mapsto X^{s,x}_t$ are continuous, $\tilde{X}^{s,x}={X}^{s,x}$ on $[s,T]$ almost surely, for every $(s,x)$.

{\bf Step 2 (flow property and countably many starting points).} Let $\cD:=([0,T)\cap\mQ)\times\mQ^d$, so that every $(s,x)\in[0,T)\times\Rd$ is the limit of points $(s_n,x_n)\in\cD$ with $s_n\le s$ (take $s_n=0$ if $s=0$). For $(s,x)\in\cD$, \eqref{Kry02} gives $\mE\big(\int_s^T|b(r,X^{s,x}_r)|\dif r\big)^2<\infty$, so \cref{thm:0311}~(ii) (applied on $[s,T]$) shows that, almost surely, $\tilde X^{s,x}$ is a path-by-path solution on $[s,T]$. Next, for fixed $s\le u\le t$ and $x$, by \eqref{0306:02}, Step 1 and the substitution argument in the proof of \cref{lem:SF} (using the continuity of $y\mapsto \tilde X^{u,y}_t$),
\begin{align*}
\tilde{X}^{s,x}_t={X}^{s,x}_t={X}^{u,X^{s,x}_u}_t=\tilde{X}^{u,y}_t\big|_{y=\tilde{X}^{s,x}_u}\quad \text{a.s.}
\end{align*}
By continuity of $\tilde X$, the flow property (ii) then holds for all $(s,u,t,x)$ simultaneously, almost surely.

{\bf Step 3 (regularity of the drift along $\tilde X$).} By Step 1, $\tilde X$ satisfies properties (a) and (b) of \cref{rem:general}, with \eqref{Kry02} and \eqref{cor23} applied to the SDE with drift $b$. Hence, for every $M\in\mN$ and $f\in\{b,|b|\}$, the conclusions (i)--(iii) of \cref{thm.DavieEst} hold with $\psi$ replaced by $\tilde X$ on an event $\Omega^X_{f,M}$ of full probability, with random constants $\Xi^X_M(f)$. In particular, for $\omega\in \Omega^X_{|b|,M}$,
\begin{align}\label{0326:00}
\sup_{(u,x)\in O_M}\int_u^T|b(r,\tilde X^{u,x}_r)|\dif r\le \Xi^X_M(|b|)T^\alpha<\infty,
\end{align}
and for $\omega\in\Omega^X_{b,M}$, $0\le u,v\le s\le t\le T$ and $x,y\in K_M$,
\begin{align}\label{0326:01}
\Big|\int_s^t b(r,\tilde X^{u,x}_r)\dif r-\int_s^t b(r,\tilde X^{v,y}_r)\dif r\Big|\le \Xi^X_M(b)(t-s)^\alpha\big(|u-v|^{\frac12}+|x-y|\big)^{1-\nu}.
\end{align}

{\bf Step 4 (all starting points).} Let $\Omega_X$ be the intersection of the events of Steps 1--3 over all $M\in\mN$ and $f\in\{b,|b|\}$, and fix $\omega\in\Omega_X$ and $(s,x)\in[0,T)\times\Rd$. Choose $(s_n,x_n)\in\cD$ with $s_n\le s$ and $(s_n,x_n)\to(s,x)$, and let $M\in\mN$ with $|x_n|\le M-1$ for all $n$. Set $Y^n:=\tilde X^{s_n,x_n}$ and $Y:=\tilde X^{s,x}$, restricted to $[s,T]$. By Step 2, $Y^n$ solves RDE$(\sigma,Z^n)$ on $[s,T]$ with $Z^n_t:=\int_{s}^t b(r,Y^n_r)\dif r$ and $Y^n_s=\tilde X^{s_n,x_n}_s$. By \eqref{0326:00}, $\|Z^n\|_{1,[s,T]}\le \Xi^X_M(|b|)T^\alpha$ uniformly in $n$, so by \eqref{0124:07} (applied on $[s,T]$) the norms $\|\sigma(Y^n)\|_{q_*,[s,T]}+\|R^{Y^n}\|_{r_*,[s,T]}$ are bounded uniformly in $n$. By \eqref{0326:01}, $Z^n_t\to Z_t:=\int_{s}^t b(r,Y_r)\dif r$ uniformly in $t\in[s,T]$, and by \eqref{0326:03}, $Y^n\to Y$ uniformly on $[s,T]$. Consequently, $\sigma(Y^n)\to\sigma(Y)$ and $R^{Y^n}\to R^Y$ uniformly. Fix $q'>q_*$ and $r'>r_*$ such that $(q',r')$ still satisfies \eqref{0124:06} with $p=p_*$; this is possible since the inequalities in \eqref{0124:06} are strict. By the interpolation inequality $\|\Xi\|_{r',[s,T]}^{r'}\le \sup_{u\le v}|\Xi_{u,v}|^{r'-r_*}\|\Xi\|^{r_*}_{r_*,[s,T]}$ (and its analogue for paths), $(Y,\sigma(Y))\in\cV^{q',r'}_W$ and $(Y^n,\sigma(Y^n))\to (Y,\sigma(Y))$ in $\cV^{q',r'}_W([s,T])$. Since \Ssigmai{} also holds with $(q',r')$, \eqref{n0927:01} yields
 \begin{align*}
     \lim_{n\to\infty} \int_{s}^t \sigma(r,Y^n_r)\dif \bW_r=\int_{s}^t \sigma(r,Y_r)\dif \bW_r,\qquad t\in[s,T].
 \end{align*}
Passing to the limit in $Y^n_t=Y^n_s+Z^n_t+\int_s^t\sigma(r,Y^n_r)\dif\bW_r$, we obtain that $Y$ satisfies \eqref{0306:03} on $[s,T]$, with $\int_s^T|b(r,Y_r)|\dif r<\infty$ by \eqref{0326:00}. By \cref{sec2:def}~(ii) (with the exponents $(p_*,q',r')$), $Y$ is a path-by-path solution on $[s,T]$. This completes the proof.
\end{proof}

\subsection{A priori estimates}
For a continuous path $Z$ of finite variation on $[s_0,T]$, we set $w_Z(s,t):=(t-s)+\|Z\|_{1,[s,t]}$, which is a control.
\begin{lemma}\label{lem.apri1}
Let $s_0\in[0,T)$, let $Z:[s_0,T]\to\Rd$ be a continuous path of finite variation with $\|Z\|_{1,[s_0,T]}\le K$, and let $Y$ be a solution to RDE$(\sigma,Z)$ on $[s_0,T]$ (for the input $\bW(\omega)$). Then there is a constant $C=C(T,\Lambda^1_{q_*,r_*}(\sigma),\nor\bW(\omega)\nor_{p_*,[0,T]},K)>0$ such that for any $s_0\le s\le t\le T$
 \begin{align}\label{est.YpsiZ0}
		|Y_t- \psi^{s,Y_s}_t|\le C\|Z\|_{1,[s,t]}
\quad\text{and}\quad
		\left|\int_s^t \sigma(r,Y_r)\dif\bW_r-\int_s^t \sigma(r,\psi_r^{s,Y_s})\dif\bW_r\right|\le C\|Z\|_{1,[s,t]}.
	\end{align}
\end{lemma}
\begin{proof}
On $[s,T]$, $Y$ is the solution to RDE$(\sigma,Z-Z_s)$ starting from $Y_s$, while $\psi^{s,Y_s}$ is the solution to RDE$(\sigma,0)$ with the same initial value. Since $r_*\ge1$, $\|Z\|_{r_*,[s,t]}\le\|Z\|_{1,[s,t]}$, and the first estimate follows from \cref{cor:0125:00} applied on $[s,T]$. The second one follows from the first one, since by \eqref{App:RDE} the difference of the rough integrals equals $Y_t-\psi^{s,Y_s}_t-(Z_t-Z_s)$.
\end{proof}

Next, we show an analogue of \cite[Lemma 4.4]{ALL23}. For $f\in\widetilde{\mL}^p_q(T)$ and $M\in\mN$, let $\Omega_{f,M}$ and $\Xi_M(f)$ be as in \cref{thm.DavieEst}, and let $\widetilde\Xi_M(\sigma)$ be as in \cref{thm.DavieEstDiff} with $g=\sigma$ (which satisfies {\bf (S$^0_\sigma$)} and $\|\sigma\|_{L^\infty C^1}<\infty$ by \cref{lem:HS}).
\begin{lemma}\label{lem:AB}
For every $f\in\widetilde{\mL}^p_q(T)$ and $M\in\mN$, there is an event $\Omega'_{f,M}\subset\Omega_{f,M}$ of full probability with the following property. Let $\omega\in\Omega'_{f,M}$, $s_0\in[0,T)$, let $Z$ be a continuous path of finite variation on $[s_0,T]$, and let $Y$ be a solution to RDE$(\sigma,Z)$ on $[s_0,T]$ (for the input $\bW(\omega)$) such that $\int_{s_0}^T|f(r,Y_r)|\dif r<\infty$ and
\begin{align}\label{AB06}
Y_u\in K_M\quad\text{and}\quad \psi^{s,Y_s}_u\in K_M\qquad\text{for all } s_0\le s\le u\le T.
\end{align}
Then there is a constant $C$, depending only on $T,\alpha,\nu$ and the constant in \eqref{est.YpsiZ0} (with $K=\|Z\|_{1,[s_0,T]}$), such that for all $s_0\le s\le t\le T$,
	\begin{align}\label{AB02}
		\left|\int_s^t [f(r,Y_r)-f(r,\psi^{s,Y_s}_r)]dr\right|\le C \Xi_M(f)w_Z(s,t)^{1- \nu}(t-s)^\alpha.
	\end{align}
Moreover, for every $\omega\in\Omega$ and $Y$ as above (without any condition on $f$),
 \begin{align}\label{AB04}
		\left|\int_s^t [\sigma(r,Y_r)-\sigma(r,\psi^{s,Y_s}_r)]\dif \bW_r\right|\le C \widetilde\Xi_M(\sigma)w_Z(s,t)^{1- \nu}(t-s)^\alpha.
	\end{align}
 In particular, $\left|\int_s^t f(r,Y_r)dr\right|\le C \Xi_M(f)(t-s)^\alpha$ and $\left|\int_s^t \sigma(r,Y_r)\dif\bW_r\right|\le C \widetilde\Xi_M(\sigma)(t-s)^\alpha$.
\end{lemma}
\begin{proof}
Throughout the proof, $\omega$, $s_0$, $Z$ and $Y$ are as in the statement, and $C$ denotes a constant as in the statement, which may change from line to line. Note that $w_Z(s_0,T)\le T+\|Z\|_{1,[s_0,T]}$. We first prove \eqref{AB04}. For $s_0\le s\le t\le T$, let $B_{s,t}:=\int_s^t \sigma(r,\psi^{s,Y_s}_r)\dif\bW_r=J^t_s(\sigma)(s,Y_s)$. By the flow property \eqref{0316:05}, for $s\le u\le t$,
\begin{align*}
    \delta B_{s,u,t}:=B_{s,t}-B_{s,u}-B_{u,t}=J^t_u(\sigma)(u,\psi^{s,Y_s}_u)-J^t_u(\sigma)(u,Y_u).
\end{align*}
By \eqref{thm22}, \eqref{AB06} and \eqref{est.YpsiZ0},
\begin{align*}
    |\delta B_{s,u,t}|\le \widetilde\Xi_{M}(\sigma)(t-u)^\alpha|\psi^{s,Y_s}_u-Y_u|^{1-\nu}\le C\widetilde\Xi_{M}(\sigma)(t-s)^\alpha w_Z(s,t)^{1- \nu}.
\end{align*}
Since $\alpha+1-\nu>1$, the sewing lemma (\cref{sewing}) yields a path $\cB$ with $|\cB_t-\cB_s-B_{s,t}|\le C\widetilde\Xi_{M}(\sigma)w_Z(s,t)^{1-\nu}(t-s)^\alpha$. On the other hand, by \cref{sec2:def}, $(Y,\sigma(Y))$ and $(\psi^{s,Y_s},\sigma(\psi^{s,Y_s}))$ are controlled rough paths with the same initial value and Gubinelli derivative at time $s$, and by \eqref{0124:07}, \eqref{n0927:00} (applied on $[s,t]$) and the proof of \cref{lem:030800}, the quantities $\|R^{\sigma(\psi^{s,Y_s})}\|_{r_*,[s,t]}$ and $\|\sigma(\psi^{s,Y_s})'\|_{q_*,[s,t]}$ are bounded by powers of controls that do not depend on $s$; the same holds for $Y$. Explicitly, by the proof of \cref{lem:030800} on $[s,t]$, $\|R^{\sigma(\psi^{s,Y_s})}\|_{r_*,[s,t]}\le \|\sigma\|_{C^{r_*\text{-var}}L^\infty,[s,t]}+C\|\psi^{s,Y_s}\|_{p_*,[s,t]}^{p_*/r_*}+C\|R^{\psi^{s,Y_s}}\|_{r_*,[s,t]}$, and by \eqref{0124:02}, \eqref{n0927:00} and \eqref{0124:07}, $\|\psi^{s,Y_s}\|_{p_*,[s,t]}+\|R^{\psi^{s,Y_s}}\|_{r_*,[s,t]}\le C\nor\bW\nor_{p_*,[s,t]}$, with $C$ independent of $s$; the derivative $\sigma(\psi^{s,Y_s})'$ is treated in the same way. Hence, by \eqref{0125:01}, $|B_{s,t}-\int_s^t\sigma(r,Y_r)\dif\bW_r|\le \tilde w(s,t)^\theta$ for a control $\tilde w$ and some $\theta>1$. Summing over partitions, $\cB_t-\cB_s=\int_s^t\sigma(r,Y_r)\dif\bW_r$, which gives \eqref{AB04}.

We now prove \eqref{AB02}. Let $A^h_{s,t}:=\int_s^t h(r,\psi^{s,Y_s}_r)\dif r=I^t_s(h)(s,Y_s)$ for a function $h$ for which this is well defined, and let $\mathcal R^h_{s,t}:=\int_s^t h(r,Y_r)\dif r-A^h_{s,t}$. We need to show $|\mathcal R^f_{s,t}|\le C\Xi_M(f)w_Z(s,t)^{1-\nu}(t-s)^\alpha$, and we say that $h$ is good if this holds with $f$ replaced by $h$ for all $\omega$, $s_0$, $Z$ and $Y$ as in the statement on an event $\Omega'_{h,M}$ of full probability.

{\bf Step 1 ($h$ bounded and continuous).} On $\Omega_{h,M}$, exactly as for $B$, \eqref{thm21}, \eqref{AB06} and \eqref{est.YpsiZ0} give $|\delta A^h_{s,u,t}|\le C\Xi_M(h)(t-s)^\alpha w_Z(s,t)^{1-\nu}$, and the sewing lemma yields a path $\cA$ with $|\cA_t-\cA_s-A^h_{s,t}|\le C\Xi_M(h)w_Z(s,t)^{1-\nu}(t-s)^\alpha$. Since $Y$ and $\psi$ take values in $K_M$ and $h$ is uniformly continuous on $[0,T]\times K_M$, \eqref{est.YpsiZ0} yields $|\mathcal R^h_{u,v}|\le (v-u)\,\varpi(Cw_Z(u,v))$ for a modulus of continuity $\varpi$ of $h$. Hence $\sum_{[u,v]\in\cP}|\mathcal R^h_{u,v}|\to0$ as $|\cP|\to0$, which shows $\cA_t-\cA_s=\int_s^t h(r,Y_r)\dif r$. Thus $h$ is good with $\Omega'_{h,M}:=\Omega_{h,M}$.

{\bf Step 2 ($h=\bbb1_U$, $U\subset[0,T]\times\Rd$ open and bounded).} Let $h_k$ be as in Step 2 of the proof of \cref{thm.DavieEst}, and let $\Omega'_{\bbb1_U,M}$ be the intersection of $\Omega_{\bbb1_U,M}$, the events $\Omega'_{h_k,M}$ and an event on which $\Xi_M(h_{k_j}-\bbb1_U)\to0$ along a subsequence. On this event, by monotone convergence, $\int_s^t h_{k_j}(r,Y_r)\dif r\to\int_s^t \bbb1_U(r,Y_r)\dif r$ and $A^{h_{k_j}}_{s,t}\to A^{\bbb1_U}_{s,t}$, while $\Xi_M(h_{k_j})\le \Xi_M(\bbb1_U)+\Xi_M(h_{k_j}-\bbb1_U)$. Hence $\bbb1_U$ is good. Combining this with \eqref{0316:04}, for $\omega\in\Omega'_{\bbb1_U,M}$,
\begin{align}\label{AB07}
\int_s^t \bbb1_U(r,Y_r)\dif r\le A^{\bbb1_U}_{s,t}+C\Xi_M(\bbb1_U)w_Z(s,t)^{1-\nu}(t-s)^\alpha\le C\Xi_M(\bbb1_U)(t-s)^\alpha.
\end{align}
That is, the time that $Y$ spends in $U$ is controlled by $\Xi_M(\bbb1_U)$, uniformly in $Y$.

{\bf Step 3 ($h$ bounded, Borel, with bounded support).} Let $f_n$ and $U_n$ be as in Step 3 of the proof of \cref{thm.DavieEst}, and let $\Omega'_{h,M}$ be the intersection of $\Omega_{h,M}$, the events $\Omega'_{f_n,M}$, $\Omega'_{\bbb1_{U_n},M}$, and an event on which $\Xi_M(f_n-h)\to0$ and $\Xi_M(\bbb1_{U_n})\to0$ along a subsequence. By \eqref{AB07} and \eqref{0316:04},
\begin{align*}
|\mathcal R^{h}_{s,t}-\mathcal R^{f_n}_{s,t}|\le 2\|h\|_\infty\Big(\int_s^t \bbb1_{U_n}(r,Y_r)\dif r+A^{\bbb1_{U_n}}_{s,t}\Big)\le C\|h\|_\infty\Xi_M(\bbb1_{U_n})(t-s)^\alpha,
\end{align*}
which tends to zero along the subsequence. Since $|\mathcal R^{f_n}_{s,t}|\le C(\Xi_M(h)+\Xi_M(f_n-h))w_Z(s,t)^{1-\nu}(t-s)^\alpha$, $h$ is good.

{\bf Step 4 (general $f$).} For $R\in\mN$ with $R\ge M\sqrt d$, let $f_R:=f\bbb1_{\{|x|\le R\}}$ and $f_{R,n}:=(f_R\wedge n)\vee(-n)$, which is bounded, Borel and has bounded support. Let $\Omega^R$ be as in Step 5 of the proof of \cref{thm.DavieEst}, and let $E_R$ be the intersection of $\Omega_{f_R,M}$, the events $\Omega'_{f_{R,n},M}$, $n\in\mN$, and an event on which $\Xi_M(f_{R,n}-f_R)\to0$ along a subsequence. We set $\Omega'_{f,M}:=\Omega_{f,M}\cap\bigcap_{R}\big((\Omega\setminus\Omega^R)\cup E_R\big)$, which has full probability. Given $\omega\in\Omega'_{f,M}$, we choose $R\ge M\sqrt d$ with $\omega\in\Omega^R$. Then $\omega\in E_R$, $\Xi_M(f)=\Xi_M(f_R)$, and since $Y$ and $\psi^{s,Y_s}$ take values in $K_M\subset\{|x|\le R\}$, $\mathcal R^f=\mathcal R^{f_R}$. $\int_{s_0}^T|f(r,Y_r)|\dif r<\infty$ and $\int_s^T|f(r,\psi^{s,Y_s}_r)|\dif r<\infty$ (by \cref{thm.DavieEst}~(i)), dominated convergence gives $\mathcal R^{f_{R,n}}_{s,t}\to\mathcal R^{f_R}_{s,t}$ along the subsequence, and \eqref{AB02} follows from the bound for $f_{R,n}$ and $\Xi_M(f_{R,n})\le\Xi_M(f_R)+\Xi_M(f_{R,n}-f_R)$.

Finally, the last two estimates follow from \eqref{AB02}, \eqref{AB04}, \eqref{0316:04} and \eqref{thm22} with $u=s$.
\end{proof}

\begin{proposition}\label{prop:42}
Let $\Omega_0$ be the intersection of $\Omega_X$ with the events $\Omega'_{b,M}$, $\Omega'_{|b|,M}$ and $\{\widetilde\Xi_M(\sigma)<\infty\}$ over all $M\in\mN$; it has full probability. Let $\omega\in\Omega_0$ and let $(Y_t)_{t\in[0,T]}$ be a path-by-path solution to
    \begin{align}\label{0326:05}
        Y_t=Y_0+\int_0^t b(r,Y_r) dr+\int_0^t \sigma(r,Y_r) d \bW_r(\omega).
    \end{align}
    Then, with the control $w(s,t):=(t-s)+\int_s^t |b(r,Y_r)|\dif r$ and $\beta:=1+\alpha-\nu>1$, there is a constant $C=C(\omega,Y)$ such that
    \begin{align}\label{0326:06}
        |Y_t-\tilde X^{s,Y_s}_t|\le Cw(s,t)^\beta,\quad\forall\,0\le s\le t\le T.
    \end{align}
\end{proposition}
\begin{proof}
Fix $\omega\in\Omega_0$ and $Y$. Since $Y$ is continuous and $\tilde X$, $\psi$ are continuous in all variables, there is $M\in\mN$ such that $Y_u$, $\tilde X^{s,Y_s}_u$, $\psi^{s,Y_s}_u$ and $\psi^{u,\tilde X^{s,Y_s}_u}_r$ lie in $K_M$ for all $0\le s\le u\le r\le T$. For fixed $s$, \cref{lem:AB} applies on $[0,T]$ to $Y$ (with $Z_t=\int_0^tb(r,Y_r)\dif r$) and on $[s,T]$ to $\tilde X^{s,Y_s}$ (with $Z_t=\int_s^tb(r,\tilde X^{s,Y_s}_r)\dif r$, see \cref{thm:61}), for $f\in\{b,|b|\}$; the norms $\|Z\|_1$ are bounded uniformly in $s$ by \eqref{0326:00}. Since $\tilde X^{s,Y_s}_s=Y_s$, both paths are compared with the same $\psi^{s,Y_s}$.

Fix $s\le t$, and let $h:=t-s$, $a:=\int_s^t|b(r,Y_r)|\dif r$ and $I:=\int_s^t|b(r,\tilde X^{s,Y_s}_r)|\dif r$. By \eqref{AB02} with $f=|b|$, applied to both paths,
\begin{align*}
|I-a|\le C h^\alpha\big[(h+a)^{1-\nu}+(h+I)^{1-\nu}\big].
\end{align*}
By Young's inequality, $Ch^\alpha(h+I)^{1-\nu}\le \frac12(h+I)+C h^{\alpha/\nu}$, and since $\alpha/\nu>1$ and $\alpha+1-\nu>1$, we obtain $I\le C(h+a)=Cw(s,t)$. Now \eqref{AB02} with $f=b$ and \eqref{AB04}, applied to both paths, give
\begin{align*}
|Y_t-\tilde X^{s,Y_s}_t|&\le\left|\int_s^t [b(r,Y_r)-b(r,\tilde X^{s,Y_s}_r)]\dif r\right|+\left|\int_s^t [\sigma(r,Y_r)-\sigma(r,\tilde X^{s,Y_s}_r)]\dif\bW_r\right|\\
&\le C h^\alpha\big[(h+a)^{1-\nu}+(h+I)^{1-\nu}\big]\le C w(s,t)^{1+\alpha-\nu}.
\end{align*}
This completes the proof.
\end{proof}

\subsection{Uniqueness}
 \begin{theorem}\label{thm:64}
     Let $\Omega_0$ be as in \cref{prop:42} and $\tilde{X}$ as in \cref{thm:61}. Then, for any $\omega\in\Omega_0$ and any path-by-path solution $(Y_t)_{t\in[0,T]}$ to \eqref{0326:05}, we have
     \begin{align*}
         Y_t=\tilde{X}^{0,Y_0}_t(\omega), \qquad \forall t\in[0,T].
     \end{align*}
 \end{theorem}
 \begin{proof}
Fix $\omega\in\Omega_0$, $Y$ and $t\in(0,T]$, and define $D(u):=\tilde{X}^{u,Y_u}_{t}$ for $u\in[0,t]$. Let $M$ be as in the proof of \cref{prop:42}, $\beta>1$ as in \eqref{0326:06}, fix $\kappa\in(1/\beta,1)$, and apply \eqref{0326:03} with $M$ replaced by $M\sqrt d$. For $0\le v\le u\le t$, the flow property in \cref{thm:61} gives
 \begin{align*}
     D(u)-D(v)=\tilde{X}^{u,Y_u}_{t}-\tilde{X}^{u,\tilde{X}^{v,Y_v}_u}_{t},
 \end{align*}
so by \eqref{0326:03} (with $s_1=s_2=u$) and \eqref{0326:06},
 \begin{align*}
     |D(u)-D(v)|\le C|Y_u-\tilde{X}^{v,Y_v}_u|^\kappa\le C w(v,u)^{\kappa\beta}.
 \end{align*}
For the partitions $t_k:=kt2^{-n}$, $k=0,\dots,2^n$, since $\kappa\beta>1$ and $w$ is continuous and super-additive,
 \begin{align*}
   |D(t)-D(0)|\le \sum_{k=1}^{2^n}|D(t_k)-D(t_{k-1})|\le C \max_{k}w(t_{k-1},t_k)^{\kappa\beta-1}\,w(0,t)\to0\quad\text{as }n\to\infty.
 \end{align*}
Hence $Y_t=\tilde X^{t,Y_t}_t=D(t)=D(0)=\tilde{X}^{0,Y_0}_t$.
 \end{proof}
\cref{thm:61,thm:64} together show that, for every $\omega\in\Omega_0$ and every $x\in\Rd$, \eqref{0326:05} has exactly one path-by-path solution on $[0,T]$ with $Y_0=x$, namely $\tilde X^{0,x}(\omega)$.

\begin{appendix}
\renewcommand{\thetable}{A\arabic{table}}
\numberwithin{equation}{section}

\section{Proofs for rough differential equations}\label{S:App}

\subsection{Sewing lemma and rough integrals}\label{appsec:A}

The following generalized sewing lemma is from \cite[Theorem 2.5]{FZ18}.

\begin{lemma}[Generalized sewing lemma]\label{sewing}
Given a Banach space $E$, assume $\Xi=\Xi_{s,t}$ is a continuous function from $\Delta([0,T])$ to $E$, and let $\delta\Xi_{s,u,t}:=\Xi_{s,t}-\Xi_{s,u}-\Xi_{u,t}$ for $0\le s\le u\le t\le T$. Suppose that for some $N\in\mN$, controls $w_{1,j},w_{2,j}$ and exponents $\alpha_{1,j},\alpha_{2,j}\ge0$ with $\theta:=\min_{j}(\alpha_{1,j}+\alpha_{2,j})>1$,
\begin{align*}
    \|\delta\Xi_{s,u,t}\|_E\le \sum_{j=1}^Nw_{1,j}(s,t)^{\alpha_{1,j}}w_{2,j}(s,t)^{\alpha_{2,j}}\quad\text{for all }0\le s\le u\le t\le T.
\end{align*}
Then there is a unique function $(I\Xi_t)_{t\in[0,T]}$ with $I\Xi_0=0$ such that for any $(s,t)\in \Delta([0,T])$
\begin{align*}
I\Xi_{s,t}:=I\Xi_t-I\Xi_s=\lim_{|\cP|\to 0}\sum_{[u,v]\in \cP}\Xi_{u,v},
\end{align*}
where $\cP$ is any partition of $[s,t]$ and $|\cP|:=\sup_{[u,v]\in \cP}|u-v|$.
Moreover, there is a constant $C=C(\theta)>0$ such that
\begin{align}\label{0123:00}
\|I\Xi_{s,t}-\Xi_{s,t}\|_E\le C \sum_{j=1}^Nw_{1,j}(s,t)^{\alpha_{1,j}}w_{2,j}(s,t)^{\alpha_{2,j}}.
\end{align}
\end{lemma}

\begin{proof}[Proof of \cref{App:lem:A4}]
Define $\Xi_{s,t}:=Y_sX_{s,t}+Y_s'\mX_{s,t}$. By Chen's relation,
$$
\delta\Xi_{s,u,t}=-R^Y_{s,u}X_{u,t}-Y'_{s,u}\mX_{u,t},
$$
so that
$|\delta\Xi_{s,u,t}|\le \|R^Y\|_{r,[s,t]}\|X\|_{p,[s,t]}+\|Y'\|_{q,[s,t]}\|\mX\|_{p/2,[s,t]}$.
By \eqref{0124:06} we have $\frac1r+\frac1p>1$ and $\frac1q+\frac2p>1$, and \cref{sewing} yields the result.
\end{proof}

\subsection{Controlled rough paths}\label{appsec:B}

Throughout this subsection, $p\in[2,3)$, $\bX=(X,\mX)$ is a $p$-rough path and $q,r$ satisfy \eqref{0124:06} together with $q\ge p$ and $r\ge \frac p2$. In view of the remarks after \cref{App:def:A2} and \cref{asm.sigma}, this is no loss of generality for the proofs of the results in \cref{sec:2}. In the proofs below, lower order H\"older norms of $\sigma$ are bounded by $\Lambda^k_{q,r}(\sigma)$; strictly speaking this holds only up to a factor depending on $p,q,r$ (see the remark after \cref{asm.sigma}), and the explicit numerical constants below are to be read up to such a factor.

We shall frequently use the following elementary facts. For $1\le p_1\le p_2$, $\|\Xi\|_{p_2,[0,T]}\le \|\Xi\|_{p_1,[0,T]}$. By H\"older's inequality, for any $\Xi^i=\Xi^i_{s,t}$, $i=1,2$, and $p,p_1,p_2\ge1$ with $\frac1p\le \frac1{p_1}+\frac1{p_2}$,
\begin{align}\label{Holder}
    \|\Xi^1\Xi^2\|_{p,[0,T]}\le \|\Xi^1\|_{p_1,[0,T]}\|\Xi^2\|_{p_2,[0,T]}.
\end{align}

For $(Y,Y'),(\tilde{Y},\tilde{Y}')\in \cV^{q,r}_X$ we use the notations
\begin{align*}
    \Delta Y:=Y-\tilde{Y},\quad \Delta Y':=Y'-\tilde{Y}',\quad \Delta R^Y:=R^Y-R^{\tilde{Y}},
\end{align*}
and
\begin{align*}
&M_Y:=|Y'_0|+\|Y'\|_{q,[0,T]},\quad \Delta M_Y:=|\Delta Y'_0|+\|\Delta Y'\|_{q,[0,T]};\\
    &N_Y:=M_Y+\|R^Y\|_{r,[0,T]},\quad \Delta N_Y:=\Delta M_Y+\|\Delta R^Y\|_{r,[0,T]}.
\end{align*}
Note that $\|Y'\|_\infty\le M_Y$. For $k=0,1$, we use $\Theta_k$ to denote the set of parameters on which a constant may depend:
$$
\Theta_k:=(p,q,r,N_{Y},N_{\tilde{Y}},\Lambda_{q,r}^k(\sigma),\nor \bX\nor_{p,[0,T]}).
$$
All constants $C(\Theta_k)$ below can be chosen nondecreasing in $N_Y$, $N_{\tilde Y}$, $\Lambda^k_{q,r}(\sigma)$ and $\nor\bX\nor_{p,[0,T]}$.
Since $Y_{s,t}=Y'_sX_{s,t}+R^Y_{s,t}$ and $r\le p$, we have for any $0\le s\le t\le T$
\begin{align}\label{0124:02}
\|Y\|_{p,[s,t]}\le M_Y\|X\|_{p,[s,t]}+\|R^Y\|_{r,[s,t]},\qquad \|\Delta Y\|_{p,[s,t]}\le \Delta M_Y\|X\|_{p,[s,t]}+\|\Delta R^Y\|_{r,[s,t]}.
\end{align}

We now give stability results for rough integrals (see \cite[Theorem 4.17]{FH} and \cite[Lemma 3.4]{FZ18}).
\begin{theorem}[Stability of rough integration]\label{thm:0125:01}
Suppose $(Y,Y'),(\tilde{Y},\tilde{Y}')\in \cV_X^{q,r}$ take values in $E=\cL(\Rd,\mR^n)$. Then
\begin{align*}
(I_\bX (Y),Y):=\left(\int_0^\cdot Y\dif \bX, Y\right)\in \cV_X^{p,p/2}\subset \cV_X^{q,r},
\end{align*}
and analogous conclusion holds for $(I_\bX (\tilde{Y}),\tilde{Y})$.
Moreover, there is a constant $C=C(p,q,r)>0$ such that
\begin{align*}
\|R^{I_\bX(Y)}-R^{I_\bX (\tilde{Y})}\|_{p/2,[0,T]}\le C\left( \Delta M_Y\|\mX\|_{p/2,[0,T]}+\|X\|_{p,[0,T]}\|\Delta R^Y\|_{r,[0,T]}\right).
\end{align*}
\end{theorem}
\begin{proof}
We have
\begin{align*}
R^{I_\bX(Y)}_{s,t}-R^{I_\bX(\tilde{Y})}_{s,t}=\int_s^t \Delta Y_r\dif \bX_r-\Delta Y_sX_{s,t}.
\end{align*}
Since $(\Delta Y,\Delta Y')\in\cV^{q,r}_X$ and $|\Delta Y'_s|\le \Delta M_Y$, it follows from \eqref{0125:01} that
\begin{align*}
|R^{I_\bX(Y)}_{s,t}-R^{I_\bX(\tilde{Y})}_{s,t}|\le \Delta M_Y|\mX_{s,t}|+C\left(\|\Delta R^Y\|_{r,[s,t]}\|X\|_{p,[s,t]}+\|\Delta Y'\|_{q,[s,t]}\|\mX\|_{p/2,[s,t]}\right).
\end{align*}
Since $\frac p{2r}+\frac12\ge1$ (because $r\le p$) and $\frac{p}{2q}+1\ge 1$, H\"older's inequality for sums and the super-additivity of $(s,t)\mapsto \|\Delta R^Y\|^r_{r,[s,t]}$, $\|X\|^p_{p,[s,t]}$, $\|\Delta Y'\|^q_{q,[s,t]}$, $\|\mX\|^{p/2}_{p/2,[s,t]}$ imply the stated bound. Taking $\tilde Y=\tilde Y'=0$, we see that $R^{I_\bX(Y)}$ has finite $p/2$-variation; $Y$ has finite $p$-variation by \eqref{0124:02}. This completes the proof.
\end{proof}

Let $\sigma=\sigma(t,x)$ satisfy {\bf(S$_\sigma^0$)} and recall $\sigma(Y)_t=\sigma(t,Y_t)$, $\sigma(Y)'_t=\nabla \sigma(t,Y_t)Y'_t$. In the following proofs, for a function $f$ and any $t$, we abbreviate $\|f(t)\|_{C^\alpha}:=\|f(t,\cdot)\|_{C^\alpha}$ and $\|f\|_{\infty}:=\|f\|_{L^\infty([0,T]\times\Rd)}$.
The following is the composition of controlled rough paths with regular functions (see \cite[Lemma 7.3]{FH} and \cite[Lemma 3.5]{FZ18}).
\begin{lemma}\label{lem:030800}
Assume {\bf(S$_\sigma^0$)} holds and $(Y,Y')\in \cV_X^{q,r}$ takes values in $\Rd$. Then $(\sigma(Y),\sigma(Y)')\in \cV_X^{q,r}$ and
\begin{align}\label{0125:00}
N_{\sigma(Y)}=|\sigma(Y)'_0|+\|\sigma(Y)'\|_{q,[0,T]}+\|R^{\sigma(Y)}\|_{r,[0,T]}\le 7 \Lambda_{q,r}^0(\sigma) (1+N_Y)^2(1+\|X\|_{p,[0,T]})^2.
\end{align}
\end{lemma}
\begin{proof}
Write $\Lambda:=\Lambda^0_{q,r}(\sigma)$. Since $|\sigma(t,Y_t)-\sigma(s,Y_s)|\le \|\sigma(t)-\sigma(s)\|_{C^0}+\|\nabla\sigma\|_\infty|Y_{s,t}|$ and $r\le p$, $\sigma(Y)$ has finite $p$-variation. Next,
\begin{align*}
|\sigma(Y)'_t-\sigma(Y)'_s|&\le \|\nabla\sigma(t)-\nabla\sigma(s)\|_\infty\|Y'\|_\infty+|\nabla\sigma(s,Y_t)-\nabla\sigma(s,Y_s)|\|Y'\|_\infty\\
&\quad+\|\nabla \sigma\|_\infty |Y'_{s,t}|\\
&\le \|\sigma(t)-\sigma(s)\|_{C^1}M_Y+\|\sigma\|_{L^\infty C^{1+p/q}}|Y_{s,t}|^{p/q}M_Y+\|\sigma\|_{L^\infty C^1} |Y'_{s,t}|,
\end{align*}
where we used $p/q\le 1$. Hence, by \eqref{0124:02},
\begin{align*}
        \|\sigma(Y)'\|_{q,[0,T]}&\le \Lambda M_Y\big(1+\|Y\|_{p,[0,T]}^{p/q}\big)+\Lambda\|Y'\|_{q,[0,T]}\\
    &\le 3\Lambda N_Y(1+N_Y)(1+\|X\|_{p,[0,T]}).
\end{align*}
Next, we have the following representation
\begin{align*}
R^{\sigma(Y)}_{s,t}&=\sigma(t,Y_t)-\sigma(s,Y_s)-\nabla\sigma(s,Y_s)Y'_sX_{s,t}\\
&=\left(\sigma(t,Y_t)-\sigma(s,Y_t)\right)+\left(\sigma(s,Y_t)-\sigma(s,Y_s)-\nabla\sigma(s,Y_s)Y_{s,t}\right)\\
&\quad+\nabla\sigma(s,Y_s)R^Y_{s,t}\\
&=:I^1_{s,t}+I^2_{s,t}+I^3_{s,t}.
\end{align*}
Clearly $\|I^1\|_{r,[0,T]}\le \|\sigma\|_{C^{r-var}L^\infty }\le\Lambda$ and $\|I^3\|_{r,[0,T]}\le \|\nabla\sigma\|_{\infty}\|R^Y\|_{r,[0,T]}\le \Lambda N_Y$. For $I^2$, since $1<p/r\le 2$, by the fundamental theorem of calculus,
\begin{align*}
    |I^2_{s,t}|&=\left|Y_{s,t}\cdot \int_0^1(\nabla \sigma(s,Y_s+\theta Y_{s,t})-\nabla \sigma(s,Y_s))\dif \theta \right|\le \|\sigma\|_{L^\infty C^{p/r}}|Y_{s,t}|^{p/r},
\end{align*}
which by \eqref{0124:02} implies that
\begin{align*}
    \|I^2\|_{r,[0,T]}&\le\Lambda\|Y\|_{p,[0,T]}^{p/r}\le \Lambda(1+N_Y)^2(1+\|X\|_{p,[0,T]})^2.
\end{align*}
Combining these bounds with $|\sigma(Y)'_0|\le \Lambda N_Y$ gives \eqref{0125:00}.
\end{proof}

Following the same calculation as in the previous proof, we have the following stability estimate (see \cite[Theorem 7.6]{FH} and \cite[Lemma 3.5]{FZ18}).
\begin{theorem}[Stability of composition]\label{thm:0125:03}
Assume \Ssigmai{} holds and $(Y,Y'), (\tilde{Y},\tilde{Y}')\in \cV_X^{q,r}$ take values in $\Rd$. Then there is a constant $C=C(\Theta_1)>0$ such that
\begin{align*}
\|\sigma(Y)'-\sigma(\tilde{Y})'\|_{q,[0,T]}+\|R^{\sigma(Y)}-R^{\sigma(\tilde{Y})}\|_{r,[0,T]}\le C\left(|\Delta Y_0|+\Delta N_Y\right).
\end{align*}
\end{theorem}
\begin{proof}
Write $\Lambda:=\Lambda^1_{q,r}(\sigma)$ and $\mathfrak{D}:=|\Delta Y_0|+\Delta N_Y$. By \eqref{0124:02},
\begin{align}\label{n0927:02}
\|\Delta Y\|_{p,[0,T]}+\|\Delta Y\|_\infty\le |\Delta Y_0|+2\|\Delta Y\|_{p,[0,T]}\le 2(1+\|X\|_{p,[0,T]})\mathfrak{D},
\end{align}
and $\|Y\|_{p,[0,T]}+\|\tilde Y\|_{p,[0,T]}\le (N_Y+N_{\tilde Y})(1+\|X\|_{p,[0,T]})$.
For a function $g$ on $\Rd$ we write $\Delta g(Y)_t:=g(Y_t)-g(\tilde{Y}_t)$.

{\it Gubinelli derivative.} We have
\begin{align*}
    \sigma(Y)'_t-\sigma(\tilde{Y})'_t=\Delta(\nabla\sigma(t))(Y)_tY'_t+\nabla\sigma(t,\tilde{Y}_t)\Delta Y'_t=:I^1_{t}+I^2_{t}.
\end{align*}
Exactly as in the proof of \cref{lem:030800} (with $\tilde Y$ in place of $Y$ and $\Delta Y'$ in place of $Y'$),
\begin{align*}
    \|I^2\|_{q,[0,T]}\le 3\Lambda(1+N_{\tilde Y})(1+\|X\|_{p,[0,T]})\Delta M_Y.
\end{align*}
For $I^1$, we write
\begin{align*}
    I^1_{t}-I^1_s
        &=\Delta(\nabla\sigma(t)-\nabla\sigma(s))(Y)_tY'_t+\left[\Delta(\nabla\sigma(s))(Y)_t-\Delta(\nabla\sigma(s))(Y)_s\right]Y'_t\\
    &\quad+\Delta(\nabla\sigma(s))(Y)_sY'_{s,t}\\
    &=:I^{1,1}_{s,t}+I^{1,2}_{s,t}+I^{1,3}_{s,t}.
\end{align*}
By the mean value theorem, $|I^{1,1}_{s,t}|\le \|\sigma(t)-\sigma(s)\|_{C^2}\|\Delta Y\|_\infty M_Y$ and $|I^{1,3}_{s,t}|\le\|\sigma\|_{L^\infty C^2}\|\Delta Y\|_\infty|Y'_{s,t}|$, so that
\begin{align*}
\|I^{1,1}\|_{q,[0,T]}+\|I^{1,3}\|_{q,[0,T]}\le 2\Lambda M_Y\|\Delta Y\|_\infty.
\end{align*}
For $I^{1,2}$, we omit $s$ in $\sigma(s)$ and write $\Delta(\nabla\sigma)(Y)_{u}=\Sigma_u\Delta Y_u$ with $\Sigma_u:=\int_0^1 \nabla^2\sigma (\theta Y_u+(1-\theta)\tilde{Y}_u) \dif \theta$. Since $\sigma\in L^\infty C^{2+p/q}$,
\begin{align*}
|\Sigma_t|\le \Lambda,\qquad |\Sigma_t-\Sigma_s|\le \Lambda(|Y_{s,t}|+|\tilde Y_{s,t}|)^{p/q},
\end{align*}
hence $|I^{1,2}_{s,t}|\le \Lambda M_Y\big(|\Delta Y_{s,t}|+\|\Delta Y\|_\infty(|Y_{s,t}|+|\tilde Y_{s,t}|)^{p/q}\big)$ and, since $q\ge p$,
\begin{align*}
    \|I^{1,2}\|_{q,[0,T]}\le \Lambda M_Y\big(\|\Delta Y\|_{p,[0,T]}+\|\Delta Y\|_\infty(\|Y\|_{p,[0,T]}+\|\tilde Y\|_{p,[0,T]})^{p/q}\big).
\end{align*}

{\it Remainder.} We have
\begin{align*}
        &R^{\sigma(Y)}_{s,t}-R^{\sigma(\tilde{Y})}_{s,t}\\
    &= \Delta (\sigma(t)-\sigma(s))(Y)_t\\
        &\quad+\big[\Delta(\sigma(s))(Y)_{t}-\Delta(\sigma(s))(Y)_{s}\\
    &\qquad-\nabla\sigma(s,Y_s)Y_{s,t}+\nabla\sigma(s,\tilde{Y}_s)\tilde{Y}_{s,t}\big]\\
    &\quad+\big[\nabla\sigma(s,Y_s)R^Y_{s,t}-\nabla\sigma(s,\tilde{Y}_s)R^{\tilde{Y}}_{s,t}\big]
    =:J^1_{s,t}+J^{2}_{s,t}+J^{3}_{s,t}.
\end{align*}
Clearly,
\begin{align*}
\|J^1\|_{r,[0,T]}\le \|\sigma\|_{C^{r-var}C^1}\|\Delta Y\|_\infty\le\Lambda\|\Delta Y\|_\infty.
\end{align*} Writing
$J^3_{s,t}=\nabla\sigma(s,Y_s)\Delta R^Y_{s,t}+\Delta(\nabla\sigma(s))(Y)_sR^{\tilde Y}_{s,t}$, we also have
\begin{align*}
    \|J^{3}\|_{r,[0,T]}\le  \Lambda(\|\Delta R^Y\|_{r,[0,T]}+\|R^{\tilde Y}\|_{r,[0,T]}\|\Delta Y\|_\infty).
\end{align*}
For $J^2$, again omitting $s$ in $\sigma(s)$, we set
\begin{align*}
    K_{s,t}(Y):=\int_0^1 [\nabla \sigma(Y_s+\theta Y_{s,t})-\nabla \sigma(Y_s)]\dif \theta,\qquad J^2_{s,t}=Y_{s,t}K_{s,t}(Y)-\tilde{Y}_{s,t}K_{s,t}(\tilde{Y}).
\end{align*}
Since $\sigma\in L^\infty C^{p/r}\cap L^\infty C^{1+p/r}$ and $0<p/r-1\le 1$, we have $|K_{s,t}(Y)|\le \Lambda|Y_{s,t}|^{\frac pr-1}$ and, by inserting $\nabla\sigma(\tilde Y_s+\theta Y_{s,t})-\nabla\sigma(\tilde Y_s)$,
\begin{align*}
   |K_{s,t}(Y)-K_{s,t}(\tilde{Y})|\le \Lambda\left(|\Delta Y_{s,t}|+|\Delta Y_s||Y_{s,t}|^{\frac{p}{r}-1}\right).
\end{align*}
Therefore,
\begin{align*}
    |J^{2}_{s,t}|&\le |\Delta Y_{s,t}||K_{s,t}(Y)|+|\tilde{Y}_{s,t}||K_{s,t}(Y)-K_{s,t}(\tilde{Y})|\\
    &\le \Lambda\left(|\Delta Y_{s,t}||Y_{s,t}|^{\frac{p}{r}-1}+|\tilde{Y}_{s,t}||\Delta Y_{s,t}|+\|\Delta Y\|_\infty|\tilde{Y}_{s,t}||Y_{s,t}|^{\frac{p}{r}-1}\right).
\end{align*}
Since $\frac1r=\frac1p+\frac{p/r-1}{p}$ and $\frac1r\le\frac2p$, \eqref{Holder} implies
\begin{align*}
    \|J^{2}\|_{r,[0,T]}\le 3\Lambda\left(1+\|Y\|_{p,[0,T]}+\|\tilde Y\|_{p,[0,T]}\right)^2\left(\|\Delta Y\|_{p,[0,T]}+\|\Delta Y\|_\infty\right).
\end{align*}
Collecting the above estimates and using \eqref{n0927:02}, we complete the proof.
\end{proof}

Let $(Y,Y'), (\tilde{Y}, \tilde{Y}')\in \cV^{q,r}_X$ take values in $\Rd$ and define
\begin{align*}
    (Z,Z')_t:=\left(\int_0^t \sigma(s,Y_s)\dif \bX_s,\sigma(t,Y_t)\right),\quad \text{and similarly } (\tilde{Z},\tilde{Z}').
\end{align*}
As a combination of \cref{thm:0125:01,lem:030800,thm:0125:03}, we have the following estimates for $(Z,Z')$ (see \cite[Lemmas 3.6, 3.7]{FZ18}).

\begin{lemma}\label{lem:A5}
  Assume \Ssigmao{} holds. Then for $\sigma(Y_0)_t:=\sigma(t,Y_0)$, there is a constant $C=C(\Theta_0)>0$ such that
  \begin{align}\label{n0927:00}
  \begin{split}
          &\|Z'-\sigma(Y_0)\|_{q,[0,T]}\le 2\Lambda^0_{q,r}(\sigma)
          (M_{Y}\|X\|_{p,[0,T]}+\|R^Y\|_{r,[0,T]}),\qquad
      \|R^Z\|_{r,[0,T]}\le C
      \nor\bX\nor_{p,[0,T]}.
      \end{split}
  \end{align}
  Assume \Ssigmai{} holds. Then there is a constant $C=C(\Theta_1)>0$ such that
  \begin{align}\label{n0927:01}
    \begin{split}
      &\|Z'-\tilde{Z}'\|_{q,[0,T]}\le C
      (|\Delta Y_0|+\Delta M_{Y}\|X\|_{p,[0,T]}+\|\Delta R^Y\|_{r,[0,T]}),\\
      &\|R^Z-R^{\tilde{Z}}\|_{r,[0,T]}\le C
      \nor \bX\nor_{p,[0,T]}(|\Delta Y_0|+\Delta N_{Y}).
       \end{split}
  \end{align}
\end{lemma}
\begin{proof}
Write $\Lambda_k:=\Lambda^k_{q,r}(\sigma)$.
Since
\begin{align*}
&(\sigma(t,Y_t)-\sigma(t,Y_0))-(\sigma(s,Y_s)-\sigma(s,Y_0))\\
&\quad=\big[(\sigma(t)-\sigma(s))(Y_t)-(\sigma(t)-\sigma(s))(Y_0)\big]+\big[\sigma(s,Y_t)-\sigma(s,Y_s)\big],
\end{align*}
we have, using $\|Y-Y_0\|_\infty\le\|Y\|_{p,[0,T]}$, $q\ge p$ and \eqref{0124:02},
\begin{align*}
        \|Z'-\sigma(Y_0)\|_{q,[0,T]}&\le \|\sigma\|_{C^{q-var}C^1}\|Y-Y_0\|_\infty+\|\nabla\sigma\|_\infty\|Y\|_{q,[0,T]}\\
    &\le 2\Lambda_0(M_Y\|X\|_{p,[0,T]}+\|R^Y\|_{r,[0,T]}).
\end{align*}
By \cref{thm:0125:01} with $\tilde{Y}=\tilde{Y}'=0$, $r\ge p/2$ and \eqref{0125:00},
\begin{align*}
        \|R^Z\|_{r,[0,T]}&\le\|R^{I_\bX(\sigma(Y))}\|_{p/2,[0,T]}\le C N_{\sigma(Y)}\nor\bX\nor_{p,[0,T]}\\
    &\le C\Lambda_0(1+N_Y)^2(1+\|X\|_{p,[0,T]})^2\nor\bX\nor_{p,[0,T]}.
\end{align*}
This gives \eqref{n0927:00}. For \eqref{n0927:01}, with $g:=\sigma(s)$ we write
\begin{align*}
\Delta g(Y)_t-\Delta g(Y)_s&=\int_0^1\nabla g(\tilde Y_t+\lambda\Delta Y_t)\dif\lambda\,\Delta Y_{s,t}\\
&\quad+\int_0^1\big[\nabla g(\tilde Y_t+\lambda\Delta Y_t)-\nabla g(\tilde Y_s+\lambda\Delta Y_s)\big]\dif\lambda\,\Delta Y_s,
\end{align*}
so that, together with $|\Delta(\sigma(t)-\sigma(s))(Y)_t|\le \|\sigma(t)-\sigma(s)\|_{C^1}\|\Delta Y\|_\infty$,
\begin{align*}
        \|Z'-\tilde{Z}'\|_{q,[0,T]}&=\|\sigma(Y)-\sigma(\tilde{Y})\|_{q,[0,T]}\\
    &\le \Lambda_1\left(\|\Delta Y\|_{p,[0,T]}+\|\Delta Y\|_\infty(1+\|Y\|_{p,[0,T]}+\|\tilde{Y}\|_{p,[0,T]})\right).
\end{align*}
By \eqref{0124:02}, $\|\Delta Y\|_\infty\le |\Delta Y_0|+\|\Delta Y\|_{p,[0,T]}\le |\Delta Y_0|+\Delta M_Y\|X\|_{p,[0,T]}+\|\Delta R^Y\|_{r,[0,T]}$, which gives the first estimate in \eqref{n0927:01}. For the second one, by \cref{thm:0125:01},
\begin{align}\label{n0927:05}
\begin{split}
    \|R^Z-R^{\tilde{Z}}\|_{r,[0,T]}&\le \|R^{I_\bX(\sigma(Y))}-R^{I_\bX(\sigma(\tilde{Y}))}\|_{p/2,[0,T]}\\
    &\le C\left(\Delta M_{\sigma(Y)}\|\mX\|_{p/2,[0,T]}+\|X\|_{p,[0,T]}\|\Delta R^{\sigma(Y)}\|_{r,[0,T]}\right),
\end{split}
\end{align}
where $\Delta M_{\sigma(Y)}=|\nabla\sigma(0,Y_0)Y_0'-\nabla\sigma(0,\tilde{Y}_0)\tilde{Y}_0'|+\|\sigma(Y)'-\sigma(\tilde{Y})'\|_{q,[0,T]}$ and $\Delta R^{\sigma(Y)}:=R^{\sigma(Y)}-R^{\sigma(\tilde Y)}$. Since $|\nabla\sigma(0,Y_0)Y_0'-\nabla\sigma(0,\tilde{Y}_0)\tilde{Y}_0'|\le\Lambda_1(M_Y|\Delta Y_0|+\Delta M_Y)$, the second estimate in \eqref{n0927:01} follows from \cref{thm:0125:03}.
\end{proof}

\subsection{Proofs of the results in Section \ref{sec:2}}\label{sec:A.3}
Throughout this subsection, we assume without loss of generality that $q\ge p$ and $r\ge p/2$ (see the remarks after \cref{App:def:A2} and \cref{asm.sigma}). All results of \cref{appsec:B} apply on any subinterval $[a,b]\subset[0,T]$ in place of $[0,T]$. We will use the control
\begin{align*}
\omega(s,t):=\|X\|_{p,[s,t]}^p+\|\mX\|^{p/2}_{p/2,[s,t]},\quad\text{so that}\quad \nor\bX\nor_{p,[s,t]}\le 2\,\omega(s,t)^{1/p}\ \text{ if } \omega(s,t)\le1.
\end{align*}
Given $\eta\in(0,1]$, we define $t_0:=0$ and $t_{k+1}:=\inf\{t>t_k:\omega(t_k,t)\ge \eta^p/2^p\}\wedge T$. By super-additivity, at most $K:=1+2^p\eta^{-p}\omega(0,T)$ of these points are smaller than $T$, and $\nor\bX\nor_{p,[t_k,t_{k+1}]}\le\eta$ for each $k$. Moreover, for any $(Y,Y')$ and $0\le s\le u\le t\le T$,
\begin{align}\label{n0927:06}
R^Y_{s,t}=R^Y_{s,u}+R^Y_{u,t}+Y'_{s,u}X_{u,t}.
\end{align}
Hence, if $(Y,Y')\in\cV^{q,r}_X([t_k,t_{k+1}])$ for every $k$, then $(Y,Y')\in \cV^{q,r}_X$ and
\begin{align}\label{n0927:09}
\begin{split}
&\|Y'\|_{q,[0,T]}+\|R^Y\|_{r,[0,T]}\\
&\quad\le C(K,p)\Big(\sum_k\big(\|Y'\|_{q,[t_k,t_{k+1}]}+\|R^Y\|_{r,[t_k,t_{k+1}]}\big)+\|Y'\|_\infty\|X\|_{p,[0,T]}\Big),
\end{split}
\end{align}
because a partition interval of $[0,T]$ contains at most $K$ of the points $t_k$, and \eqref{n0927:06} can be used at each of them.

We first need a lemma on the shift of coefficients by a path. For $Z\in C^{r-var}$, we set $\sigma^Z(t,x):=\sigma(t,x+Z_t)$.
\begin{lemma}\label{lem:sigmaZ}
Let $k\in\{0,1\}$, $\sigma$ satisfy {\bf (S$_\sigma^k$)} and $Z\in C^{r-var}$. Then $\sigma^Z$ satisfies {\bf (S$_\sigma^k$)} and
\begin{align*}
\Lambda^k_{q,r}(\sigma^Z)\le 5\Lambda^k_{q,r}(\sigma)(1+\|Z\|_{r,[0,T]}).
\end{align*}
If moreover $Z_0=0$ and $\sigma$ satisfies \Ssigmai{}, then $\Lambda^0_{q,r}(\sigma-\sigma^Z)\le 10\Lambda^1_{q,r}(\sigma)\|Z\|_{r,[0,T]}$.
\end{lemma}
\begin{proof}
Let $\gamma:=\frac{p+q}{q}\vee\frac pr-1=\frac pq\vee(\frac pr-1)\in(0,1]$. For any function $g$ on $\Rd$, $h\in\Rd$ and $j\ge 0$, one easily checks (considering $|h|\le1$ and $|h|>1$ separately)
\begin{align*}
\|g(\cdot+h)-g\|_{C^{j}}\le 2\|g\|_{C^{j+\gamma}}|h|^\gamma\quad \text{and}\quad \|g(\cdot+h)-g\|_{C^{j}}\le 2\|g\|_{C^{j+1}}|h|.
\end{align*}
Since $\sigma^Z(t)-\sigma^Z(s)=[\sigma(t)-\sigma(s)](\cdot+Z_t)+[\sigma(s,\cdot+Z_t)-\sigma(s,\cdot+Z_s)]$ and the norms are invariant under translations, we get
\begin{align*}
\|\sigma^Z(t)-\sigma^Z(s)\|_{C^{1+k}}&\le \|\sigma(t)-\sigma(s)\|_{C^{1+k}}+2\|\sigma\|_{L^\infty C^{1+k+\gamma}}|Z_{s,t}|^\gamma,\\
\|\sigma^Z(t)-\sigma^Z(s)\|_{C^{k}}&\le \|\sigma(t)-\sigma(s)\|_{C^{k}}+2\|\sigma\|_{L^\infty C^{1+k}}|Z_{s,t}|.
\end{align*}
Since $1+k+\gamma$ is the exponent of the last term in $\Lambda^k_{q,r}$ and $\gamma q\ge p\ge r$, the $q$-variation of $|Z_{s,t}|^\gamma$ is bounded by $\|Z\|^\gamma_{r,[0,T]}\le 1+\|Z\|_{r,[0,T]}$, and the first claim follows. For the second claim, we write
\begin{align*}
&(\sigma-\sigma^Z)(t)-(\sigma-\sigma^Z)(s)\\
&\quad=\big[(\sigma(t)-\sigma(s))-(\sigma(t)-\sigma(s))(\cdot+Z_t)\big]-\big[\sigma(s,\cdot+Z_t)-\sigma(s,\cdot+Z_s)\big],
\end{align*}
so that
\begin{align*}
\|(\sigma-\sigma^Z)(t)-(\sigma-\sigma^Z)(s)\|_{C^{1}}\le 2\|\sigma(t)-\sigma(s)\|_{C^{2}}|Z_t|+2\|\sigma\|_{L^\infty C^2}|Z_{s,t}|,
\end{align*}
and similarly with $C^0$, $C^1$ in place of $C^1$, $C^2$. Moreover, $\|\sigma(t)-\sigma^Z(t)\|_{C^{1+\gamma}}\le 2\|\sigma\|_{L^\infty C^{2+\gamma}}|Z_t|$. Since $\sup_t|Z_t|\le\|Z\|_{r,[0,T]}$ and $r\le q$, the second claim follows.
\end{proof}

\begin{proof}[Proof of \cref{sec2:def}]
(i) Let $Y$ be a solution with $|J_{s,t}|\le w(s,t)^{1/\varepsilon}$. Then $J$ has finite $1$-variation, and
\begin{align*}
Y_{s,t}=\sigma(s,Y_s)X_{s,t}+\nabla\sigma(s,Y_s)\sigma(s,Y_s)\mX_{s,t}+Z_{s,t}+J_{s,t}.
\end{align*}
Since $\sigma$ and $\nabla\sigma$ are bounded and $1\le r\le p$, $Y$ has finite $p$-variation. The path $\sigma(Y)$ has finite $p$-variation (see the proof of \cref{lem:030800}), hence finite $q$-variation. Moreover, $R^Y_{s,t}=\nabla\sigma(s,Y_s)\sigma(s,Y_s)\mX_{s,t}+Z_{s,t}+J_{s,t}$ has finite $r$-variation because $r\ge p/2$ and $r\ge1$. Thus $(Y,\sigma(Y))\in\cV^{q,r}_X$. By \cref{lem:030800}, $(\sigma(Y),\sigma(Y)')\in\cV^{q,r}_X$ with $\sigma(Y)'_s=\nabla\sigma(s,Y_s)\sigma(s,Y_s)$, so that $I_t:=\int_0^t\sigma(r,Y_r)\dif\bX_r$ is well defined. By \eqref{0125:01}, the additive function $D_{s,t}:=Y_{s,t}-Z_{s,t}-(I_t-I_s)$ satisfies
\begin{align*}
|D_{s,t}|\le w(s,t)^{1/\varepsilon}+C\left(\|R^{\sigma(Y)}\|_{r,[s,t]}\|X\|_{p,[s,t]}+\|\sigma(Y)'\|_{q,[s,t]}\|\mX\|_{p/2,[s,t]}\right)\le \tilde w(s,t)^{\theta}
\end{align*}
for some control $\tilde w$ and $\theta:=\min\{\frac1\varepsilon,\frac1r+\frac1p,\frac1q+\frac2p\}>1$. Hence for any partition $\cP$ of $[0,t]$,
$|D_{0,t}|\le \sum_{[u,v]\in\cP}|D_{u,v}|\le \max_{[u,v]\in\cP}\tilde w(u,v)^{\theta-1}\,\tilde w(0,t)\to0$ as $|\cP|\to0$, which gives \eqref{App:RDE}.

(ii) If $(Y,\sigma(Y))\in\cV^{q,r}_X$ and \eqref{App:RDE} holds, then $J_{s,t}=\int_s^t\sigma(r,Y_r)\dif\bX_r-\sigma(s,Y_s)X_{s,t}-\sigma(Y)'_s\mX_{s,t}$, and by \eqref{0125:01} $|J_{s,t}|\le \tilde w(s,t)^\theta$ with $\tilde w,\theta$ as above. Thus $Y$ is a solution with $\varepsilon=1/\theta$.
\end{proof}

\begin{proof}[Proof of \cref{App:thm:00}]
{\it Step 0 (reduction to $Z=0$).} For a continuous path $Y$, let $\bar Y:=Y-Z$. Since $\nabla\sigma^Z(s,x)=\nabla\sigma(s,x+Z_s)$, the map $J$ of \cref{def.daviesoln} for $Y$ and RDE$(\sigma,Z)$ coincides with the map $J$ for $\bar Y$ and RDE$(\sigma^Z,0)$. Hence $Y$ solves RDE$(\sigma,Z)$ if and only if $\bar Y$ solves RDE$(\sigma^Z,0)$. Moreover, $\sigma(Y)=\sigma^Z(\bar Y)$ and $R^Y_{s,t}=R^{\bar Y}_{s,t}+Z_{s,t}$. By \cref{lem:sigmaZ}, $\sigma^Z$ satisfies \Ssigmai{} with $\Lambda^1_{q,r}(\sigma^Z)\le 5\Lambda^1_{q,r}(\sigma)(1+\|Z\|_{r,[0,T]})$. Therefore, it suffices to consider $Z\equiv0$.

{\it Step 1 (local solutions).} Let $\Lambda:=\Lambda^1_{q,r}(\sigma)$, $\lambda_0:=4\Lambda+1$ and $N_*:=2\Lambda+2$. Let $C_*\ge1$ be larger than the constants $C(\Theta_0)$ and $C(\Theta_1)$ in \cref{lem:A5} evaluated at $N_Y=N_{\tilde Y}=N_*$ and $\nor\bX\nor_{p}=1$ (these constants are nondecreasing in $\nor\bX\nor_p$, and on $[a,b]$ below $\nor\bX\nor_{p,[a,b]}\le\eta\le1$). Choose $\eta=\eta(p,q,r,\Lambda,\nor\bX\nor_{p,[0,T]})\in(0,1]$ such that
\begin{align}\label{n0927:07}
2\Lambda(2\Lambda+1)\eta\le \tfrac12,\qquad \lambda_0C_*\eta\le 1,\qquad (1+4C_*)C_*\eta\le\tfrac14.
\end{align}
Let $[a,b]\subset[0,T]$ with $\nor\bX\nor_{p,[a,b]}\le\eta$ and $\zeta\in\Rd$. Consider the set
\begin{align*}
\cB:=\Big\{(Y,Y')\in \cV_X^{q,r}([a,b]):~&Y_a=\zeta,~Y_a'=\sigma(a,\zeta),\\
&\|Y'-\sigma(\cdot,\zeta)\|_{q,[a,b]}\le 1,~\|R^Y\|_{r,[a,b]}\le \tfrac{1}{\lambda_0} \Big\},
\end{align*}
on which $M_Y\le 2\Lambda+1$ and $N_Y\le N_*$. Consider the map
\begin{align*}
\cM(Y,Y'):=\left(\zeta+\int_a^\cdot \sigma(r,Y_r)\dif \bX_r,\sigma(\cdot,Y)\right).
\end{align*}
By \cref{thm:0125:01,lem:030800}, $\cM(Y,Y')\in\cV^{q,r}_X([a,b])$, and by \eqref{n0927:00} and \eqref{n0927:07}, for $(Y,Y')\in\cB$,
\begin{align*}
\|\sigma(Y)-\sigma(\cdot,\zeta)\|_{q,[a,b]}\le 2\Lambda\big((2\Lambda+1)\eta+\lambda_0^{-1}\big)\le1,\qquad \|R^{\cM(Y)}\|_{r,[a,b]}\le C_*\eta\le\lambda_0^{-1}.
\end{align*}
Thus $\cM(\cB)\subset\cB$. Moreover, $\cB$ is non-empty: for $A:=(\zeta+\sigma(a,\zeta)X_{a,\cdot},\sigma(a,\zeta))\in\cV^{q,r}_X([a,b])$, we have $R^A\equiv0$ and $M_A=|\sigma(a,\zeta)|\le\Lambda$, so the computation in the proof of \cref{lem:A5} and \eqref{n0927:07} give $\|\sigma(A)-\sigma(\cdot,\zeta)\|_{q,[a,b]}\le 2\Lambda^2\eta\le1$ and $\|R^{\cM(A)}\|_{r,[a,b]}\le C_*\eta\le\lambda_0^{-1}$; hence $\cM(A)\in\cB$. Given $(Y,Y'),(\tilde{Y},\tilde{Y}')\in \cB$, we have $\Delta Y_a=\Delta Y'_a=0$, and \eqref{n0927:01} gives
\begin{align*}
\|\Delta\cM'\|_{q,[a,b]}&\le C_*\left(\eta\|\Delta Y'\|_{q,[a,b]}+\|\Delta R^Y\|_{r,[a,b]}\right),\\
\|\Delta R^{\cM}\|_{r,[a,b]}&\le C_*\eta\left(\|\Delta Y'\|_{q,[a,b]}+\|\Delta R^Y\|_{r,[a,b]}\right),
\end{align*}
where $\Delta\cM'$ and $\Delta R^{\cM}$ denote the differences of the derivatives and remainders of $\cM(Y,Y')$ and $\cM(\tilde Y,\tilde Y')$.
In the first estimate, $\|\Delta R^Y\|_r$ appears without a small factor. To absorb it, we use on $\cB$ the metric
$$
d_\delta\big((Y,Y'),(\tilde Y,\tilde Y')\big):=\|\Delta Y'\|_{q,[a,b]}+\delta\|\Delta R^Y\|_{r,[a,b]},\qquad \delta:=4C_*,
$$
which is equivalent to the norm $\|\cdot\|_{q,r}$ on $\cB$; in particular $(\cB,d_\delta)$ is complete. By \eqref{n0927:07},
\begin{align*}
d_\delta(\cM(Y,Y'),\cM(\tilde Y,\tilde Y'))&\le (1+\delta)C_*\eta\|\Delta Y'\|_{q,[a,b]}+\left(\tfrac14+C_*\eta\right)\delta\|\Delta R^Y\|_{r,[a,b]}\\
&\le \tfrac12 d_\delta\big((Y,Y'),(\tilde Y,\tilde Y')\big).
\end{align*}
Hence $\cM$ has a unique fixed point in $\cB$, which by \cref{sec2:def}~(ii) is a solution on $[a,b]$ starting from $\zeta$.

{\it Step 2 (global existence and bound).} Let $t_k$ be the points defined above with $\eta$ from Step 1. Solving successively on $[t_k,t_{k+1}]$ with $\zeta=Y_{t_k}$, we obtain a path $Y$ on $[0,T]$. On each $[t_k,t_{k+1}]$, $\|\sigma(Y)\|_{q,[t_k,t_{k+1}]}\le \|\sigma\|_{C^{q-var}L^\infty}+1\le\Lambda+1$ and $\|R^Y\|_{r,[t_k,t_{k+1}]}\le 1$. By \eqref{n0927:09}, $(Y,\sigma(Y))\in\cV^{q,r}_X$ and $\|\sigma(Y)\|_{q,[0,T]}+\|R^Y\|_{r,[0,T]}\le C$ with $C$ depending only on $K,p,\Lambda,\nor\bX\nor_{p,[0,T]}$. By the additivity of rough integrals, \eqref{App:RDE} holds on $[0,T]$, and by \cref{sec2:def}~(ii), $Y$ is a solution. Together with Step 0, this gives the existence and \eqref{0124:07}.

{\it Step 3 (uniqueness).} Let $\tilde Y$ be another solution with $\tilde Y_0=y$, and let $\tau:=\sup\{t\in[0,T]:Y=\tilde Y\text{ on }[0,t]\}$. Suppose $\tau<T$. By \cref{sec2:def}~(i), $(Y,\sigma(Y))$ and $(\tilde Y,\sigma(\tilde Y))$ belong to $\cV^{q,r}_X$. Since the variations over $[\tau,b]$ of the continuous paths $\sigma(Y)-\sigma(\cdot,Y_\tau)$, $\sigma(\tilde Y)-\sigma(\cdot,Y_\tau)$ and, by the representation of $R^Y$ in the proof of \cref{sec2:def}~(i), those of $R^Y$, $R^{\tilde Y}$ and $\bX$ tend to zero as $b\downarrow\tau$, we can choose $b\in(\tau,T]$ such that $\nor\bX\nor_{p,[\tau,b]}\le\eta$ and both restrictions to $[\tau,b]$ belong to the set $\cB$ of Step 1 with $[a,b]=[\tau,b]$ and $\zeta=Y_\tau=\tilde Y_\tau$. Both are fixed points of $\cM$, so $Y=\tilde Y$ on $[\tau,b]$, which contradicts the definition of $\tau$. Hence $\tau=T$.
\end{proof}

\begin{proof}[Proof of \cref{App:thm:sta}]
By \eqref{0124:07}, $N_Y$ and $N_{\tilde Y}$ are bounded by a constant $N_*$ depending only on $T,p,q,r,\Lambda^1_{q,r}(\sigma),\Lambda^1_{q,r}(\tilde\sigma),\nor\bX\nor_{p,[0,T]}$; in this proof, $C$ denotes a constant depending only on these quantities. Here $Y'=\sigma(Y)$ and $\tilde Y'=\tilde\sigma(\tilde Y)$, and we set $\Lambda^0:=\Lambda^0_{q,r}(\sigma-\tilde\sigma)$. Let $[a,b]\subset[0,T]$ with $\nor\bX\nor_{p,[a,b]}\le\eta'$, where $\eta'\in(0,1]$ is chosen below. We decompose
\begin{align*}
Y'-\tilde Y'&=\big[\tilde\sigma(Y)-\tilde\sigma(\tilde Y)\big]+(\sigma-\tilde\sigma)(Y),\\
R^Y-R^{\tilde Y}&=\big[R^{I_\bX(\tilde\sigma(Y))}-R^{I_\bX(\tilde\sigma(\tilde Y))}\big]+R^{I_\bX((\sigma-\tilde\sigma)(Y))}.
\end{align*}
Applying \eqref{n0927:01} with $\tilde\sigma$ in place of $\sigma$ to the first terms, and \cref{thm:0125:01,lem:030800} with $\sigma-\tilde\sigma$ in place of $\sigma$ to the second terms (note $\|(\sigma-\tilde\sigma)(Y)\|_{q,[a,b]}\le\Lambda^0(1+\|Y\|_{p,[a,b]})$), we obtain on $[a,b]$
\begin{align*}
\|\Delta Y'\|_{q,[a,b]}&\le C\big(|\Delta Y_a|+\Delta M_Y\eta'+\|\Delta R^Y\|_{r,[a,b]}+\Lambda^0\big),\\
\|\Delta R^Y\|_{r,[a,b]}&\le C\eta'\big(|\Delta Y_a|+\Delta N_Y\big)+C\Lambda^0,
\end{align*}
where $\Delta M_Y,\Delta N_Y$ are computed on $[a,b]$ and $|\Delta Y'_a|\le \Lambda^0+C|\Delta Y_a|$. Adding the first inequality to $\delta:=4C$ times the second one and choosing $\eta'$ so small that $(1+\delta)C\eta'\le\frac14$, we get
\begin{align*}
\|\Delta Y'\|_{q,[a,b]}+\|\Delta R^Y\|_{r,[a,b]}\le C\big(|\Delta Y_a|+\Lambda^0\big),
\end{align*}
and hence, by \eqref{0124:02}, $|\Delta Y_b|\le |\Delta Y_a|+\|\Delta Y\|_{p,[a,b]}\le C\big(|\Delta Y_a|+\Lambda^0\big)$. Applying this on the intervals $[t_k,t_{k+1}]$ (defined with $\eta'$), whose number is bounded by $K$, and using \eqref{n0927:09} for $\Delta Y$, we complete the proof.
\end{proof}

\begin{proof}[Proof of \cref{cor:0125:00}]
Fix $t\in(0,T]$; we restrict all RDEs to $[0,t]$. As in Step 0 of the proof of \cref{App:thm:00}, $\bar Y:=Y^Z-Z$ is the solution to RDE$(\sigma^Z,0)$ with $\bar Y_0=Y_0$. By \cref{lem:sigmaZ}, $\Lambda^1_{q,r}(\sigma^Z)\le 5\Lambda^1_{q,r}(\sigma)(1+K)$ and $\Lambda^0_{q,r}(\sigma-\sigma^Z)\le 10\Lambda^1_{q,r}(\sigma)\|Z\|_{r,[0,t]}$. Hence \cref{App:thm:sta} (whose constant is nondecreasing in $T$ and $\nor\bX\nor_{p,[0,T]}$, hence uniform in $t\le T$) yields
\begin{align*}
\|\sigma(Y)-\sigma^Z(\bar{Y})\|_{q,[0,t]}+\|R^Y-R^{\bar{Y}}\|_{r,[0,t]}\le C\|Z\|_{r,[0,t]}.
\end{align*}
Since $Z_0=0$, we have $\sigma(0,Y_0)=\sigma^Z(0,\bar Y_0)$, so \eqref{0124:02} gives $|Y_t-\bar{Y}_t|\le\|Y-\bar{Y}\|_{p,[0,t]}\le C(1+\|X\|_{p,[0,T]})\|Z\|_{r,[0,t]}$. Since $Y^Z_t-Y_t=\bar Y_t-Y_t+Z_t$ and $|Z_t|\le\|Z\|_{r,[0,t]}$, the proof is complete.
\end{proof}

\subsection{Proof of Theorem \ref{thm:0311}}\label{sec:A.4}
\begin{proof}[Proof of \cref{thm:0311}]
(i) Let $X$ be continuous and adapted with $(X,\sigma(X))\in\cV^{q,r}_W$ almost surely; only these properties are used below. By \cref{lem:030800} applied to $g$, we have $(g(X),g(X)')\in\cV^{q,r}_W$ with $g(X)'_s=\nabla g(s,X_s)\sigma(s,X_s)$, so the right-hand side of \eqref{0314:00} is well defined and continuous in $t$. Since the left-hand side is continuous in $t$ as well, it suffices to show \eqref{0314:00} almost surely for each fixed $t\in(0,T]$. Let $s_i:=it/n$, $i=0,\dots,n$, and $\pi_n(r):=s_i$ for $r\in[s_i,s_{i+1})$. By \cref{App:lem:A4}, almost surely,
\begin{align*}
\int_0^t g(r,X_r)\dif \bW_r=\lim_{n\to\infty}\sum_{i=0}^{n-1}\left(g(s_i,X_{s_i})W_{s_i,s_{i+1}}+\nabla g(s_i,X_{s_i})\sigma(s_i,X_{s_i})\mW_{s_i,s_{i+1}}\right).
\end{align*}
Since $g$ is bounded and $r\mapsto g(r,X_r)$ is continuous, It\^o's isometry and the dominated convergence theorem give
\begin{align*}
\mE\left|\int_0^t g(r,X_r)\dif W_r-\sum_{i=0}^{n-1}g(s_i,X_{s_i})W_{s_i,s_{i+1}}\right|^2=\mE\int_0^t |g(r,X_r)-g(\pi_n(r),X_{\pi_n(r)})|^2\dif r\to0.
\end{align*}
Let $\xi_i:=\nabla g(s_i,X_{s_i})\sigma(s_i,X_{s_i})$, which is $\sF_{s_i}$-measurable and bounded by $\|\nabla g\|_\infty\|\sigma\|_\infty$. Since $\mW_{s_i,s_{i+1}}$ is independent of $\sF_{s_i}$ and centred, the cross terms vanish and
\begin{align*}
\mE\left|\sum_{i=0}^{n-1}\xi_i\mW_{s_i,s_{i+1}}\right|^2\le\sum_{i=0}^{n-1}\mE|\xi_i|^2\,\mE|\mW_{s_i,s_{i+1}}|^2\le C\|\nabla g\|_\infty^2\|\sigma\|_\infty^2\sum_{i=0}^{n-1}|s_{i+1}-s_i|^2\to0.
\end{align*}
Hence the Riemann sums converge in probability to the It\^o integral and almost surely to the rough integral, which yields \eqref{0314:00}.

(ii) Let $M_t:=\int_0^t\sigma(r,X_r)\dif W_r$ and $f_r:=\sigma(r,X_r)$, which is continuous and adapted. Fix $p_1\in(2,p]$. Since $p<2+\eps<\frac{2(2+\eps)}{\eps}$, we have $\frac1r<\frac1{p_1}+\frac12$. By \eqref{in:con} and the boundedness of $\nabla\sigma$,
\begin{align*}
\|f\|_{p_1,[0,T]}&\le \|\sigma\|_{C^{1-var}L^\infty}+\|\nabla\sigma\|_\infty\|X\|_{p_1,[0,T]}\\
&\le \|\sigma\|_{C^{1-var}L^\infty}+\|\nabla\sigma\|_\infty\Big(\int_0^T|b(r,X_r)|\dif r+\|M\|_{p_1,[0,T]}\Big).
\end{align*}
By L\'epingle's inequality (see \cite[(1.2)]{FZ23}) and the BDG inequality, $\mE\|M\|^k_{p_1,[0,T]}\lesssim \|\sigma\|_\infty^k$ for any $k\ge1$; together with the assumption on $b$, this gives $\mE\|f\|^{1+\delta}_{p_1,[0,T]}<\infty$. In particular, $X$ and $f$ have finite $p_1$-variation, hence finite $p$-variation and finite $q$-variation. Moreover,
\begin{align*}
R^X_{s,t}=X_{s,t}-\sigma(s,X_s)W_{s,t}=\int_s^t b(r,X_r)\dif r+\int_s^t (f_r-f_s)\dif W_r.
\end{align*}
The first term has finite $1$-variation, hence finite $r$-variation. For the second term, we apply \cite[Theorem 1.1, (1.18)]{FZ23} with $g=W$, $q_1=1+\delta$ and any $q_0\in[1,\infty)$: since $\mE(\sup_{t\le T}|W_t|)^{q_0}<\infty$, the $r$-variation of $(s,t)\mapsto\int_s^t (f_r-f_s)\dif W_r$ has a finite moment of order $(1/q_0+1/q_1)^{-1}$, so it is finite almost surely. Therefore $(X,\sigma(X))\in\cV^{q,r}_W$ almost surely.
By \cref{lem:HS}, $g=\sigma$ satisfies \Ssigmao{}, and (i) yields $M_t=\int_0^t\sigma(r,X_r)\dif\bW_r$ for all $t$, almost surely. Hence \eqref{0311:02} holds almost surely, and by the remark after \cref{def:413}, $X(\omega)$ is a path-by-path solution for almost every $\omega$.

(iii) Let $X$ be continuous, adapted, and a path-by-path solution for almost every $\omega$. Since $b$ is Borel measurable, $(b(t,X_t))_{t\ge0}$ is progressively measurable, and $\int_0^T|b(t,X_t)|\dif t<\infty$ almost surely. By the remark after \cref{def:413}, $(X,\sigma(X))\in\cV^{q,r}_W$ almost surely, and (i) with $g=\sigma$ shows that the rough integral in \eqref{0311:02} coincides almost surely with the It\^o integral $\int_0^t\sigma(r,X_r)\dif W_r$. Hence $X$ is a weak solution to \eqref{0311:00}.
\end{proof}

\section{Anisotropic fractional Sobolev spaces and embeddings}\label{sec:D}
Fix $n\in\mN$. Let $\ell=(\ell_1,...,\ell_n)\in\mN^n$ and define $N:=\ell_1+\cdots+\ell_n$. Also let $a=(a_1,..,a_n)\in [1,\infty)^n$. We introduce the following distance on $\mR^N$:
\begin{align*}
|x-y|_a:=\sum_{i=1}^n|x_i-y_i|^{1/a_i},\quad x=(x_1,\dots,x_n),\ y=(y_1,\dots,y_n),\ x_i,y_i\in\mR^{\ell_i},
\end{align*}
where $|\cdot|$ denotes the Euclidean norm in $\mR^{\ell_i}$, and we call $D:=\sum_{i=1}^n a_i\ell_i$ the homogeneous dimension of $|\cdot|_a$. A box is a set of the form $O=\bigtimes_{i=1}^n\{x_i\in\mR^{\ell_i}:|x_i-c_i|_\infty<b_i\}$ with $c_i\in\mR^{\ell_i}$ and $b_i>0$, where $|\cdot|_\infty$ is the maximum norm; that is, a product of cubes. Replacing the Euclidean norms in $|\cdot|_a$ by equivalent norms only changes the constants below.

For any $\alpha\in(0,1)$, $p\in[1,\infty)$ and open set $O\subset \mR^N$, we define
\begin{align*}
[f]_{W^{\alpha,p}_a(O)}:=\left(\int_{O^2}\frac{|f(x)-f(y)|^p}{|x-y|_a^{\alpha p+D}}\dif x\dif y\right)^{1/p}.
\end{align*}
The anisotropic fractional Sobolev space (also called Sobolev--Slobodeckij space) is defined by
\begin{align*}
W^{\alpha,p}_a(O):=\{f\in L^p(O)~;~\|f\|_{W^{\alpha,p}_a(O)}:=\|f\|_{L^p(O)}+[f]_{W^{\alpha,p}_a(O)}<\infty\},
\end{align*}
where $\|f\|_{L^p(O)}:=(\int_{O}|f(x)|^p\dif x)^{1/p}$. For $\beta\in(0,1)$, we also set
$$
\|f\|_{C^{\beta}_a(O)}:=\|f\|_{L^\infty (O)}+\sup_{x,y\in O,x\ne y}\frac{|f(x)-f(y)|}{|x-y|_a^\beta}.
$$
The following embedding theorem is probably well known to experts. For the readers' convenience, we give a proof at the end of this section.
\begin{theorem}\label{thm:0317}
Let $\alpha\in(0,1)$ and $p\in[1,\infty)$ with $\alpha>D/p$, and let $O$ be a box. Then there is a constant $C=C(n,\ell,a,\alpha,p,O)>0$ such that every $f\in W^{\alpha,p}_a(O)$ has a continuous version, still denoted by $f$, with
\begin{align*}
    \|f\|_{C^{\alpha-D/p}_a(O)}\le C \|f\|_{W^{\alpha,p}_a(O)}.
\end{align*}
\end{theorem}

\begin{proof}[Proof of \cref{lem:GRR}]
Let $O^\circ$ be the interior of $O$, let $F:O\to\mR$ be continuous, and set $\bar F:=|O|^{-1}\int_{O}F(z)\dif z$. By Jensen's inequality,
\begin{align*}
\|F-\bar F\|^m_{L^m(O)}\le |O|^{-1}\int_{O^2}|F(z_1)-F(z_2)|^m\dif z_1\dif z_2\le |O|^{-1}\,{\rm diam}_a(O)^{\theta m+D}\,[F]^m_{W^{\theta,m}_a(O^\circ)},
\end{align*}
where ${\rm diam}_a(O)$ is the diameter of $O$ with respect to $|\cdot|_a$. Applying \cref{thm:0317} to $F-\bar F$ on $O^\circ$ (note that $F-\bar F$ is continuous, hence equal to its continuous version), we obtain for $z,z'\in O^\circ$
\begin{align*}
|F(z)-F(z')|\le C|z-z'|_a^{\theta-\frac Dm}\|F-\bar F\|_{W^{\theta,m}_a(O^\circ)}\le C|z-z'|_a^{\theta-\frac Dm}[F]_{W^{\theta,m}_a(O^\circ)}.
\end{align*}
By continuity, this extends to $z,z'\in O$.
\end{proof}

\begin{proof}[Proof of \cref{thm.Kolmogorov}]
Let $\eta$ be as in the statement, and let $\alpha:=\eta+D/m\in(D/m,1)$. By Fubini's theorem,
    \begin{align*}
        \mE\|X(\cdot)\|_{W^{\alpha,m}_a(O^\circ)}^m&\lesssim\int_{O^2}\frac{\mE|X(z)-X(z')|^m}{|z-z'|_a^{\eta m+2D}}\dif z\dif z'+\int_{O}\mE|X(z)|^m\dif z\\
        &\lesssim \int_{O^2}|z-z'|_a^{\kappa-\eta m-2D}\dif z\dif z'+1<\infty,
    \end{align*}
since $\kappa-\eta m-2D>-D$. Hence there is a set $\Omega_0$ of full probability such that $X(\cdot)(\omega)\in W^{\alpha,m}_a(O^\circ)$ for all $\omega\in \Omega_0$. For $\omega\in\Omega_0$, let $\tilde{X}(\cdot)(\omega)$ be the continuous version given by \cref{thm:0317}, extended to $O$ by continuity; for $\omega\notin\Omega_0$, let $\tilde{X}(\cdot)(\omega):=0$. By \cref{thm:0317}, the expectation in the statement is bounded by $C(\mE\|X\|^m_{W^{\alpha,m}_a(O^\circ)})^{1/m}<\infty$. Moreover, by Fubini's theorem,
    \begin{align*}
        \int_O\mE|X(z)-\tilde{X}(z)|^m \dif z=0,
    \end{align*}
so there is a set $O'\subset O$ of full Lebesgue measure such that for all $z\in O'$, $X(z)=\tilde{X}(z)$ almost surely. For any $z\in O$, there is a sequence $\{z_n\}\subset O'$ with $z_n\to z$. Since $X(z_n)\to X(z)$ in $L^m(\Omega)$ and $\tilde{X}(z_n)(\omega)\to \tilde{X}(z)(\omega)$ for every $\omega$, we have $X(z)=\tilde{X}(z)$ almost surely. This completes the proof.
\end{proof}
We first prove the following extension theorem.
\begin{theorem}\label{thm:extend}
Let $y\in\mR^N$, $p\in[1,\infty)$, $\alpha\in(0,1)$, $(b_i)_{i=1}^n\in \mR_+^{n}$ and 
$
O:=\bigtimes_{i=1}^n\{|x_i-y_i|_\infty<b_i\}.
$
 Then $W^{\alpha,p}_a(O)$ is continuously embedded in $W^{\alpha,p}_a(\mR^N)$. More precisely, there is a constant $C=C(N,a,\ell,p,\alpha,O)>0$ such that for all $f\in W^{\alpha,p}_a(O)$, there is a $\tilde{f}\in W^{\alpha,p}_a(\mR^N)$ such that $\tilde{f}|_O=f$ and
\begin{align}\label{0404:02}
\|\tilde{f}\|_{W^{\alpha,p}_a(\mR^N)}\le C\|f\|_{W^{\alpha,p}_a(O)}.
\end{align}
\end{theorem}
\begin{proof}
Without loss of generality, replacing $f(x)$ by $f(x+y)$, we can let $y=0$. Moreover, we can assume $n=N$ and $\ell_1=\ell_2=\cdots=\ell_n=1$, otherwise, we can let $\tilde{\ell}=(1,...,1)\in \mN^N$ and $\tilde{a}:=(a_1,..,a_1,a_2,...,a_2,...,a_n,...,a_n)\in [1,\infty)^N$.
We divide the proof into two steps. \\
{\bf (Step 1):} In this step, we show that for any $m\in \mN$ and
$
Q_m:=\bigtimes_{i=1}^n\{|x_i|<(1+m) b_i\}
$, there is a function $\tilde{f}_m$ such that $\tilde{f}_m|_{O}=f$ and
\begin{align*}
\|\tilde{f}_m\|_{W^{\alpha,p}_a(Q_m)}\lesssim_{m} \|f\|_{W^{\alpha,p}_a(O)}.
\end{align*}
We first define $O_1:=\{-b_1<x_1<2b_1\}\times \left(\bigtimes_{i=2}^n\{|x_i|<b_i\}\right)$ and 
\begin{align*}
f_1(x):=\begin{cases}
f(x),&\quad x\in \bar{O},\\
f(T_1x),&\quad x\in \bar{O}_1\backslash \bar{O},
\end{cases}
\end{align*}
where $T_1x:=(2b_1-x_1,x_2,...,x_n)$. 
Then noting that $|x_1-y_1|\ge|2b_1-x_1-y_1|$ for all $|y_1|<b_1<x_1<2b_1$, by a change of variable, we have
\begin{align*}
[f_1]_{W^{\alpha,p}_a(O_1)}^p&=\int_{O^2}\frac{|f(x)-f(y)|^p}{|x-y|_a^{\alpha p+D}}\dif x\dif y+2\int_{(O_1\backslash \bar{O})\times O}\frac{|f(T_1x)-f_1(y)|^p}{|x-y|_a^{\alpha p+D}}\dif x\dif y\\
&+\int_{(O_1\backslash \bar{O})^2}\frac{|f(T_1x)-f(T_1y)|^p}{|T_1x-T_1y|_a^{\alpha p+D}}\dif x\dif y\le 4\int_{O^2}\frac{|f(x)-f(y)|^p}{|x-y|_a^{\alpha p+D}}\dif x\dif y,
\end{align*}
which implies that 
\begin{align*}
\|f_1\|_{W^{\alpha,p}_a(O_1)}\le 4^{1/p}\|f\|_{W^{\alpha,p}_a(O)}.
\end{align*}
Similarly, for $\tilde{O}_1:=\{|x_1|<2b_1\}\times \left(\bigtimes_{i=2}^n\{|x_i|<b_i\}\right)$ we can define 
\begin{align*}
\tilde{f}_1(x):=\begin{cases}
f_1(x),&\quad x\in \bar{O_1},\\
f_1(\tilde{T}_1x),&\quad x\in \overline{\tilde{O}_1}\backslash \bar{O}_1,
\end{cases}
\end{align*}
where $\tilde{T}_1x:=(-2b_1-x_1,x_2,...,x_n)$, and have $\|\tilde{f}_1\|_{W^{\alpha,p}_a(\tilde{O}_1)}\le 4^{1/p}\|f_1\|_{W^{\alpha,p}_a(O_1)}$. Eventually, by induction, for $\tilde{O}_N:=\bigtimes_{i=1}^n\{|x_i|<2b_i\}$ we can find $\tilde{f}_N$ such that $\tilde{f}_N|_O=f$ and
\begin{align*}
\|\tilde{f}_N\|_{W^{\alpha,p}_a(\tilde{O}_N)}\le 4^{2N/p} \|f\|_{W^{\alpha,p}_a(O)}.
\end{align*}
By induction again, for any $m\in\mN$, we can find $\tilde{f}_m$ such that $\tilde{f}_m|_O=f$ and
\begin{align*}
\|\tilde{f}_m\|_{W^{\alpha,p}_a(Q_m)}\le 4^{2Nm/p}\|f\|_{W^{\alpha,p}_a(O)},
\end{align*}
which is what we need.\\
{\bf (Step 2):}
Let $\phi\in C^\infty_c(Q_2)$ such that $\phi(x)=1$ for all $x\in O$. Then we define
\begin{align*}
\tilde{f}:=\tilde{f}_2\phi,\quad \text{where $\tilde{f}_2$ is defined in step 1,}
\end{align*}
and complete the proof.
\end{proof}
We now prove Theorem \ref{thm:0317}:
\begin{proof}[Proof of Theorem \ref{thm:0317}]
For any $f\in W^{\alpha,p}_{a}(O)$, we let $\tilde{f}\in W^{\alpha,p}_{a}(\mR^N)$ be the one defined in Theorem \ref{thm:extend}. We introduce the anisotropic Besov space $B^{\alpha,a}_{p,q}$ defined in \cite[Definition 2.3]{HZZZ21}. Based on \cite[Theorem 2.6]{HZZZ21} and \cite[Theorem 2.7]{HZZZ21}, we have
\begin{align*}
\|\tilde{f}\|_{C^{\alpha-\frac{D}{p}}_{a}(\mR^N)}\asymp\|\tilde{f}\|_{B^{\alpha-\frac{D}{p},a}_{\infty,\infty}(\mR^N)}\lesssim\|\tilde{f}\|_{B^{\alpha-\frac{D}{p},a}_{\infty,p}(\mR^N)}\lesssim \|\tilde{f}\|_{B^{\alpha,a}_{p,p}(\mR^N)}\lesssim\|\tilde{f}\|_{W^{\alpha,p}_{a}(\mR^N)}.
\end{align*}
Then by the definition of H\"older space and \eqref{0404:02}, we have
\begin{align*}
\|f\|_{C^{\alpha-\frac{D}{p}}_{a}(O)}\le\|\tilde{f}\|_{C^{\alpha-\frac{D}{p}}_{a}(\mR^N)}\lesssim \|\tilde{f}\|_{W^{\alpha,p}_{a}(\mR^N)}\lesssim \|f\|_{W^{\alpha,p}_{a}(O)}
\end{align*}
and complete the proof.
\end{proof}

\section{Interpolation inequality}\label{App:C}

\begin{lemma}\label{lem:C1}
    Let $\alpha_0,\alpha_1\ge0$, $\theta\in[0,1]$ and $\alpha_\theta:=\theta \alpha_1+(1-\theta)\alpha_0$. There is a constant $C>0$ such that for all $f\in C^{\alpha_0\vee \alpha_1}$,
\begin{align*}
\|f\|_{C^{\alpha_\theta}}\lesssim\|f\|_{C^{\alpha_1}}^\theta \|f\|_{C^{\alpha_0}}^{1-\theta}
\end{align*}
\end{lemma}
\begin{proof}
Since it is obvious for $\theta=0,1$, we assume $\theta\in(0,1)$ in the following. When $\alpha\notin\mN_0$,
\cite[Definition 2.68]{BCD11} yields that the equivalence between H\"older space and Besov space $C^\alpha = B^{\alpha}_{\infty,\infty}$. When $\alpha\in\mN_0$, based on \cite[Definition 2.68]{BCD11} and \cite[Lemma 2.4-(i)]{HZ23}, we have for any $f\in B^{\alpha}_{\infty,1}$,
\begin{align}\label{250219:00}
\|f\|_{B^{\alpha}_{\infty,\infty}}\lesssim\|f\|_{C^{\alpha}}\asymp\sum_{i=0}^\alpha\|\nabla^i f\|_{L^\infty }\lesssim \sum_{i=0}^\alpha\|\nabla^i f\|_{B^{0}_{\infty,1}}\lesssim \| f\|_{B^{\alpha}_{\infty,1}}.
\end{align}
Therefore, \eqref{250219:00} holds for all $\alpha\ge0$, which by \cite[Theorem 2.80]{BCD11} implies that
\begin{align*}
\|f\|_{C^{\alpha_\theta}}\lesssim \|f\|_{B^{\alpha_\theta}_{\infty,1}}\lesssim\|f\|_{B^{\alpha_1}_{\infty,\infty}}^\theta \|f\|_{B^{\alpha_0}_{\infty,\infty}}^{1-\theta}\lesssim\|f\|_{C^{\alpha_1}}^\theta \|f\|_{C^{\alpha_0}}^{1-\theta}.
\end{align*}
This completes the proof. 
\end{proof}
\begin{lemma}\label{lem:C2}
    For any $\theta\in(0,1)$, $p_0,p_1\in(0,\infty]$ and $\alpha_0,\alpha_1\in[0,\infty)$, 
\begin{align}\label{0305:01}
\|f\|_{C^{p_\theta-var}C^{\alpha_\theta} }\lesssim \|f\|_{C^{p_1-var}C^{\alpha_1}}^\theta \|f\|_{C^{p_0-var}C^{\alpha_0}}^{1-\theta},
\end{align}
where $1/p_\theta=\theta/p_1+(1-\theta)/p_0$ and $\alpha_\theta:=\theta\alpha_1+(1-\theta)\alpha_0$.
\end{lemma} 
\begin{proof}
    We only prove the case $p_0,p_1<\infty$, the proof for $p_0=\infty$ or $p_1=\infty$ is similar. 
    It follows from Lemma \ref{lem:C1} that
        
    \begin{align*}
        \|f\|_{C^{p_\theta-var}C^{\alpha_\theta}}&=\sup_{\cP}\left(\sum_{[s,t]\in\cP}\|f(t)-f(s)\|_{C^{\alpha_\theta}}^{p_\theta}\right)^{1/p_\theta}\\
        &\lesssim \sup_{\cP}\left(\sum_{[s,t]\in\cP}\|f(t)-f(s)\|_{C^{\alpha_1}}^{\theta p_\theta}\|f(t)-f(s)\|_{C^{\alpha_0}}^{(1-\theta)p_\theta}\right)^{1/p_\theta}.
    \end{align*}
    We note that $ \theta p_\theta\le p_1$. Then by H\"older's inequality, we have
    \begin{align*}
        &\sum_{[s,t]\in\cP}\|f(t)-f(s)\|_{C^{\alpha_1}}^{\theta p_\theta}\|f(t)-f(s)\|_{C^{\alpha_0}}^{(1-\theta)p_\theta}\\
        &\quad \le \left(\sum_{[s,t]\in\cP}\|f(t)-f(s)\|_{C^{\alpha_1}}^{p_1}\right)^{\theta p_\theta/p_1}\left(\sum_{[s,t]\in\cP}\|f(t)-f(s)\|_{C^{\alpha_0}}^{(1-\theta)p_\theta\frac{p_1}{p_1-\theta p_\theta}}\right)^{(p_1-\theta p_\theta)/p_1}.
    \end{align*}
    Noting that $(1-\theta)p_\theta p_1/(p_1-\theta p_\theta)=p_0$, we get \eqref{0305:01} and finish the proof.
    \end{proof}

\end{appendix}

\bibliographystyle{alpha}
\bibliography{biblio}

\end{document}